\documentclass[11pt,a4paper]{article}
\usepackage{algorithm}
\usepackage{algpseudocode}
\usepackage[utf8]{inputenc}
\usepackage[T1]{fontenc}
\usepackage[english]{babel}
\usepackage{xcolor}
\usepackage{graphicx}
\usepackage{subcaption}
\usepackage{hyperref}
\usepackage{amsmath, amstext, amsopn, amssymb, amsthm}
\usepackage{thmtools}
\usepackage{mathtools}
\usepackage{siunitx}
\usepackage{etoolbox}
\usepackage{bm}
\usepackage{booktabs}
\usepackage{makebox}
\usepackage{stmaryrd}
\usepackage{setspace}
\usepackage[shortlabels]{enumitem}
\setlist[enumerate,1]{label= (\alph*)}
\usepackage{bbm}
\usepackage[authoryear]{natbib}

\graphicspath{{Figures/}}

\setcitestyle{semicolon}

\setcitestyle{aysep={}}

\setcitestyle{yysep={,}}

\addto\extrasenglish{

}
\hypersetup{
  colorlinks   = true,
  urlcolor     = blue,
  linkcolor    = blue,
  citecolor    = blue
}

\DeclareMathOperator*{\argmin}{arg\,min}

\DeclareMathOperator{\sgn}{sgn}

\DeclarePairedDelimiterX{\inner}[2]{\langle}{\rangle}{{#1},{#2}}

\newcommand{\gammaK}[1]{\Gamma_{#1}^{\mathbb{K}}}
\newcommand{\gammaG}[1]{\Gamma_{#1}^{\mathbb{G}}}

\newcommand{\de}{\mathrm{d}}
\newcommand{\diff}{\,\mathrm{d}}
\newcommand{\invmap}{\Phi_{\mathrm{inv}}}

\newcommand{\R}{\mathbb{R}}
\newcommand{\N}{\mathbb{N}}

\newcommand{\Uninv}{\mathbf{U}_n^{-1}}
\newcommand{\Un}{\mathbf{U}_n}
\newcommand{\U}{\mathbf{U}}
\newcommand{\Uinv}{\mathbf{U}^{-1}}

\newcommand{\abs}[1]{\left|#1\right|}
\newcommand{\norm}[1]{\left\lVert#1\right\rVert}

\newcommand{\iid}{\overset{\mathrm{i.i.d.}}{\sim}}

\newcommand{\eqD}{\stackrel{D}{=}}
\newcommand{\convAS}{\stackrel{a.s.}{\longrightarrow}}
\newcommand{\outerconvAS}{\stackrel{\mathrm{as}^\star}{\longrightarrow}}
\newcommand{\condconvAS}{\underset{M}{\overset{\mathrm{as}^\star}{\rightsquigarrow}}}
\newcommand{\condconvP}{\underset{M}{\overset{\mathbb{P}^\star}{\rightsquigarrow}}}
\newcommand{\convP}{\stackrel{\mathbb{P}}{\longrightarrow}}
\newcommand{\outerconvP}{\stackrel{\mathbb{P}^\star}{\longrightarrow}}

\DeclareMathOperator{\Var}{Var}
\DeclareMathOperator{\diam}{diam}
\DeclareMathOperator{\Cov}{Cov}

\DeclareMathAlphabet{\mymathbb}{U}{BOONDOX-ds}{m}{n}

\theoremstyle{plain}
\newtheorem{sectioncount}{xxxxxxx}[section]
\newtheorem{theorem}[sectioncount]{Theorem}
\newtheorem{proposition}[sectioncount]{Proposition}
\newtheorem{lemma}[sectioncount]{Lemma}

\theoremstyle{definition}
\newtheorem{example}[sectioncount]{Example}
\newtheorem{condition}[sectioncount]{Condition}

\newtheorem{remark}[sectioncount]{Remark}
\newtheorem{setting}[sectioncount]{Setting}



\newcommand{\footremember}[2]{
	\footnote{#2}
	\newcounter{#1}
	\setcounter{#1}{\value{footnote}}
}
\newcommand{\footrecall}[1]{
	\footnotemark[\value{#1}]
}

\newcommand{\blind}{1}

\begin{document}

\def\spacingset#1{\renewcommand{\baselinestretch}%
{#1}\small\normalsize} \spacingset{1}

\if1\blind
{
  \title{\bf Gromov-Wasserstein Barycenter Surrogates: \\ Statistical Methodology, Distributional Limits and Applications\thanks{
The authors gratefully acknowledge support of the Deutsche Forschungsgemeinschaft (DFG, German Research Foundation) FOR 5381 ``Mathematical Statistics in the Information Age''.}}
\author{\begin{tabular}{ c c c }
	Florian Steinkamp\hspace{-0.3em}\footremember{ims}{Institute for Mathematical Stochastics, University of G{\"o}ttingen, Goldschmidtstra{\ss}e 7, 37077 G{\"o}ttingen, Germany} & Luis Rodriguez \hspace{-0.3em}\footrecall{ims} &
    Jan Victor Otte \hspace{-0.1em}\footrecall{ims}
  \end{tabular} \\[1em]
  \begin{tabular}{c}
    Axel Munk\hspace{-0.3em}\footrecall{ims}\hspace{-0.3em}\footremember{mbexc}{Cluster of Excellence ``Multiscale Bioimaging: from Molecular Machines to Networks of Excitable Cells'' (MBExC), University of G{\"o}ttingen, Robert-Koch-Stra{\ss}e 40, 37075 G{\"o}ttingen, Germany}
\end{tabular}}
\maketitle
} \fi

\if0\blind
{
  \bigskip
  \bigskip
  \bigskip
  \begin{center}
    {\LARGE\bf Gromov-Wasserstein Barycenter Surrogates: \\ Statistical Methodology, Distributional Limits and Applications}
  \end{center}
  \medskip
} \fi

\bigskip\vspace{-0.6cm}
\begin{abstract}
We introduce statistical theory for the matching of finitely many objects, represented as metric measure spaces (mm-spaces). The approach is based on the second lower bound (SLB) of the Gromov-Wasserstein distance and thus is able to identify deviations in the distributions of the (pairwise) distances within each mm-space. We introduce a surrogate of the SLB barycenter which can be easily computed and expressed explicitly in terms of the distance distributions of each object. When comparing $m$ mm-spaces for $n$ randomly drawn samples in each space, the resulting statistic then can be calculated efficiently in $O(m \cdot n^2 \log(n))$ basic operations. We derive the asymptotic distribution and finite-sample bounds of the proposed test statistic, which serves as a basis for a variety of tools for statistical inference, specifically an asymptotic test for pose-invariant object discrimination and a classification method (based on the SLB barycenter) with controlled error rates. These methods are investigated in simulations and applied to the structural comparison of protein domains.
\end{abstract}

\noindent%
\vfill

\newpage

\section{Introduction}\label{sec:intro}
As capabilities for acquiring and storing data with intrinsic geometric structure have increased massively over the years, finding meaningful ways of describing and comparing such data is an ever so important task in many different fields. An example of this is the protein database (PDB) \citep{bermanprotein}, which has seen drastic growth over recent years, starting from just seven registered structures in 1970 \citep{BURLEY2021100559} and counting a new record number of over $250,\!000$ structures in 2026. In databases such as SCOPe \citep{chandonia2022scope} or CATH \citep{ORENGO19971093}, these proteins are decomposed into smaller subparts, so-called protein domains, which are the building blocks of proteins \citep{schulz2013principles}. These are then, among other factors, classified based on functional, evolutionary and structural similarity. Since the geometric structure of proteins is closely linked to both their functionality and evolutionary origin \citep{kolodny2005comprehensive}, a principled comparison of these structures is essential for understanding previously uncharacterized protein domains. However, a key challenge arises from the fact that protein domains have an intrinsic variation and do not possess a canonical orientation, as illustrated in \autoref{fig: two protein domains same protein}. Therefore, a robust and statistically sound method is required that allows these objects to be identified up to their orientation, meaning rotations, translations or reflections. 
\begin{figure}[b]
\begin{minipage}[t]{0.32\textwidth}
    \vspace{0pt}
    \textbf{A}\par
    \includegraphics[width=\linewidth]{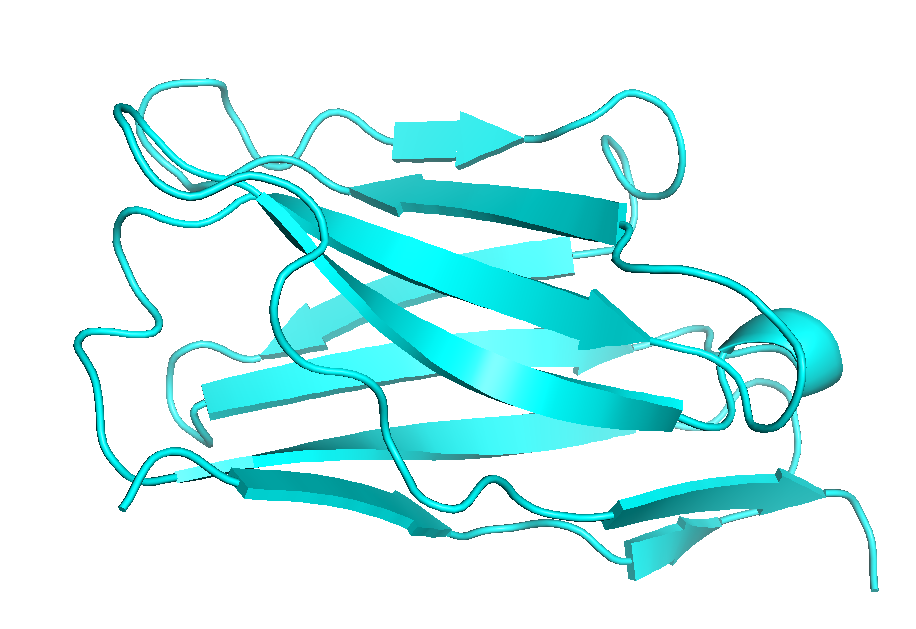}
\end{minipage}
\hfill
\begin{minipage}[t]{0.25\textwidth}
    \vspace{0pt}
    \textbf{B}\par
    \includegraphics[width=\linewidth]{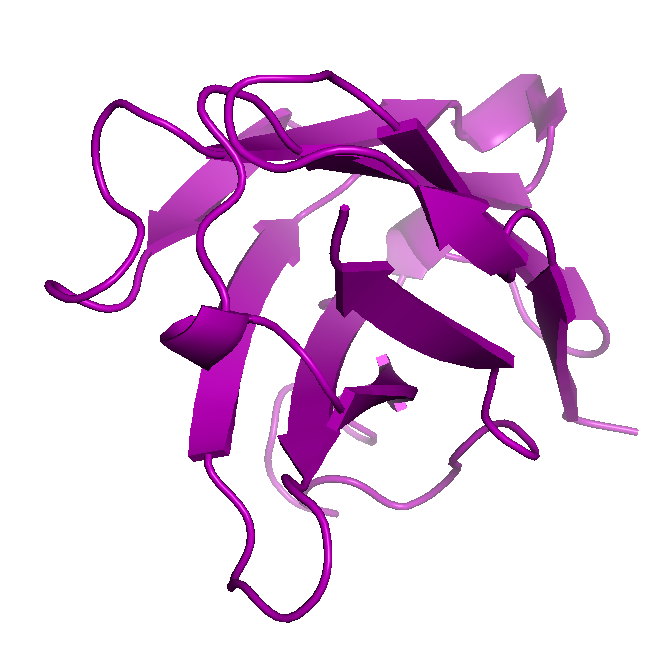}
\end{minipage}
\hfill
\begin{minipage}[t]{0.32\textwidth}
    \vspace{0pt}
    \textbf{C}\par
    \includegraphics[width=\linewidth]{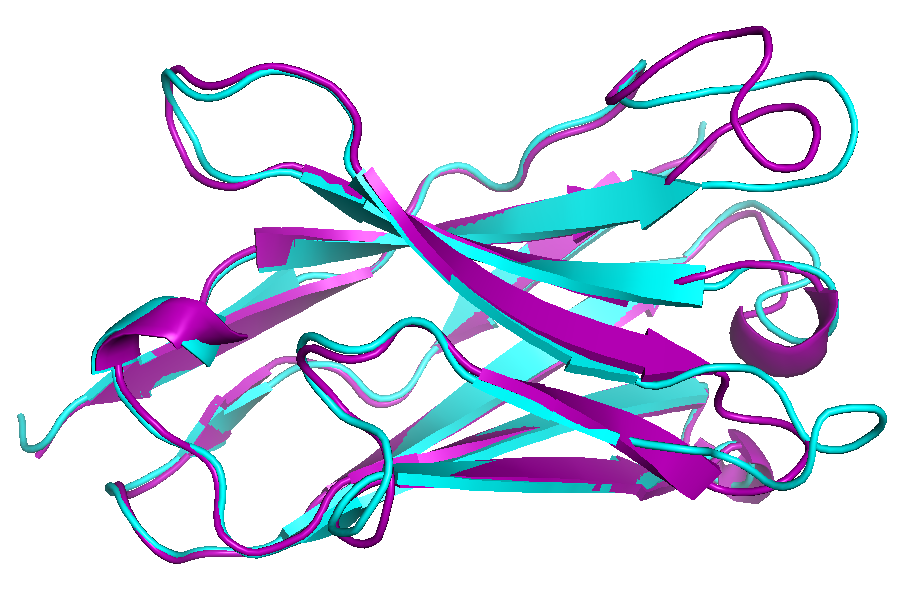}
\end{minipage}
\caption{\textbf{A and B:} Cartoon representations of protein domains d1nfdf1 (\textbf{A}) and d1xhgb1 (\textbf{B}) from SCOPe database, created with PyMOL \citep{pymol}. \textbf{C:} Alignment of d1nfdf1 and d1xhgb1 using PyMOL.}\label{fig: two protein domains same protein}
\end{figure}
\\
Besides protein matching, the (pose-invariant) comparison of objects based on structural similarity is relevant to many fields. Examples include computational anatomy where the goal is to compare different anatomical objects to each other based on their shape or structural features \citep{glaunes2004landmark}, and more generally a variety of tasks in computer vision, to name a few.
\subsection{How to Model the Data}
We follow the approach of \citet{memoli2007use,memoli2011gromov} and model such data as \textit{metric measure spaces} (mm-space), i.e.\ a triple $(\mathcal{X},d_\mathcal{X},\mu)$ where $(\mathcal{X},d_\mathcal{X})$ is a compact metric space and $\mu$ a probability measure on $\mathcal{X}$. It has been shown that this provides a highly flexible framework to capture the intrinsic geometry of the data well, recent applications range from graph matching and node embedding \citep{xu2019scalable} over network analysis \citep{chowdhury2019gromov} to robotics \citep{lyu2026last}. In order to compare such mm-spaces, we say that $(\mathcal{X},d_\mathcal{X}, \mu)$ and $(\mathcal{Y},d_\mathcal{Y}, \nu)$ are \textit{isomorphic} if there exists an isometry $\phi:\mathcal{X} \to \mathcal{Y}$ such that $\nu = \phi_\# \mu$, where $\phi_\#\mu$ denotes push-forward of $\mu$ under $\phi$. This is denoted by $(\mathcal{X},d_\mathcal{X}, \mu) \cong (\mathcal{Y},d_\mathcal{Y}, \nu)$, or for the ease of notation, $\mathcal{X} \cong \mathcal{Y}$ or $\mu \cong \nu$ depending on the context. Notably, if both mm-spaces can be embedded in the Euclidean space, they are isomorphic if and only if they can be transformed onto each other by a composition of translation, reflection and rotation, while also preserving the total mass. As in many applications the goal is to identify objects represented as Euclidean vectors precisely up to those rigid transformations, the above notion of similarity between mm-spaces is particularly attractive. A tool for identifying isomorphic relationships between metric measure spaces is the \textit{Gromov-Wasserstein (GW) distance}, which was introduced by \citet{memoli2007use, memoli2011gromov} and is based on earlier work of \citet{Gromov1999}, see also \citet{sturm2006geometry} for a related concept. It is defined for $p \geq 1$ as
\begin{equation}
    \mathcal{GW}_p^p(\mathcal{X}, \mathcal{Y}) = \inf_{\pi \in \Pi(\mu, \nu)} \int_{\mathcal{X} \times \mathcal{Y}} \int_{\mathcal{X} \times \mathcal{Y}} |d_\mathcal{X}(x,x^\prime) - d_\mathcal{Y}(y, y^\prime)|^p \diff \pi(x,y) d\pi(x^\prime, y^\prime),
    \label{eq: GW-distance}
\end{equation}
where $\Pi(\mu,\nu)$ denotes the set of couplings between the metric measure spaces, i.e.\ probability measures $\pi$ on $\mathcal{X} \times \mathcal{Y}$ for which $\pi(A \times \mathcal{Y}) = \mu(A)$ and $\pi(\mathcal{X} \times B) = \nu(B)$ for every measurable $A \subset \mathcal{X}$ and $B \subset \mathcal{Y}$. Importantly, it has been shown that the GW distance defines a metric on the space of isomorphism classes of $\mathcal{G}_w$, the collection of mm-spaces \citep[Theorem 5.1]{memoli2011gromov}. Hence, it holds that $\mathcal{GW}_p(\mu , \nu) = 0$ if and only if $\mu \cong \nu$. Despite its conceptual appeal, the downside of the GW distance for practical applications lies in the difficulty of its computation. Already in the case that $\mathcal{X}$ and $\mathcal{Y}$ consist of finitely many points, its computation boils down to solving a highly non-convex quadratic optimization problem, which is NP-hard, in general \citep{memoli2011gromov,memoli2023ultrametric,kravtsova2024np}. Thus, its usage in practical situations is severely hindered. While there are approaches to numerical approximations of the GW distance, see e.g.\ \citet{scetbon2022linear, ryner2023globally, rioux2024entropic,doi:10.1137/25M1777116}, no convergence guarantees to the global optimum are known, in general. To circumvent this issue, one may try to approximate the GW distance using lower bounds, which keep much of the discriminatory power of the original distance while allowing for easier computation. The lower bound considered in this article is the so called \textit{Second Lower Bound} (SLB), introduced by \citet{memoli2011gromov}, and defined for $p \geq 1$ as a relaxation of the GW distance via
\begin{equation}
    \mathcal{SLB}_p^p(\mathcal{X}, \mathcal{Y}) = \inf_{\pi \in \Pi(\mu \otimes\mu, \nu \otimes \nu)} \int_{\mathcal{X}^2\times \mathcal{Y}^2} |d_\mathcal{X}(x,x^\prime)- d_\mathcal{Y}(y, y^\prime)|^p \diff \pi(x,x^\prime,y,y^\prime) .
    \label{eq: SLB-distance}
\end{equation}
Note that in contrast to~\eqref{eq: GW-distance}, the couplings are given by those with independent marginals $\mu \otimes \mu$ and $\nu \otimes \nu$, hence ignoring the absolute local information of points in $\mathcal{X}$ and $\mathcal{Y}$. In particular, one has \citep[Proposition 6.2]{memoli2011gromov}
\begin{equation}
    \mathcal{SLB}_p(\mathcal{X}, \mathcal{Y}) \leq \mathcal{GW}_p(\mathcal{X}, \mathcal{Y}), \qquad p \geq 1.
    \label{eq: SLB bounds GW}
\end{equation}
In fact, rather than finding the optimal matching between all points in the mm-spaces,~\eqref{eq: SLB-distance} compares the \textit{distance distributions} of the mm-spaces, $U(t) = \mathbb{P}(d_\mathcal{X}(X,X^\prime) \leq t)$ and $V(t) = \mathbb{P}(d_\mathcal{Y}(Y, Y^\prime) \leq t)$. More precisely, denoting $U^{-1}$ and $V^{-1}$ the corresponding quantile functions, it holds that \citep[Theorem 24]{chowdhury2019gromov}
\begin{equation*}
    \mathcal{SLB}_p^p(\mathcal{X}, \mathcal{Y}) = \int_0^1 \left| U^{-1}(t) - V^{-1}(t) \right|^p \diff t , \qquad p \geq 1.
\end{equation*}
The above equation also shows that the SLB can be interpreted as the $p$-th Wasserstein distance between the distance distributions of the mm-spaces, further implying that $\mathcal{SLB}_p(\mathcal{X}, \mathcal{Y}) = 0$ is equivalent to $U=V$. While the SLB does not keep track of the individual location of points, and hence may lose some discriminatory power in comparison to the GW distance, our experiments demonstrate that the SLB still distinguishes well between non-isomorphic mm-spaces in the proposed scenarios. This is in line with practical experience that distance distribution or distance matrices themselves have been successfully used for comparing objects based on their geometry, see for example \citet{campbell2011performance}  or \citet{holm2020dali}. While the discriminatory power of distance matrices and of the SLB is well known, the investigation of its statistical properties when estimated from data is still at its infancy.
\subsection{The Proposed Approach}\label{subsection: the proposed approach}
Within this article, we aim to extend the work of \citet{weitkamp2024distribution} from the two-sample comparison of mm-spaces to a finite number of mm-spaces $(\mathcal{X}_i,\mu_i,d_i)_{1 \leq i \leq m}$. In particular, we will introduce a SLB based \emph{barycenter} of these spaces and based on this provide a test for the hypothesis that all these spaces carry the same distance distribution,
 \begin{equation} \label{eq: nullhypothesis}
 H_0: \mathcal{SLB}_p(\mathcal{X}_i, \mathcal{X}_j) = 0, \, 1\leq i < j \leq m.
 \end{equation}
Apart from testing, the SLB based barycenter can be easily displayed graphically as it represents the collection of mm-spaces via a single distribution, rendering it particularly useful for classification tasks, see \autoref{theorem: barycenter classifier} ahead. To formalize, denote by $\mu$ a coupling of $\mu_1, \ldots , \mu_m$, i.e.\ a joint distribution $\mu$ on $\bigotimes_{i=1}^m (\mathcal{X}_i,d_i)$ with marginals $\mu_1, \ldots, \mu_m$. Furthermore, assume we observe an i.i.d.\ sample $X_1, \ldots, X_n$ from this coupling, where $X_i = (X_{1,i}, \ldots, X_{m,i})$ for each $i = 1, \ldots,n$. We denote the cdf of the distribution of distance measure $\mu^{U_i}$ of the $i$-th metric measure space by $U_i(t) = \mathbb{P}(d_i(X_{i,1},X_{i,2}) \leq t)$ and its empirical counterpart
\begin{equation*}
    U_{i,n}(t) = \frac{2}{n(n-1)} \sum_{1 \leq k < l \leq n} \mathbbm{1} \! \left(d_i(X_{i,k}, X_{i,l}) \leq t \right), \qquad i=1, \ldots, m.
\end{equation*}
Moreover, write $U_i^{-1}$ and $U_{i,n}^{-1}$ for their respective generalized inverses (quantile functions) and denote $\mathbf{U}^{-1} = (U_1^{-1}, \ldots , U_m^{-1})$, and similarly $\Uninv$, the vector of the (empirical) quantile functions. \\
We introduce now the Wasserstein barycenter of the distance distributions. Denoting the set of all quantile functions of probability measures on the real line by $\mathcal{Q}$, let 
\begin{equation} \label{eq: Wasserstein SLB barycenter}
    T^p \coloneqq T^p(\mathcal{X}_1, \ldots, \mathcal{X}_m) =\inf_{V^{-1} \in \mathcal{Q}} \frac{1}{m} \sum_{i=1}^m \int_0^1 \left| U_i^{-1}(t) - V^{-1}(t) \right|^p \diff t \qquad p \geq 1.
\end{equation}
As this is a Wasserstein barycenter of probability measures on the real line, an explicit quantile function solving the minimization is readily derived. Let for some $p > 1$ the pointwise barycenter $\Lambda_p:\mathbb{R}^m \to \mathbb{R}$ be defined via $\Lambda_p(x) = \argmin_{y \in \mathbb{R}} \sum_{i=1}^m \left| x_i - y\right|^p$, and for $p=1$ specify a particular minimizer $\Lambda_1(x) = 1/2(x_{(\lfloor (m+1)/2 \rfloor)} + x_{(\lceil (m+1)/2 \rceil)})$, where $x_{(i)}$ is the $i$-th order statistic of $x$. Of special interest are the cases $p=2$ and $p=1$, as then $\Lambda_p$ maps a vector $x \in \mathbb{R}^m$ to its mean or median, respectively. Further, $\Lambda_p(\Uinv(t)) \coloneqq \Lambda_p(U_1^{-1}(t), \ldots , U_m^{-1}(t))$ preserves the monotonicity and left-continuity of the input functions and as such is again a quantile function. Consequently, it holds that 
\begin{equation*}
    T^p  = \frac{1}{m}\sum_{i=1}^m \int_0^{1}\left| U_i^{-1}(t) - \Lambda_p(\Uinv(t)) \right|^p \diff t, \qquad p \geq 1.
\end{equation*}
We will from now on call $\Lambda_p\left(\Uinv \right)$, or more precisely the distribution induced by this quantile function, the \textit{SLB barycenter}. Note that $\Lambda_p$ is invariant under permutations of the components of the vector $\Uinv$. Further, it is completely contained in the band
\begin{equation*}
    \Lambda_p\left(\Uinv(t) \right) \in \left[ \min_{i=1,\ldots, m} U_{i}^{-1}(t), \max_{i=1,\ldots, m} U_{i}^{-1}(t) \right], \qquad t \in (0,1),
\end{equation*}
and assigns to each $t\in (0,1)$ the pointwise barycenter of $U_1^{-1}(t), \ldots, U_m^{-1}(t) \in \mathbb{R}$. \newline
It is easy to verify that $T^p = 0$ if and only if $H_0$ is satisfied in~\eqref{eq: nullhypothesis}, making it a suitable candidate to detect deviations from $H_0$ in the distance distributions, i.e.\ \begin{equation} \label{eq: alternative definition}
H_1 : \exists i \neq j \text{ s.t. }\mathcal{SLB}_p(\mathcal{X}_i, \mathcal{X}_j) >0.
\end{equation}
We will show that (\autoref{prop: SLB bounds GW}) $T^p$ is a lower bound to the GW barycenter evaluation
\begin{equation}
    T_{GW}^p \coloneqq\inf_{\mathcal{X} \in \mathcal{G}_w} \frac{1}{m} \sum_{i=1}^m \mathcal{GW}_p^p(\mathcal{X}_i, \mathcal{X}),
\end{equation}
generalizing~\eqref{eq: SLB bounds GW} to the case of barycenters. As mentioned before, the computation of the GW distance is generally NP-hard, which immediately transfers to the quantity $T_{GW}^p$. Hence, the SLB barycenter $T^p$ can be viewed as a computationally feasible relaxation of the GW barycenter $T_{GW}^p$ in the same way that the SLB~\eqref{eq: SLB-distance} relaxes the GW distance~\eqref{eq: GW-distance}. \newline
\noindent
Denote the sample counterpart of $T^p$ by
\begin{equation*}
    T_n^p  = \frac{1}{m}\sum_{i=1}^m \int_0^{1}\left| U_{i,n}^{-1}(t) - \Lambda_p(\Uninv(t)) \right|^p \diff t.
\end{equation*}
Writing $\lbrace d_{(j)}^{\mathcal{X}_i} \rbrace_{1 \leq j \leq n(n-1)/2}$ for the order statistics of $\lbrace d_i(X_{i,k}, X_{i,l}) \rbrace_{1 \leq l < k \leq n }$, it is thus possible to rewrite
\begin{equation} \label{eq: elementary operations for calculation}
    T_n^p =  \frac{2}{mn(n-1)}\sum_{i=1}^m \sum_{k=1}^{n(n-1)/2} \left| d_{(k)}^{\mathcal{X}_i} - \Lambda_p(d_{(k)}^{\mathcal{X}_1},\ldots , d_{(k)}^{\mathcal{X}_m}) \right|^p,
\end{equation}
which implies it can be calculated in $O(m \cdot n^2 \log(n))$ elementary operations, assuming that calculating each distance takes only one elementary operation and calculating each barycenter takes $O(m)$. The $\log(n)$ factor stems from the sorting of the distance distribution sample required for calculating its order statistic.
\begin{remark}\label{rem: comments SLB approach}
We close this section by commenting on alternative ways of testing for $H_0$ and the advantages of taking the SLB barycenter approach.
\begin{enumerate}
\item{(Exact barycenter).} 
     Note that $\Lambda_p(\Uinv)$ may not always induce an mm-space, in the sense that there may not be an mm-space $(\mathcal{Y},d_\mathcal{Y}, \nu) \in \mathcal{G}_w$ such that $V^{-1} = \Lambda_p(\Uinv)$ is the quantile function of its distance distribution. However, $T^p>0$ holds under every alternative $H_1$ in~\eqref{eq: alternative definition}, which renders it useful for testing purposes. Furthermore, the simple representation of the barycenter quantile function $\Lambda_p(\Uinv)$ allows for a fruitful mathematical treatment. In contrast, deriving the precise SLB barycenter itself requires solving an optimization problem over the highly non-convex set
     \begin{equation*}
        \Tilde{\mathcal{Q}} = \left\{ F^{-1} \in \mathcal{Q}: \exists (\mathcal{Y},d_\mathcal{Y}, \nu) \in \mathcal{G}_w \text{ s.t.\ }F(t) = \mathbb{P}\left( d_\mathcal{Y}(Y, Y^\prime) \leq t \right) \text{ with } Y,Y^\prime \iid \nu \right\},
    \end{equation*}
    i.e.\ the set of quantile functions belonging to the distance distribution of some mm-space. Since $\Tilde{\mathcal{Q}} \subset \mathcal{Q}$, $T^p$ is a convex relaxation of the exact SLB barycenter evaluation.
    \item{(Pairwise testing).} One could also design a test statistic for $H_0$ using the pairwise SLB's between the metric measure spaces, i.e.
    \begin{equation*}
        T_{PW,n}^p \coloneqq \frac{2n^{p/2}}{m(m-1)} \sum_{1 \leq i < j \leq m} \int_{0}^1 \abs{U_{i,n}^{-1}(t) - U_{j,n}^{-1}(t)}^p \diff t.
    \end{equation*}
    However, a direct calculation of the pairwise statistic $T_{PW,n}^p$ takes $O(m^2 \cdot n^2 \log(n))$ elementary calculation operations, which is one order (in $m$) of magnitude larger than computing $T_n^p$. This can become critical when comparing a large number of mm-spaces. Further, the barycenter approach bears the additional advantage that it summarizes an ensemble of mm-spaces via a single representative quantile function.
    \item{(Bayesian viewpoint).} Wasserstein barycenters recently have also gained interest from a Bayesian point of view and our SLB barycenter may serve as a proxy for an extension to mm-spaces. Given a prior $\Pi \in \mathcal{P}_p(\mathcal{P}_p(\mathbb{R}))$ and a sampling model $p(D \! \mid \! \pi)$, the Bayes estimator $\hat{\pi}$ arises as the Wasserstein barycenter of the posterior distribution $\Pi(\de  \pi \! \mid \! D)$ under the loss function $L(\pi_1, \pi_2) = \mathcal{W}_p^p(\pi_1, \pi_2)$ (the $p$-Wasserstein distance, see \citet{villani2021topics}), i.e.\ $
        \hat{\pi} = \argmin_{\pi \in \mathcal{P}_p(\mathbb{R})} \int_{\mathcal{P}_p(\mathbb{R})} L(\pi, \pi^\prime) \Pi(\diff \pi^\prime \mid D)$. We mention e.g.\ \citet{srivastava2018scalable} for its use in approximate Bayesian computation or \citet{backhoff2022Bayesian} for predictive estimation. In our context, the SLB barycenter is thus the Bayes estimator under Wasserstein loss if the posterior law is given by 
        \begin{equation*}
            \Pi(\de \pi \! \mid \!  D) = \frac{1}{m} \sum_{i=1}^m \delta_{\mu_{i,n}^U}(\de \pi),
        \end{equation*}
        where $\mu_{i,n}^U$ is the probability measure corresponding to $U_{i,n}$.
\end{enumerate}
\end{remark}
\subsection{Main Results: Limit Distributions}
Our main theoretical contribution is the limit law of the (possibly trimmed) SLB barycenter statistic
\begin{equation} \label{eq: untrimmed test statistic}
    T_{\beta,n}^p = \frac{1}{m} \sum_{i=1}^m \int_\beta^{1-\beta} \left| U_{i,n}^{-1}(t) - \Lambda_p(\Uninv(t)) \right|^p \diff t, \qquad \beta \in [0,1/2),
\end{equation}
the population counterpart of which will be denoted by $T_\beta^p$. The trimming of $\beta \geq 0$ can be beneficial in several ways: It can robustify the presented methodology, which is reflected in simpler conditions under which the asymptotic results will be derived (see \autoref{theorem: Convergence under the Null} and \autoref{theorem: Convergence under the alt}). Further, it allows to highlight certain regions of the distance-to-distance distributions, which are particularly informative for a given task. We will derive the asymptotic distribution of $T_{\beta,n}^p$ for $\beta \geq 0$ both in the case that the distance distributions are equal and not. Based on these results, we define a statistical test for $H_0$. We furthermore provide an explicit finite sample risk bound including a bound for the bias of $T_{\beta,n}^p$, which is then applied in order to bound the finite sample error rate of the SLB barycenter based classifier in \autoref{theorem: barycenter classifier}. \\
Many of the theoretical results of this article hinge on the joint weak convergence of the multivariate quantile process $\mathbb{U}_n^{-1} \coloneqq (\mathbb{U}_{1,n}^{-1}, \ldots, \mathbb{U}_{m,n}^{-1})$, where $\mathbb{U}_{i,n}^{-1} (t) \coloneqq \sqrt{n}(U_{i,n}^{-1}(t) - U_i^{-1}(t))$, as well as the related barycenter process $\lambda_{p,n} \coloneqq \sqrt{n}(\Lambda_p(\Uninv) - \Lambda_p(\Uinv))$, which may be of interest by its own. For $\beta > 0$, we will view each marginal of $\mathbb{U}_n^{-1}$ as an element of the space of bounded functions $\ell^\infty([\beta,1-\beta])$ under mild conditions. In order to derive the asymptotic distribution of the process, the delta method combined with the quantile inversion $\Phi_{\text{inv}}: F \mapsto F^{-1}$, plays a pivotal role. Considering $\Phi_{\text{inv}}$ as a map from the space of trimmed distribution functions (i.e.\ c\`{a}dl\`{a}g) onto $\ell^\infty([\beta,1-\beta])$, the following condition is imposed in order to ensure Hadamard-differentiability of the functional at $U_i$.
\begin{condition}\label{cond 1.2}
    For $\beta \in (0,1/2)$ and $p \geq 1$, assume that for each $i\in \lbrace 1,\ldots, m \rbrace$ the distribution function $U_i$ is continuously differentiable on an interval $[C_{i,1},C_{i,2}] = [U_i^{-1}(\beta)- \varepsilon,U_i^{-1}(1- \beta)+ \varepsilon]$ for some $\varepsilon > 0$ with strictly positive derivative $u_i$.
\end{condition}
For $\beta = 0$, special care needs to be taken as it might happen that $\mathbb{U}_{i,n}^{-1}$ grows in an uncontrollable way at the boundaries of the interval $(0,1)$. While \citet[Lemma 3.10.24]{wellner2013weak} also provides conditions under which $\invmap$ is Hadamard differentiable as a function mapping to the space of bounded functions, it turns out that integrability of the process is sufficient to derive its limiting distribution. Hence, only the following, milder conditions are required.
\begin{condition}\label{cond 1.3}
    Assume that for each $i=1, \ldots,m$, the distribution function $U_i$ is continuously differentiable on its support $[0, D_i]$, such that the derivative satisfies $u_i(t) > 0$ for every $t \in (0,D_i)$. Further, assume that there is some $\theta \geq 1$ for which exist constants $-1/\theta - 1/2 <\gamma_{i,1}, \gamma_{i,2}< \infty$ and $c_i$ such that $(U_i^{-1})^\prime(t) \leq c_i t^{\gamma_{i,1}}(1-t)^{\gamma_{i,2}}$ for $t \in (0,1)$.
\end{condition}
These conditions resemble the ones used by \citet{weitkamp2024distribution} and the examples of mm-spaces given there immediately transfer to our setting.
\paragraph{Limit Distribution under $H_0$}
In the following, we denote by $\rightsquigarrow$ weak convergence in the sense of Hoffmann-J\o rgensen \citep[chap.~1.3]{wellner2013weak}. Under $H_0:T^p = 0$ (or more specifically, $H_0^{(\beta)} : T_\beta^p = 0$) and \autoref{cond 1.2} for some $\beta > 0$ and $p \geq 1$, we prove that (\autoref{theorem: Convergence under the Null})
\begin{equation}\label{eq: limit distribution under the Null}
    n^{p/2} T_{\beta,n}^p \rightsquigarrow  \frac{1}{m} \sum_{i=1}^m \int_\beta^{1-\beta} \left|\mathbb{G}_i(t) - \Lambda_p(\mathbb{G}(t)) \right|^p \diff t,
\end{equation}
where $\mathbb{G} = (\mathbb{G}_1, \ldots, \mathbb{G}_m)$ is a certain m-dimensional centered Gaussian process. If additionally, \autoref{cond 1.3} is satisfied with $\theta = p$, the same distributional limit as in equation~\eqref{eq: limit distribution under the Null} holds with $\beta = 0$, as well. We further accompany the distributional limit~\eqref{eq: limit distribution under the Null} with a bootstrap scheme based on a bootstrap statistic $T_{\beta,n}^{p , \ast}$ and prove its consistency in \autoref{theorem: bootstrap consistency}. The limit law~\eqref{eq: limit distribution under the Null} as well as its bootstrap counterpart are further investigated in extensive simulation studies (\autoref{subsec:sim:NH} and \autoref{subsec:sim:bootstrap}), while the corresponding statistical test is evaluated on synthetic data in \autoref{subsec:sim:test}.
\paragraph{Limit Distribution under the Alternative $H_1$}
We again start by considering $\beta > 0$. Then, for $p \geq 2$, under the assumption that $T_\beta^p >0$ and \autoref{cond 1.2}, we show that (\autoref{theorem: Convergence under the alt}) the statistic $T_{\beta,n}^p$, upon a proper centering and rescaling, is asymptotically normal with mean zero and variance $\sigma_\beta^2$, which is defined in terms of the underlying distance distributions. We also provide a distributional result for $p = 1$, in which case the limiting distribution is not necessarily Gaussian. In the untrimmed case $\beta = 0$, we further show that for $p = 1$ and $p \geq 2$, the same limit laws hold under \autoref{cond 1.3} with $\theta = 1$. These results are supported by simulations, which are provided in \autoref{subsec:sim:alternative}. Finally, in \autoref{theorem: local alternatives} we provide the asymptotic distribution of the centered test statistic under local alternatives.
\begin{figure}[b!]
    \centering
    \includegraphics[width=1\linewidth]{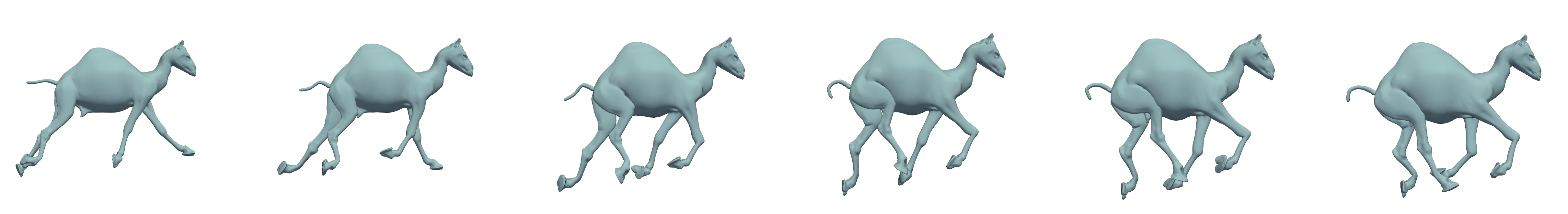}
    \includegraphics[width=0.6\linewidth]{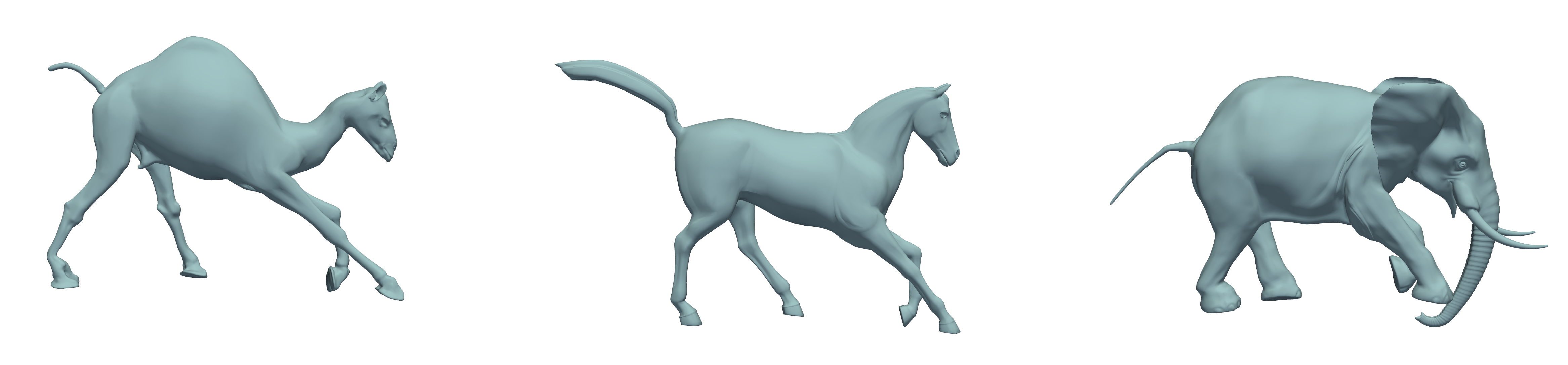}
    \caption{\textbf{First row:} Example meshes from \textit{camel-gallop} class. \textbf{Second row:} Examples for camel, horse and elephant objects.}\label{fig: Meshes pictures}
\end{figure}
\subsection{Applications}\label{subsec:applications}
We claim that distance distributions and SLB barycenters can be a powerful tool for the classification of objects. To this end, we illustrate the SLB barycenters by means of the triangulated objects database \citep{sumner2004deformation}, which contains meshes corresponding to a variety of different animals, see also \autoref{fig: Meshes pictures}. Upon sampling points uniformly on the surfaces of meshes contained in the classes \emph{horse-gallop}, \emph{elephant-gallop} and \emph{camel-gallop}, the resulting distance distributions reflect the prior conception of similarity well. As in each of the classes, animals are displayed in different poses and thus can be seen as metric preserving perturbations of the same object, the distance distributions within these classes should be rather similar, see part A of \autoref{fig: visualization of distance distributions meshes and proteins}. When looking at the resulting distance distributions of each class, the distinct structure of the classes is clearly reflected in the shape of the corresponding distance distributions, which lies at the core of the SLB barycenter classifier method. The performance of the classifier on this dataset is documented in \autoref{subsec:sim:classifier}. 
\paragraph{Classifications of Protein Domains}
\begin{figure}[t!]
    \centering
    \includegraphics[width=1\linewidth]{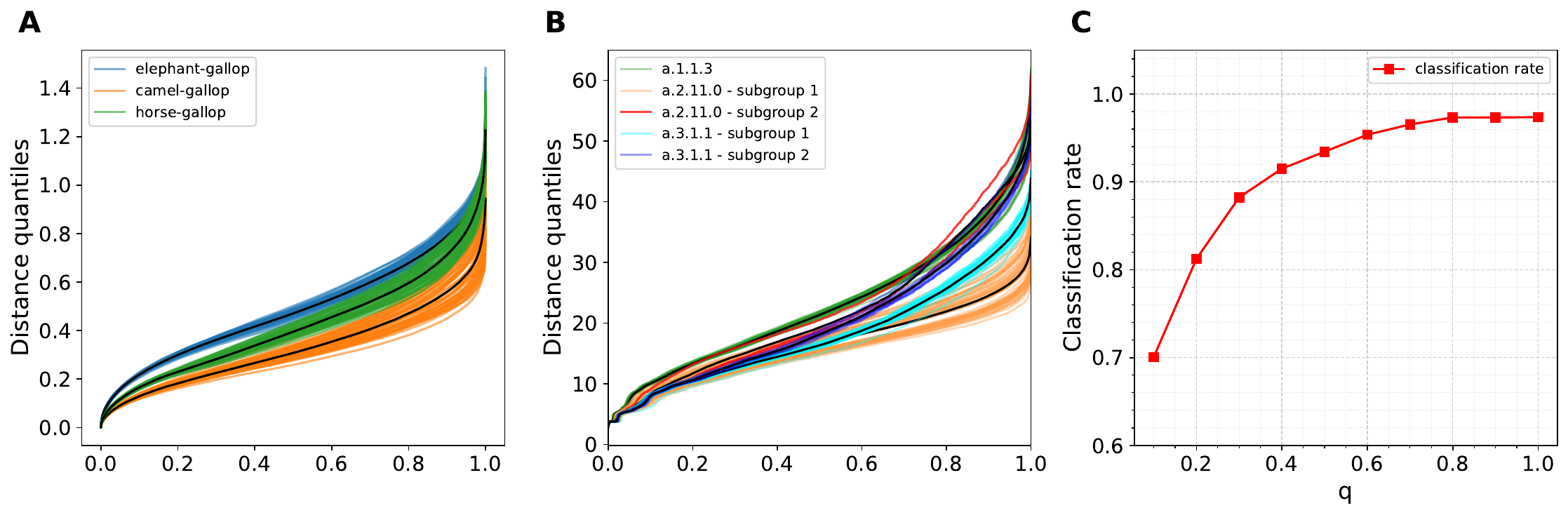}
    \caption{\textbf{A:} Quantile functions induced by the distances of $250$ uniformly sampled points on the surface of a mesh from the triangulated mesh database \citep{sumner2004deformation}. The color of the quantile function displays the class. The SLB barycenter of each group is plotted in black. \textbf{B:} Distance distributions of protein domains from three different SCOPe families. Distance distributions from the groups a.3.1.1 and a.2.11.0 were subclustered via Wasserstein $k$-means into $2$ clusters each. The barycenters of each (sub)-group are displayed in black. \textbf{C:} Rate of correct classification of SLB barycenter classifier for protein domains from families a.2.11.0, a.3.1.1 and a.1.1.3 when subsampling a proportion $q$ of the C-alpha atoms of each domain.}\label{fig: visualization of distance distributions meshes and proteins}
\end{figure}
In the following, we present the results for the classification of structural protein domain data. To this end, in \autoref{theorem: barycenter classifier} we provide a finite sample bound for the classification error of the SLB barycenter classifier introduced in \autoref{subsec: barycenter classifier}. The data stems from the SCOPe protein database \citep{fox2014scope}, where protein domains are classified into classes, folds, superfamilies and families. Families usually consist of closely related protein domains on an evolutionary and structural level. In superfamilies, families that are believed to share a common evolutionary ancestor and similar functionality are assembled. Superfamilies which contain protein domains sharing some structural features but lack evidence of evolutionary connection are then grouped into a fold. On the class level, the folds are classified based on similarity of secondary structure content and organization of the contained protein domains \citep{chandonia2022scope}. For the classification, we consider protein domains contained in the SCOPe families a.2.11.0, a.3.1.1 and a.1.1.3. As encoded in their SCOPe IDs, they are families consisting of protein domains that are made up primarily of alpha helices. However, they are not contained in the same fold, meaning that there is little to no evolutionary or functional connection between them. The protein domain data is provided in terms of the 3D coordinates of atoms belonging to the protein backbone, of which we only take into account the locations of the C-alpha atoms. The backbone of a protein consists of a repeating chain of atoms that links amino acids via peptide bonds. It is well known that the locations of the C-alpha atoms of a protein determine the structure of its backbone reasonably well \citep{holm1993protein}, which is why in many applications they are used to model the respective proteins \citep{mayr2007comparative}. The metric used to calculate the distances between the respective C-alpha atoms is the Euclidean distance. We then try to reproduce the SCOPe classification of the protein domains by performing a repeated leave-one-out classification.
In order to robustify the method, in a first step the distance distributions of each protein family may be divided into subclusters by applying Wasserstein $k$-means clustering, using Lloyd's algorithm \citep{lloyd1982least} with Wasserstein barycenters as centroids and the Wasserstein distance as metric between the distributions (see \citet{jaffe2025asymptotic} for statistical guarantees of the procedure). This way, we account for differences in the within-family variation of the distance distributions, such that mm-spaces belonging to more diverse groups can be identified more reliably, see \autoref{subsec:sim:classifier} for an example on synthetic data. Note that the preclustering only serves as a preprocessing step on the known, already labeled data, and as such errors in the procedure do not necessarily induce a further statistical error on the classifier. For our experiments, we chose cluster amounts $(2,2,1)$ for the families a.2.11.0, a.3.1.1 and a.1.1.3 upon a visual inspection of the distance distributions of the groups. The resulting subclusters and their respective barycenters are visualized in part B of \autoref{fig: visualization of distance distributions meshes and proteins}. Then, the left-out protein domain is classified to the family to one of the barycenters of which its distance distribution has the lowest Wasserstein distance to. In order to validate the method and its finite sample classification error rate stated in \autoref{theorem: barycenter classifier}, we carry out the classification based on a proportion of $q$ randomly chosen subsamples of the C-alpha atoms of each of the protein domains, for different values of $q \in (0,1]$. Part C of \autoref{fig: visualization of distance distributions meshes and proteins} shows that the rate of correct classification increases as $q \nearrow 1$ in accordance with \autoref{theorem: barycenter classifier}. For $q = 1$, i.e.\ when taking into account all C-alpha atoms of each protein domain, the classification rate goes up to over $0.96$, showing that the SLB barycenter method reflects the SCOPe classification of the protein domains well. \autoref{fig:confusion-matrices} shows confusion matrices of the empirical rates at which a protein domain from family $x$ is classified to family $y$. Further, we want to point out that the results of the classification differ only slightly if the procedure is carried out with similar precluster configurations (see \autoref{table: classification proteins other amts}), meaning that it is robust with respect to such small modelling changes.
\begin{figure}[t]
\centering
\includegraphics[width=\linewidth]{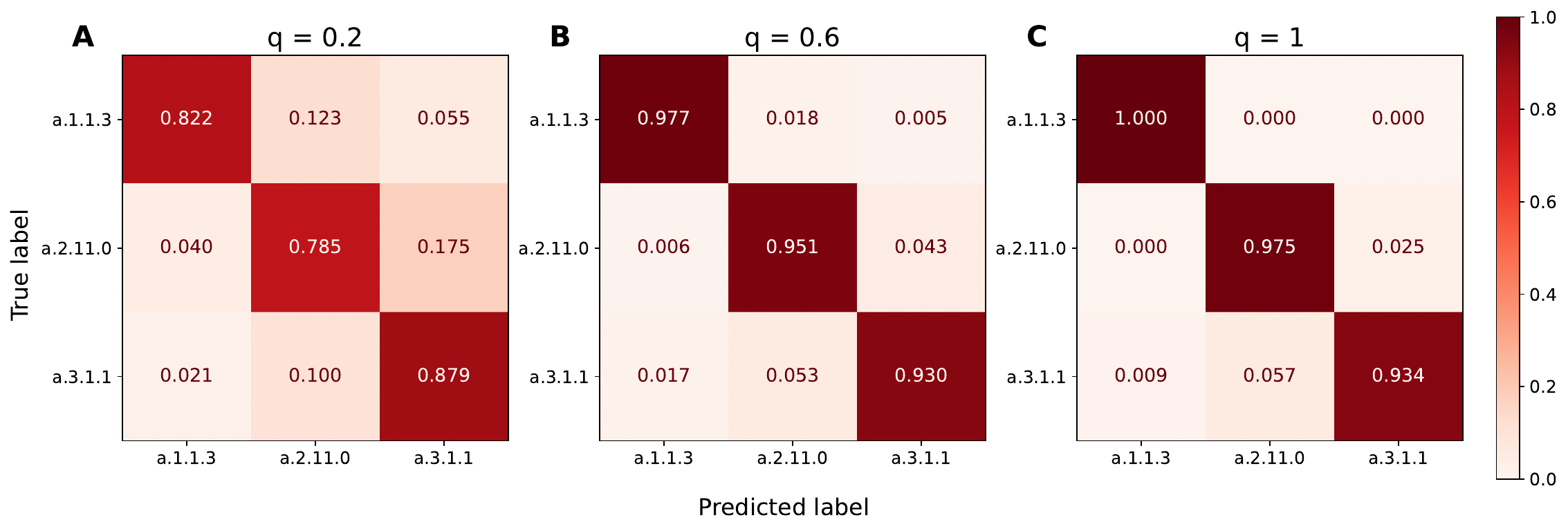}
\caption{Confusion matrices displaying how often a protein domain from \emph{true label} is classified to \emph{predicted label}, for different subsample proportions $q = 0.2,0.6,1$.}\label{fig:confusion-matrices}
\end{figure}
\begin{table}[b]
\centering
\begin{tabular}{cc@{\qquad}cc}
\toprule
Preclusters & Classification rate &
Preclusters & Classification rate \\
\midrule
$(1,1,1)$ & $0.926$ & $(2,3,1)$ & $0.967$ \\
$(1,2,1)$ & $0.964$ & $(3,1,1)$ & $0.949$ \\
$(2,1,1)$ & $0.94$ & $(3,2,1)$ & $0.987$ \\
$(2,2,1)$ & $0.969$     & $(3,3,1)$ & $0.977$ \\
\bottomrule
\end{tabular}
\caption{Rate of correct classifications of SLB barycenter classifier for protein domains from families a.2.11.0, a.3.1.1 and a.1.1.3 over $1,\!000$ repetitions, in each of which one domain from each family is left out and classified, taking into account every C-alpha atom of each domain (i.e.\ $q=1$). \emph{Preclusters} is the number of clusters used in the Wasserstein $k$-means preclustering procedure. The results are rounded to the third digit.}\label{table: classification proteins other amts}
\end{table}
\paragraph{Testing Structural Similarity of Protein Domains} We use the statistical test for $H_0$ from \autoref{theorem: bootstrap consistency} to assess the structural similarity of protein domains. To this end, we first apply the test to Immunoglobulin heavy chain variable domains, which are protein domains from the same SCOPe family, that share additional structural features. In accordance with prior biological knowledge, the test does not reject the null hypothesis of homology between the domains. In a second step, we compare protein domains from a representative subset of SCOPe family b.1.1.3. As such, the domains to be compared are still structurally related, albeit not as closely as in the first test. In this case, the null hypothesis is rejected reliably, which shows the sensitivity of the method even against small deviations from the null hypothesis. Further details are presented in \autoref{sec:application}.
\subsection{Related Work}
Our work extends \citet{weitkamp2024distribution} from the comparison of two mm-spaces to SLB based barycenters of an ensemble of spaces. In addition, we derive the limit distribution under local alternatives and a novel finite sample risk bound, which serves as the basis for the SLB classifier. Moreover, we introduce the possibility of drawing dependent and coupled samples from the metric measure spaces. We finally stress that our novel technique of proof of \autoref{theorem: Convergence under the alt} relies on a careful analysis of the SLB barycenter functional along with a sandwich inequality for the test statistic. In particular, following this strategy allows to generalize \citet[Theorem 2.7]{weitkamp2024distribution} from $p=2$ to general $p \geq 1$. \vspace{0.05cm} \\ 
As mentioned before, the SLB can be viewed as the \textit{Wasserstein distance} between the distance distributions of metric measure spaces. Statistical inference for the Wasserstein distance between one-dimensional cdfs has seen much attention over recent years, with distributional limits and convergence rates being derived by, among others, \citet{munk1998nonparametric, del1999central} or \citet{bobkov2019one}, see also the reviews by \citet{panaretos2019statistical} and \citet{del2025distributional}. Our work is inspired by \citet{del2019central}, who provide a central limit theorem for Wasserstein barycenters of one-dimensional distributions when $p=2$, which they use in order to assess the goodness-of-fit of a deformation model. One of the main differences between these works and the distributional limits in this article is that the empirical distance distributions $U_{i,n}$ are not based on i.i.d. observations anymore, but instead carry an intrinsic dependency induced by the distance-to-distance structure. Hence, it is not possible to directly apply classical empirical process theory \citep{wellner2013weak}. However, $\sqrt{n}(U_{i,n}-U_i)$ may be viewed as a $U$-process. The asymptotic behavior and bootstrap of these types of processes has been intensively studied by Arcones and Gin\'e, see e.g.\ \citet{arcones1994u}. Extending this theory, we derive convergence of these processes in the space of bounded functions $\ell^\infty([C_{i,1},C_{i,2}])$ for properly chosen constants $C_{i,1},C_{i,2}$ under mild assumptions. \vspace{0.05cm} \\ 
The study of optimal transport based barycenters for ground spaces beyond the real line has received considerable attention over the last years. In the special case $d=1$, as considered here, the explicit representation of the barycenter as $\Lambda_p(U_1^{-1}, \ldots, U_m^{-1})$ goes back to \citet{agueh2011barycenters}. One dimensional Wasserstein barycenters were further exploited by \citet{10.1214/15-AOS1387} in the study of the separation of amplitude and phase variation of point processes.
  \vspace{0.05cm} \\ 
Lastly, we want to comment on related works using GW based methods for the analysis of protein data. A statistical, pairwise comparison of proteins can be found in \citet{weitkamp2024distribution}. Note that, in contrast to this work, the examples provided there focus on the comparison of complete proteins, which are composed of multiple domains. Here, we focus on the comparison and classification of the protein domains themselves, which is of particular relevance, as they are the building blocks of proteins and as such determine their function \citep{schulz2013principles}. \citet{gellert2019substrate} applied the \emph{first} lower bound of the GW distance, which is in general less informative than our SLB \citep[sec.~6.1]{memoli2011gromov} in the context of clustering thioredoxins and glutaredoxins, with the goal of finding a functional classification of these proteins. Instead of using the locations of the C-alpha atoms of the proteins, they compared the 3-dimensional isosurfaces of the positive and negative potential of each protein. Moreover, \citet{riahi2025joint} define a generalization of the GW distance, which they applied in order to match complete proteins to their domains. However, they rely on the approximation of an entropically regularized version of their proposed objective, and their results do not include any statistical guarantees.
\subsection{Organization and Software}
In \autoref{section: theory: limit distributions}, we present the theoretical results involving the SLB barycenter statistic, including finite sample risk bounds and distributional limits. From these results we derive a variety of statistical tools, some of which are presented in \autoref{sec: statistical methodology}. In \autoref{sec: Simulations}, the distributional convergence results and statistical methods are investigated in Monte Carlo studies using synthetically generated data. Further details on the application of the SLB barycenter test to structural protein data (recall \autoref{subsec:applications}) are presented in \autoref{sec:application}. Proofs of this paper are rather technical and thus postponed to the appendix, which begins with the derivation of the joint limit law of the $U$-quantile processes $\mathbb{U}_{i,n}^{-1}, \, i=1, \ldots ,m$ and the barycenter process $\lambda_{p,n}$. In \autoref{appendix: proof of main results}, the main results of the paper are proven. \autoref{appendix: bootstrap} contains the derivation of the consistency of the proposed bootstrap procedure. Then, some supplementary simulations to \autoref{sec: Simulations} are, for the sake of readability, presented in \autoref{appendix: simulations}. A Python package, in which the SLB barycenter test (\autoref{theorem: bootstrap consistency}) and the classifier (\autoref{theorem: barycenter classifier}) are implemented is made available at \url{https://github.com/FlorianSteinkamp/SLBBarycenter}.

\section{Theory: Main Results}\label{section: theory: limit distributions}
In the following, we relate the SLB barycenter to the GW barycenter, provide finite sample bounds for the (trimmed) SLB barycenter statistic~\eqref{eq: untrimmed test statistic} and derive limit laws both under the null hypothesis and under (local) alternatives. We begin with the following simple observation.
\begin{proposition}\label{prop: SLB bounds GW}
    In the setting of \autoref{subsection: the proposed approach}, define the GW barycenter evaluation
    \begin{equation*}
    T_{GW}^p \coloneqq\inf_{\mathcal{X} \in \mathcal{G}_w} \frac{1}{m} \sum_{i=1}^m \mathcal{GW}_p^p(\mathcal{X}_i, \mathcal{X}), \qquad p \geq 1.
\end{equation*}
    For $T_\beta^p$, the population counterpart of~\eqref{eq: untrimmed test statistic}, it holds that $T_\beta^p \leq T_{GW}^p$, for any $\beta \in [0,1/2)$.
\end{proposition}
For a proof, see \autoref{proof:prop:barycenter-bound}. The above proposition implies that $T_\beta^p >0$ guarantees $T_{GW}^p>0$, which is a desirable property in the context of testing whether the underlying mm-spaces are isomorphic using the SLB barycenter.
\subsection{Finite Sample Bounds}
In this section, we study the finite sample properties of the SLB barycenter statistic.
\begin{theorem}\label{theorem: finite sample bound the alternative}
    Recall the setting of \autoref{subsection: the proposed approach} and the trimmed SLB barycenter $T_{\beta,n}^p$ given in~\eqref{eq: untrimmed test statistic}. Define $D = \max_{i=1, \ldots, m} \diam(\mathcal{X}_i)$. Then, for any $p,q \geq 1$ and $\beta \in [0,1/2)$
    \begin{equation*}
        \mathbb{E} \left[\abs{T_{\beta,n}^p -T_\beta^p}^q \right] \leq 2^{qp - 3q/2 + 1} \left(p(m+1) \right)^q D^{pq} \Gamma(1 + q/2) \lfloor n/2 \rfloor ^{-q/2},
    \end{equation*}
    for $\Gamma$ the gamma function and $\lfloor x \rfloor$ the largest integer less than or equal to $x \in \mathbb{R}$.
\end{theorem}
When $q = 1$, the above theorem characterizes the finite sample mean absolute error of $T_{\beta,n}^p$, which then decreases at the rate $O(n^{-1/2})$. Together with the case $q=2$, this yields the rate $O(n^{-1})$ for the MSE. The proof of \autoref{theorem: finite sample bound the alternative} is given in \autoref{proof:thm:finitesample-bias-alt}. If the null hypothesis~\eqref{eq: nullhypothesis} is satisfied, the bound of \autoref{theorem: finite sample bound the alternative} can be further improved under minimal additional assumptions: In \citet{bobkov2019one}, it is shown that the finite sample bias of the Wasserstein distance can be bounded in terms of 
\begin{equation*}
    J_p(\mu) \coloneqq \int_{-\infty}^\infty \frac{\left[F(x)(1-F(x)) \right]^{p/2}}{f(x)^{p-1}} \diff x, \qquad p \geq 1,
\end{equation*}
for absolutely continuous $\mu \in \mathcal{P}(\mathbb{R})$ with cdf $F$ and positive density $f$.
\begin{theorem}\label{theorem: sample bias of DoD}
    Assume the setting of \autoref{theorem: finite sample bound the alternative} and let $p,q \geq 1$. If additionally, $H_0: \mathcal{SLB}_p(\mathcal{X}_i, \mathcal{X}_j) = 0$ for $1 \leq i < j \leq m$ is satisfied and $J_{pq}(\mu^{U_1})< \infty$, then it holds that
    \begin{equation*}
        \mathbb{E}\left[\left(T_{\beta, n}^p \right)^q \right] \leq  \left( \frac{10 pq}{\sqrt{n+2}} \right)^{pq} J_{pq}(\mu^{U_1}).
    \end{equation*}
\end{theorem}
For a proof, consult \autoref{proof:thm:finite-sample-bias}.
\subsection{Limit Distribution under the Null}
Next, we turn our attention to the study of the asymptotic distributional behavior of the SLB barycenter statistic. We first provide a limit theorem under the null hypothesis $H_0$.
\begin{theorem}\label{theorem: Convergence under the Null}
    Assume the setting of \autoref{theorem: finite sample bound the alternative}. Further, let $p \geq 1$ and assume that the null hypothesis $H_0$ (see~\eqref{eq: nullhypothesis}) is satisfied.
    \begin{enumerate}
        \item\label{item: first part of theorem under Null}Let $\beta > 0$. Under \autoref{cond 1.2} , it holds that 
        \begin{equation} \label{eq: limiting distribution under the Null}
            n^{p/2}T_{\beta,n}^p \rightsquigarrow  \frac{1}{m} \sum_{i=1}^m \int_\beta^{1-\beta} \left| \mathbb{G}_i(t) - \Lambda_p(\mathbb{G}(t)) \right|^p \diff t,
        \end{equation}
        where $\mathbb{G} = (\mathbb{G}_1, \ldots, \mathbb{G}_m)$ is an $m$-dimensional centered Gaussian process. For $1 \leq i,j \leq m$ and $t,s \in (\beta,1-\beta)$, its covariance is given by
        \begin{equation}\label{eq:Covariance Operator G}
            \Cov\left(\mathbb{G}_i(t), \mathbb{G}_j(s)\right) = \gammaG{i,j}(t,s) \coloneqq \frac{\gammaK{i,j}\left(U_i^{-1}(t), U_j^{-1}(s) \right)}{u_i\Bigl(U_i^{-1}(t)\Bigr)u_j\left(U_j^{-1}(s)\right)},
        \end{equation}
        where for $x, y \in \mathbb{R}$, we define
        \begin{eqnarray*}
            \gammaK{i,j}(x,y) \coloneqq 4 \int \left[ \int \mathbbm{1} \!\left(d_i(a,a^\prime) \leq x \right) \diff \mu_i(a)\int \mathbbm{1} \!\left(d_j(b,b^\prime) \leq y \right) \diff \mu_j(b) \right] \diff \mu_{i,j}(a^\prime,b^\prime) \\
            - 4\int \int \mathbbm{1} \!\left(d_i(a,a^\prime) \leq x \right) \diff \mu_i(a)\diff \mu_i(a^\prime) \int\int \mathbbm{1} \!\left(d_j(b,b^\prime) \leq y \right) \diff \mu_j(b)\diff \mu_j(b^\prime).
        \end{eqnarray*}
        Here, $\mu_{i,j}$ denotes the coupling between $\mu_i$ and $\mu_j$ induced by $\mu$ via
        \begin{equation*}
            \mu_{i,j}(A_i \times A_j) = (\pi_{i,j})_\# \mu(A_i \times A_j), \qquad A_i \in \mathcal{B}(\mathcal{X}_i), A_j \in \mathcal{B}(\mathcal{X}_j),
        \end{equation*}
        where $\pi_{i,j}(x) = (x_i,x_j)$ for $x \in \mathcal{X}_1 \times \ldots \times \mathcal{X}_m$ is the $(i,j)$-th coordinate projection.
        \item If furthermore both parts of \autoref{cond 1.3} are fulfilled with $\theta = p$, then~\eqref{eq: limiting distribution under the Null} also holds for $\beta = 0$.
    \end{enumerate}
\end{theorem}
The proof is given in \autoref{proof:thm:convergence-null}.
\begin{remark}[Different sample sizes]\label{remark: extensions Theorem under null} So far, we have only considered the situation that the samples $X_1, \ldots, X_n$ are drawn iid from a coupling $\mu$ of the mm-spaces. In particular, this means that the number of samples from each space is equal. However, it is easily possible to extend this to a different sample size framework. For simplicity, assume that for each $i = 1, \ldots, m$, we observe independent (also between the mm-spaces) samples $X_{i,1}, \ldots, X_{i,n_i} \sim \mu_i$ (the extension to dependent samples is straightforward). Define the normalization factor $\rho_n \coloneqq \prod_{i=1}^m n_i / \sum_{i=1}^m \prod_{j \neq i} n_j$ and assume that the sample sizes $n_i$ are chosen such that as $n_1, \ldots, n_m \to \infty$ also $\rho_n/n_i \to \lambda_i \in [0,1]$ for each $i =1, \ldots,m$. In the following we also write $T_{\beta,n}^p$ for the empirical counterpart of $T_\beta^p$ in this different sample size setup. Then a slight adjustment of the proof of \autoref{theorem: Convergence under the Null} leads to 
        \begin{equation}
            \rho_n^{p/2} T_{\beta,n}^p\rightsquigarrow \frac{1}{m} \sum_{i=1}^m \int_\beta^{1-\beta} \left| \sqrt{\lambda_i}\mathbb{G}_i(t) - \Lambda_p\left( \sqrt{\lambda_1}\mathbb{G}_1(t), \ldots , \sqrt{\lambda_m} \mathbb{G}_m(t) \right) \right|^p \diff t
        \end{equation}
        under equivalent conditions as in \autoref{theorem: Convergence under the Null} for $\beta >0$ and $\beta = 0$, respectively.
\end{remark}
\subsection{Limit Distribution under the Alternative}\label{subsec:alternative-limit}
In the following, we provide asymptotic results for the centered statistic $\sqrt{n}\left( T_{\beta,n}^p - T_\beta^p \right)$. To this end, for $p \geq 2$, we show the Fr\'echet differentiability of the barycenter $\Lambda_p$. For $p=1$ however, this is not valid anymore and we follow a combinatorial argument. The definition of the map $\psi_x^1$, which characterizes the underlying limiting process of $\lambda_{1,n}$, is notationally involved and thus stated at the beginning of \autoref{section: JWC of DoD process}.
\begin{theorem}\label{theorem: Convergence under the alt}
    Assume the setting of \autoref{theorem: finite sample bound the alternative}. Denote by $\Delta_i = U_i^{-1} - \Lambda_p\left(\Uinv \right)$.
    \begin{enumerate}
        \item For $\beta > 0 $, $p \geq  2$ and under \autoref{cond 1.2}, with $\mathbb{G}$ as in \autoref{theorem: Convergence under the Null}, it holds that
        \begin{equation} \label{eq: first limit thm alternative}
        \sqrt{n}\left( T_{\beta,n}^p - T_\beta^p \right) \rightsquigarrow\frac{1}{m} \sum_{i=1}^m  \int_{\beta}^{1-\beta}   p  |\Delta_i(t)|^{p-1} \sgn(\Delta_i(t)) \mathbb{G}_i(t)\diff t.
    \end{equation}
    For $p = 1$, we obtain
    \begin{equation} \label{eq: second limit thm alternative}
    \begin{split}
    \sqrt{n}\left(T_{\beta,n}^1-T_\beta^1\right)
    \rightsquigarrow {}
    \frac{1}{m}\sum_{i=1}^m
    \int_\beta^{1-\beta} &
    \left[\mathbb{G}_i(t)-\psi_{\Uinv}^1(\mathbb{G})(t)\right]
    \sgn(\Delta_i(t))
    \mathbbm{1} \!\left(\abs{\Delta_i(t)}>0\right)
    \\
    &
    +\abs{\mathbb{G}_i(t)-\psi_{\Uinv}^1(\mathbb{G})(t)}
    \mathbbm{1} \!\left(\Delta_i(t)=0\right)\diff t .
    \end{split}
    \end{equation}
        \item If \autoref{cond 1.3} holds with $\theta = 1$, then~\eqref{eq: second limit thm alternative} and~\eqref{eq: first limit thm alternative} hold for $\beta = 0$.
        \end{enumerate}
\end{theorem}
\autoref{theorem: Convergence under the alt} is proven in \autoref{proof:thm:convergence-alt}.
\begin{remark}[Normality]\label{cor: variance of limiting distr}
    For $p \geq2$, the limiting distribution in \autoref{theorem: Convergence under the alt} is a linear transformation of $\mathbb{G}$, and as such it is normally distributed with mean zero. Define for each $i=1, \ldots, m$ the functions $q_i:(\beta,1-\beta) \to \mathbb{R}, \, t \mapsto p \abs{\Delta_i(t)}^{p-1} \sgn(\Delta_i(t))$. Then the variance of the limiting distribution is given by
    \begin{equation*}
        \sigma_{\beta,p}^2 =\frac{1}{m^2} \sum_{i,j=1}^m  \int_{\beta}^{1-\beta} \int_\beta^{1-\beta} q_i(s)q_j(t) \gammaG{i,j}(s,t) \diff s \diff t.
    \end{equation*}
    where $\gammaG{i,j}$ is the covariance operator of $\mathbb{G}$ defined in \autoref{theorem: Convergence under the Null}.
\end{remark}
\begin{remark} Below we state extensions and implications of \autoref{theorem: Convergence under the alt}.  
    \begin{enumerate}
        \item{(Degenerate limit).} For $p \geq 2$, the limiting distribution provided in \autoref{theorem: Convergence under the alt} is degenerate when $\Delta_i(t) = 0$ (Lebesgue) almost everywhere on $[\beta,1-\beta]$ for every $i =1, \ldots,m$. This corresponds precisely to the case that $T_\beta^p = 0$. When $p=1$, due to the fact that the convergence rates from \autoref{theorem: Convergence under the Null} and \autoref{theorem: Convergence under the alt} are equal in this case, the limit distribution derived in \autoref{theorem: Convergence under the alt} agrees with the limit distribution derived in \autoref{theorem: Convergence under the Null} when $H_0^{(\beta)}:T_\beta^p = 0$ is satisfied.
        \item{(Testing significant differences).} Similarly to the methods used by \citet{munk1998nonparametric}, it is possible to use the result from \autoref{theorem: Convergence under the alt} in order to test for significant differences $H^\prime_0: T_\beta^p \leq \delta$ as well as for hypotheses of the form $H_0^{\prime \prime}: T_\beta^p > \delta$, for some arbitrary $\delta > 0$. To this end, note that the plug-in estimator $\widehat{\sigma}_n$ for the standard deviation of the limit distribution from \autoref{theorem: Convergence under the alt} given in \autoref{cor: variance of limiting distr} is consistent, which by Slutzky's Theorem implies that $\sqrt{n}(T_{\beta,n}^p - T_\beta^p)/\widehat{\sigma}_n \rightsquigarrow \mathcal{N}(0,1)$. Thus, we may reject $H_0^\prime$ if $\sqrt{n}(T_{\beta,n}^p - \delta)/\widehat{\sigma}_n > z_{1-\alpha}$, where $z_{1-\alpha}$ is the $(1\!-\!\alpha)$-quantile of the standard normal distribution. Testing for $H_0^\prime$ instead of $H_0$ has the potential upside that in practical applications, it may sometimes be of advantage to not only identify isomorphic relationships but also mm-spaces that are close to being isomorphic. For example, similar proteins exhibit similar functionalities. Thus, something is to be gained by finding similar groups of proteins up to threshold $\delta$ instead of only considering isomorphic ones. 
        \item{(Different sample sizes).} Note that \autoref{theorem: Convergence under the alt} may be generalized to the case of different sample sizes analogously as in \autoref{remark: extensions Theorem under null}.
        \end{enumerate}
\end{remark}
\subsection{Local Power Analysis}
Previously, we have stated the behavior of the SLB barycenter statistic both under the null hypothesis and the alternative. In the following, we provide further results for distributional convergence under so-called \textit{local alternatives}. To this end, we slightly adjust the framework from \autoref{subsection: the proposed approach}. 
\begin{setting}\label{setting: local alternatives}
For $i = 1, \ldots ,m$, consider compact metric spaces $(\mathcal{X}_i,d_i)$ and probability measures $\mu_i$ such that $\mathrm{supp}(\mu_i) \subseteq \mathcal{X}_i$ and such that their support is compact. Furthermore, for $i = 1, \ldots,m$, we let $(\nu_{i,n})_{n \in \mathbb{N}}$ be a sequence of probability measures that are compactly supported on a subset of $\mathcal{X}_i$ and let $\rho_i$ be a $\sigma$-finite measure dominating $\mu_i + \nu_{i,n}$ for every $n \in \mathbb{N}$. For each $i = 1, \ldots,m$ and $n \in \N$ let $X_{i,1}^n, \ldots, X_{i,n}^n \iid \nu_{i,n}$ and assume, for simplicity, that there are no cross-dependencies between samples from $\nu_{i,n}$ and $\nu_{j,n}$ for each $j \neq i$. The corresponding empirical distance distributions are denoted
\begin{equation*}
    V_{i,n}(t) = \frac{2}{n(n-1)} \sum_{1 \leq k < l \leq n} \mathbbm{1} \!\left( d_i(X_{i,k}^n, X_{i,l}^n) \leq t \right).
\end{equation*}
The corresponding quantile function will be denoted $V_{i,n}^{-1}$, and for each $i=1, \ldots, m$, let 
\begin{equation} \label{eq: RN derivatives}
g_{i} = \frac{\de \mu_i}{\de \rho_i}, \qquad g_{i,n} = \frac{\de \nu_{i,n}}{\de \rho_i}    
\end{equation} 
the Radon-Nikodym derivatives with respect to the dominating measure $\rho_i$.
\end{setting} 
The question we are concerned with in this section then can be posed as follows: If the distance distributions $\mu^{U_i}$ induced by $\mu_i$, $1 \leq i \leq m$, satisfy the null hypothesis $H_0$ in~\eqref{eq: nullhypothesis}, and the sequences of measure $\nu_{i,n}$ come \emph{asymptotically closer} to $\mu_i$ for each $i = 1, \ldots, m$, then how does the limit distribution of 
\begin{equation*}
    n^{p/2}\widetilde{T}_{\beta,n}^p \coloneqq \frac{n^{p/2}}{m} \sum_{i=1}^m \int_\beta^{1-\beta} \abs{V_{i,n}^{-1}(t) - \Lambda_p \left(V_{1,n}^{-1}(t), \ldots, V_{m,n}^{-1}(t) \right)}^p \diff t
\end{equation*}
behave. 
\begin{theorem}\label{theorem: local alternatives}
    Assume \autoref{setting: local alternatives} and let $\beta >0$. Moreover, let $\mu^{U_1} = \ldots = \mu^{U_m}$ and assume that their corresponding distribution functions $U_1, \ldots, U_m$ satisfy \autoref{cond 1.2}. If for each $i = 1, \ldots,m$, there is some measurable $h_i: \mathcal{X}_i \to \mathbb{R}$ such that
    \begin{equation} \label{eq: hellinger derivative}
        \int_{\mathcal{X}_i} \left[ \sqrt{n} \left( \sqrt{g_{i,n}(x)} - \sqrt{g_{i}(x)} \right) - \frac{1}{2}h_i(x) \sqrt{g_{i}(x)} \right]^2 \diff \rho_i(x) \overset{n \to \infty}{\longrightarrow} 0,
    \end{equation}
    where $g_{i}$ and $g_{i,n}$ are defined in~\eqref{eq: RN derivatives}, then it holds that
    \begin{equation*}
        n^{p/2}\widetilde{T}_{\beta,n}^p \rightsquigarrow \frac{1}{m} \sum_{i=1}^m \int_\beta^{1-\beta} \abs{\mathbb{G}_i(t) + 2 D_i(t) - \Lambda_p\left( \mathbb{G}_1(t) + 2 D_1(t), \ldots, \mathbb{G}_m(t) + 2 D_m(t) \right)}^p \diff t.
    \end{equation*}
    Here, $\mathbb{G}$ is the centered $m$-dimensional Gaussian process from \autoref{theorem: Convergence under the Null} and $D_i$ is a drift term given by $D_i(t) = - \widetilde{D}_i\left( U_i^{-1}(t)\right)/u_i \! \left(U_i^{-1}(t) \right)$, with 
    \begin{equation*}
        \widetilde{D}_i(t) = \int_{\mathcal{X}_i} \int_{\mathcal{X}_i} \mathbbm{1} \!\left(d_i(x,y) \leq t\right)h_i(y) \diff \mu_i(x) \diff \mu_i(y).
    \end{equation*}
\end{theorem}
The proof of \autoref{theorem: local alternatives} is presented in \autoref{proof:thm:local-alternatives}. Note that~\eqref{eq: hellinger derivative} is the standard condition for contiguous local alternatives in the context of empirical processes \citep[Lemma 3.11.11]{wellner2013weak} and may be viewed as a sequential version of differentiability in quadratic mean (see e.g.\ \citet[eq.~(3.1)]{kosorok2008introduction}). In order to provide some intuition about \autoref{theorem: local alternatives}, we give an example in which we explicitly calculate the drift terms.
\begin{example}
    Assume \autoref{setting: local alternatives}. For each $i = 1, \ldots,m$, let $\nu_i$ be a probability measure on $\mathcal{X}_i$ and set $\nu_{i,n} = \left(1-\frac{1}{\sqrt{n}} \right) \mu_i + \frac{1}{\sqrt{n}} \nu_i$, i.e.\ drawing a sample from $\nu_{i,n}$ amounts to drawing a sample from $\nu_i$ with probability $1/\sqrt{n}$ and otherwise drawing from $\mu_i$. Thus, the empirical distance distribution $V_{i,n}$ from \autoref{setting: local alternatives} can be separated into three parts: distances induced by two random variables drawn from $\mu_i$ or two random variables drawn from $\nu_i$, and distances between a random variable drawn from $\mu_i$ and one drawn from $\nu_i$. If additionally $\nu_i \ll \mu_i$ for each $i = 1, \ldots, m$, equation~\eqref{eq: hellinger derivative} holds with $\rho_i = \mu_i$ and $h_i(t) = \frac{\de \nu_i}{\de \mu_i} - 1$. Then $\widetilde{D}_i$ takes the form $\widetilde{D}_i(t) = \mathbb{P}_{X \sim \mu_i, Y \sim \nu_i} \left( d_i(X,Y) \leq t \right) - U_i(t)$. Thus, the following picture emerges: The randomness in the limiting distribution stems only from the samples drawn from $\mu_i$. The drift is influenced by cross-distances of samples from $\mu_i$ and $\nu_i$. And finally, distances between samples drawn from $\nu_i$ become asymptotically negligible.
\end{example}
\section{Statistical Methodology using SLB Barycenters} \label{sec: statistical methodology}
Based on theory developed in \autoref{section: theory: limit distributions} and bootstrap consistency results, a variety of statistical inference tools can be derived. In the following, we focus on the statistical testing for $H_0$ (see~\eqref{eq: nullhypothesis}) and a SLB barycenter based classification method.
\subsection{Bootstrap Test for $\mathbf{H}_0$}\label{subsec: bootstrap}
Consider the limit distribution provided in \autoref{theorem: Convergence under the Null}. In order to carry out an asymptotic $\alpha$-level test for the hypothesis $H_0$ using this result, one would require to calculate the quantiles of this limiting distribution. As in applications, the underlying distance distributions of the mm-spaces which are to be compared are usually unknown, a direct calculation of the limit becomes impossible. Instead, a simple $n$-out-of-$n$ bootstrap can be applied. Recall our statistical framework from \autoref{subsection: the proposed approach} and let $X_i^\ast = (X_{1,i}^\ast, \ldots, X_{m,i}^\ast ) $ for $i = 1, \ldots ,n$ be an i.i.d.\ sample of $\widehat{\mu}_n = \frac{1}{n} \sum_{j=1}^n \delta_{(X_{1,j}, \ldots , X_{m,j})}$. Then, for $k = 1, \ldots m$, define the bootstrap analogues of $U_{k,n}$ via
\begin{equation*}
    U_{k,n}^\ast(t) =  \frac{2}{n(n-1)} \sum_{i<j} \mathbbm{1} \! \left( d_k\left(X_{k,i}^\ast, X_{k,j}^\ast \right) \leq t \right),
\end{equation*}
and denote its generalized inverse $U_{i,n}^{\ast, -1}$. The corresponding bootstrap quantile process is denoted by $\mathbb{U}_{i,n}^{\ast, -1}(t) = \sqrt{n}\left( U_{i,n}^{\ast,-1}(t) - U_{i,n}^{-1}(t)\right)$ for $t \in (\beta,1-\beta)$. The proposed bootstrap statistic takes the form
\begin{equation*}
    T_{\beta,n}^{p, \ast} = \frac{1}{m} \sum_{i=1}^m \int_\beta^{1-\beta} \left| \mathbb{U}_{i,n}^{\ast, -1}(t) - \Lambda_p\left( \mathbb{U}_{1,n}^{\ast,-1}(t), \ldots , \mathbb{U}_{m,n}^{\ast, -1}(t) \right) \right|^p \diff t.
\end{equation*}
\begin{theorem}\label{theorem: bootstrap consistency}
    Assume the bootstrap framework described above and denote $\zeta_{1,n}, \ldots, \zeta_{B,n}$ an iid sample from $T_{\beta,n}^{p,\ast}$, given $X_1, \ldots, X_n$. Further, denote $\zeta_{1-\alpha}^{B,n}$ the $(1\!-\!\alpha)$-quantile of the bootstrap replicates. Then, for $\beta > 0$ and if \autoref{cond 1.2} and $H_0^{(\beta)} : T_\beta^p =0$ are satisfied, 
    \begin{equation*}
                \limsup_{B,n \to \infty}\mathbb{P}\left(n^{p/2}T_{\beta,n}^p > \zeta_{1-\alpha}^{B,n}\right) \leq \alpha.
    \end{equation*}
\end{theorem}
The proof of the above theorem is given in \autoref{appendix: bootstrap}.
\begin{remark}
\begin{enumerate}
    \item{(Computation).} It is easy to see that carrying out the bootstrap test takes $O(B \cdot m \cdot n^2 \log(n))$ calculation operations.
    \item{(Post-hoc testing).} After rejecting $H_0$, it is furthermore possible to detect pairs of
non-isomorphic mm-spaces using an analogue of Tukey's honestly significant
difference (HSD) test \citep[chap.~2]{miller1981simultaneous}. As the maximum is a continuous
map, an adjustment of the proof of \autoref{theorem: Convergence under the Null} yields
\begin{equation}\label{eq: tukey limit}
\max_{i<j} n^{p/2}
\int_{\beta}^{1-\beta}
\left| U_{i,n}^{-1}(t) - U_{j,n}^{-1}(t) \right|^p \diff t
\rightsquigarrow
\max_{i<j}
\int_{\beta}^{1-\beta}
\left| \mathbb{G}_i(t) - \mathbb{G}_j(t) \right|^p \diff t .
\end{equation}
for $\beta \geq 0$ under equivalent conditions as in
\autoref{theorem: Convergence under the Null}. Writing $t_{1-\alpha}$ for the $(1\!-\!\alpha)$
quantile of the limiting distribution given in~\eqref{eq: tukey limit}, we may thus reject
$H_{0}^{i,j} : \mathcal{SLB}_p(\mathcal{X}_i, \mathcal{X}_j) = 0 $ if
\begin{equation*}
n^{p/2}
\int_{\beta}^{1-\beta}
\left| U_{i,n}^{-1}(t) - U_{j,n}^{-1}(t) \right|^p \diff t
> t_{1-\alpha}.
\end{equation*}
Note that the quantile $t_{1-\alpha}$ may once again be consistently estimated
employing a bootstrap similarly as in \autoref{theorem: bootstrap consistency}. Of course, other post-hoc
testing techniques such as Bonferroni correction \citep[chap.~3]{miller1981simultaneous} or FDR adjustments \citep[chap.~15]{efron2016computer} can also be applied in this context.
\end{enumerate}
\end{remark}
\subsection{Second Lower Bound Barycenter Classifier}\label{subsec: barycenter classifier}
Let $\mathcal{G}_i = \left\{ \left(\mathcal{X}_1^{(i)} , d_1^{(i)}, \mu_1^{(i)}\right), \ldots, \left(\mathcal{X}_{m_i}^{(i)} , d_{m_i}^{(i)}, \mu_{m_i}^{(i)}\right) \right\},\, i=1, \ldots K$ be groups of mm-spaces and denote by $n_{i,j}$ the number of samples drawn from the $j$-th mm-space in the $i$-th group. Further, let $X_{j,1}^{(i)}, \ldots, X_{j,n_{i,j}}^{(i)} \iid \mu_j^{(i)}$. Based on these samples, let $U_{(i,j),n_{i,j}}^{-1}$ be the empirical quantile function of the distances $\left\{ d_j^{(i)}\left( X_{j,k}^{(i)}, X_{j,l}^{(i)} \right) \right\}_{1 \leq k <l \leq n_{i,j}}$ and $U_{(i,j)}^{-1}$ the underlying quantile function of the distance distribution of the $j$-th mm-space from group $\mathcal{G}_i$. The goal will be to classify a new, unlabeled mm-space $(\mathcal{Y}, d_\mathcal{Y}, \nu)$ into one of the groups $\mathcal{G}_1, \ldots, \mathcal{G}_K$ upon observing samples $Y_1, \ldots, Y_{n_V} \iid \nu$. Denote the corresponding empirical quantile function of the distances $\left\{d_\mathcal{Y}(Y_i,Y_j) \right\}_{1 \leq i < j \leq n_V}$ by $V_{n_V}^{-1}$ and the underlying quantile function of the distance distribution by $V^{-1}$. We propose to classify the new mm-space based on the (trimmed) Wasserstein distances between $V^{-1}$ and the SLB barycenter of each group, i.e.
\begin{equation*}
    \gamma_i(\beta,p) = \int_\beta^{1-\beta} \abs{V^{-1}(t) - \Lambda_p\left( U_{(i,1)}^{-1}(t), \ldots, U_{(i,m_i)}^{-1}(t) \right)}^p \diff t.
\end{equation*}
After specifying $\beta$ and $p$, we will simply write $\gamma_i$ in order to ease notation. Write $\widehat{\gamma}_{i,n}$ for the empirical version of $\gamma_i$. By methods employed in the proof of \autoref{theorem: finite sample bound the alternative}, we may provide a risk bound for this classifier.
\begin{theorem}\label{theorem: barycenter classifier}
    Assume the above setting and let $\beta \in [0,1/2)$ and $p \geq 1$. Further, denote by $i^\star$ the index such that $\gamma_{i^\star} < \gamma_j$ for each $j \neq i^\star$. Moreover, let $\widehat{i}^\star = \argmin_{i = 1, \ldots, K} \widehat{\gamma}_{i,n}$ be the empirical SLB barycenter classifier. Letting $D = \max_{i = 1, \ldots,K} \max_{j = 1, \ldots, m_i} \mathrm{diam}\left(\mathcal{X}_j^{(i)} \right) \vee \mathrm{diam}(\mathcal{Y})$, we then obtain
    \begin{equation*}
        \mathbb{P} \left( \widehat{i}^\star \neq i^\star \right) \leq 2^{p+1/2}pD^p \sum_{i \neq i^\star} \frac{ 2n_V^{-1/2} + \sum_{k=1}^{m_{i^\star}} n_{i^\star,k}^{-1/2} + \sum_{k=1}^{m_i} n_{i,k}^{-1/2} }{\gamma_i - \gamma_{i^\star}}.
    \end{equation*}
\end{theorem}
For a proof of the theorem, see \autoref{proof:thm:barycenter-classifier}.\newline 
Note that if all samples are balanced, i.e.\ of the same order $n$, the classification error converges to zero at the rate $O(n^{-1/2})$.
\begin{remark}\label{remark: barycenter classifier} We briefly remark on the appeal of using SLB barycenters for classification purposes and potential problems and remedies of the classifier. \begin{enumerate}
    \item Using the (surrogate) SLB barycenters for classification purposes has the advantage that each class can be assigned a mathematically meaningful representative distribution to serve as a reference. For example, this means that new mm-spaces only have to be compared to $K$ distributions instead of $\sum_{i=1}^K m_i$, making it computationally more tractable when the overall number of mm-spaces is high.
    \item In practical situations, it may be the case that, although the groups are well separated in terms of their distance distributions, the condition $\gamma_{i^\star} < \gamma_j$ for each $j \neq i^\star$ from \autoref{theorem: barycenter classifier} is not satisfied. This may, for instance, be caused by a difference in the variation of the distance distributions within each group. A possible remedy for this, we suggested in \autoref{subsec:applications}: Apply a preclustering to each of the groups $\mathcal{G}_i$, $i = 1, \ldots,K$, e.g.\ by a Wasserstein $k$-means clustering on their distance distributions, resulting in smaller groups $\mathcal{G}_{i,j}$ with $i =1, \ldots,K$ and $j = 1, \ldots k_i$. The number of clusters selected for each group may be chosen after a visual inspection of their distance distributions or tuned on a training dataset. Then, the Wasserstein distances $\widehat{\gamma}_{i,j}$ between the distance distributions of the unlabeled space and the barycenters of the groups $\mathcal{G}_{i,j}$ can be calculated and the predicted label may be chosen as $\widehat{i}^\star = \argmin_{i=1, \ldots,K} \min_{j = 1, \ldots, k_i} \widehat{\gamma}_{i,j}$. Note that the $k$-means preclustering may not only serve to stabilize the SLB barycenter classifier, but also to identify potential subgroups that have been previously undiscovered, see \autoref{sec: Simulations}.
\end{enumerate}
\end{remark}

\section{Simulations}\label{sec: Simulations}
In the following, we provide simulations for the limiting distributions stated in the previous
sections. We furthermore apply the SLB barycenter test and classifier on synthetic data.
\subsection{Convergence under the Null}\label{subsec:sim:NH}
We proceed with simulations for the convergence under the null hypothesis $H_0$ shown in \autoref{theorem: Convergence under the Null}. In order to have a clear reference for our simulations, we thus provide an example in which the distance distributions of the mm-spaces are known explicitly and we can thus numerically calculate the covariance kernel $\gammaG{i,j}$ in \autoref{theorem: Convergence under the Null} for any $i,j = 1, \ldots, m$.
\begin{example}\label{example: rate of covnergence}
Let $m \in N$ and for $i = 1, \ldots ,m$
\begin{equation*}
    (\mathcal{X}_i, d_i, \mu_i) \coloneqq \left([0,1]^2,d_\infty, U\left([0,1]^2 \right) \right)
\end{equation*}
where $d_\infty(x,y) = \max(|x_1 - y_1|, |x_2-y_2|)$ is the Chebyshev distance. Furthermore, we choose the independent coupling $\mu = \otimes_{i=1}^m \mu_i$. By elementary calculations, we find that in this example we can explicitly state the distribution functions of the corresponding distance distributions of the mm-spaces by $U_i(x) = (2x-x^2)^2$ for $x \in [0,1]$ and the corresponding quantile function $U_i^{-1}(t) = 1- \sqrt{1- t^{1/2}}$ for $t \in (0,1)$. This mm-space furthermore satisfies \autoref{cond 1.3} for any $\theta  = p$ with $1 \leq p \leq 2$.
\end{example}
\begin{figure}[t!] 
    \centering
    \includegraphics[width=1\textwidth]{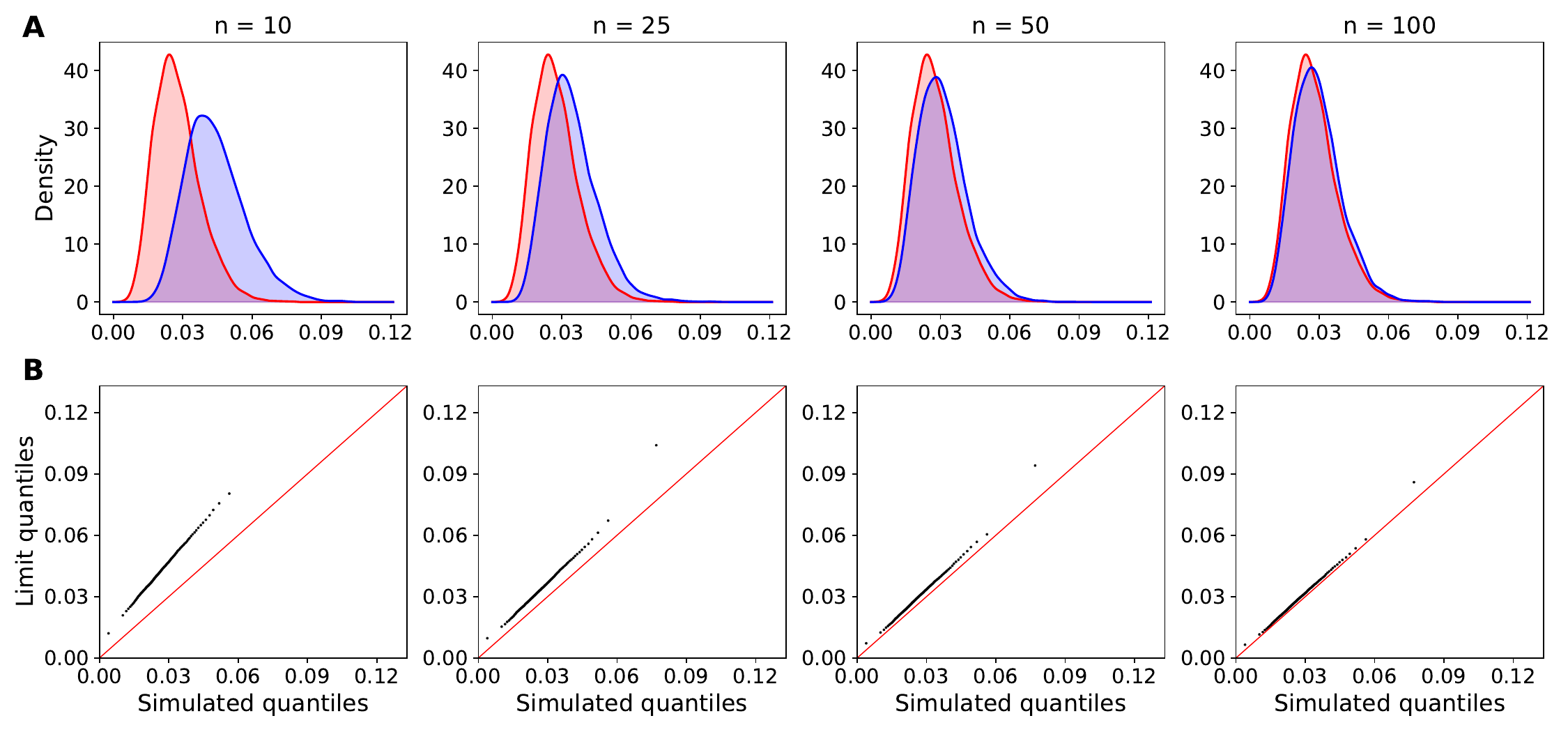} 
    \caption{\textbf{A}: Kernel density estimates of the limit distribution (red) and $n T_{0,n}^2$ (blue) for $n = 10,25,50,100$ in \autoref{example: rate of covnergence} with $m = 15$ based on $10,\!000$ realizations each. \textbf{B}: QQ-plots of the corresponding quantiles and red reference line.}\label{figure: convergence rate of test stat}
\end{figure}
Based on the setting of \autoref{example: rate of covnergence}, we simulated the limiting distribution in \autoref{theorem: Convergence under the Null} by taking Monte Carlo samples based on an explicit calculation of the covariance matrix induced by $\mathbb{G}$ (see \autoref{theorem: Convergence under the Null}) over a discretization of the interval $[\beta,1-\beta]$. In \autoref{figure: convergence rate of test stat}, kernel density estimates with Gaussian kernel and based on Silverman's rule of thumb for $n^{p/2} T_{\beta,n}^p$ are displayed in the setting of \autoref{example: rate of covnergence} and with $p=2,\beta = 0$ and $m=15$ for different sample sizes. These are then compared to the limiting distribution. We furthermore provide QQ-plots between the simulated and limiting distribution. Plots for other hyperparameters are provided in \autoref{appendix: simulations}, see \autoref{figure: Simulation Nullhypothese 1}, \autoref{figure: Simulation Nullhypothese 2} and \autoref{figure: Simulation Nullhypothese 3}.
\subsection{Bootstrap}\label{subsec:sim:bootstrap}
Next, we visually demonstrate the bootstrap consistency of $n^{p/2} T_{\beta,n}^{p,\star}$ proven in \autoref{theorem: bootstrap consistency}. In order to also cover the setting of possible dependencies between the mm-spaces (i.e. $\mu$ is not equal to the product measure in \autoref{subsection: the proposed approach}), we introduce the following example.
\begin{example}\label{example: bs convergence example}
Define $m$ different metric measure spaces via
\begin{equation} \label{eq: example mm spaces uniform euclidean}
    (\mathcal{X}_i, d_i, \mu_i) \coloneqq \left([0,1]^3,d_{\infty}, U\left([0,1]^3\right)\right)
\end{equation}
for $i = 1, \ldots,m$. Furthermore, let $K \in \mathbb{N}$. Then we split the set $[0,1]^3$ into $K^3$ different quadrants by letting
\begin{equation*}
    Q_{i,j,l} =  [i/K,(i+1)/K] \times [j/K, (j+1)/K] \times [l/K, (l+1)/K] 
\end{equation*}
for indices $0 \leq i,j,l \leq K-1$. Based on this definition, we define a coupling $\mu$ of the metric measure spaces via the density
\begin{equation*}
    f_{m,K}(b_1, \ldots, b_m) = \begin{cases}
        K^{3(m-1)} &\text{if } \exists i,j,l \in \mathbb{N}: b_k \in Q_{i,j,l} \text{ for }k = 1, \ldots, m, \\
        0 &\text{else},
    \end{cases}
\end{equation*}
for $b_1, \ldots, b_m \in [0,1]^3$. The distance distributions of the marginal spaces are given by $U_i(x) = \left(2x-x^2\right)^3$ for $x \in [0,1]$ and thus satisfy \autoref{cond 1.2} for each $\beta >0$.
\end{example}
\begin{figure}[t!] 
    \centering
    \includegraphics[width=1\textwidth]{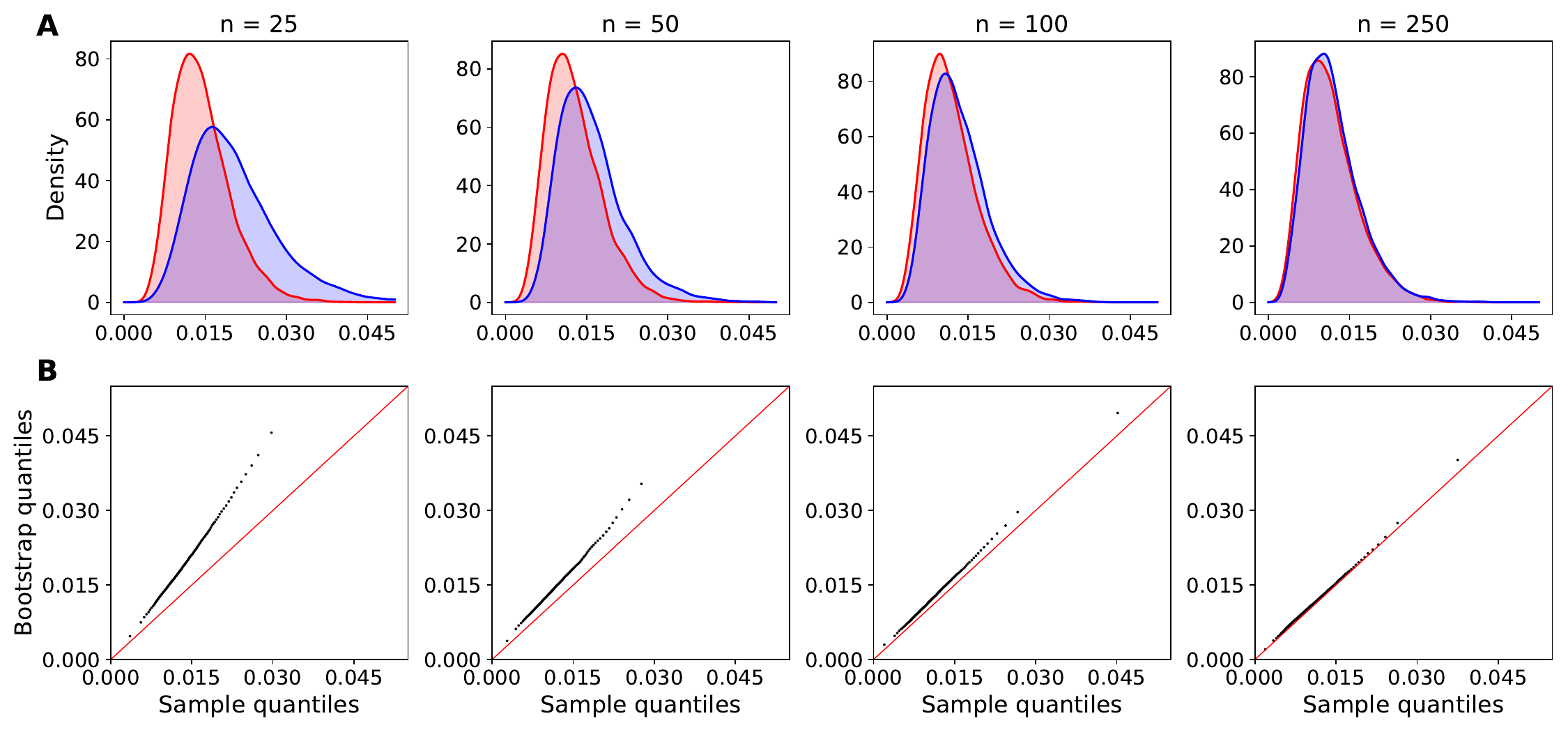} 
    \caption{\textbf{A}: Kernel density estimates of $n T_{0.01,n}^2$ (red) and $n T_{0.01,n}^{2,\ast}$ (blue) for $n = 25,50,100,250$ in \autoref{example: bs convergence example} with $m = 10$ and $K = 3$ based on $1,\!000$ samples from $n T_{0.01,n}^2$ and 100 bootstrap realizations of each sample. \textbf{B}: QQ-plots of the corresponding quantiles.}\label{figure: bs convergence rate}
\end{figure}
By definition, a sample from this distribution always fulfills the property that marginally, every set of coordinates $(x_k,y_k,z_k)$ for $k=1, \ldots,m$ lies in the same quadrant $Q_{i,j,l}$ for some $0 \leq i ,j,l \leq K-1$. For $K = 1$, this corresponds to independence, while $K \to \infty$ means that the dependence increases and the marginal coordinates of the sample become equal. Kernel density estimates and QQ-plots of the bootstrap statistic $n T_{\beta,n}^{p,\star}$ and $n T_{\beta,n}^p$ using the parameters $K = 3, m=10, p=2$ and $\beta = 0.01$ are shown in \autoref{figure: bs convergence rate}. It can be seen that for low sample sizes, the bootstrap distribution tends to be a conservative approximation of the limit distribution. Further plots for different hyperparameters are presented in \autoref{appendix: simulations} (see \autoref{figure: Simulation bootstrap 1} and \autoref{figure: Simulation bootstrap 2}).
\subsection{Convergence under the Alternative}\label{subsec:sim:alternative}
In order to simulate the convergence result under the alternative stated in \autoref{theorem: Convergence under the alt}, we slightly adjust the setting of \autoref{example: rate of covnergence}.
\begin{example}\label{example: convergence alternative}
Let $m \geq 2$. Then for each $i = 1, \ldots, m$, set $K_i = 1 + 5(i-1)/(m-1)$
\begin{equation*}
    (\mathcal{X}_i, d_i, \mu_i) = \left( [0,K_i]^2, d_\infty, U\left([0,K_i]^2 \right)\right).
\end{equation*}
Hence, the mm-spaces are an interpolation between $(\mathcal{X}_1, d_1, \mu_1) = \left( [0,1]^2, d_\infty, U\left([0,1]^2 \right)\right)$ (recall \autoref{example: rate of covnergence}) and $(\mathcal{X}_m, d_m, \mu_m) = \left( [0,6]^2, d_\infty, U\left([0,6]^2 \right)\right)$. Thus, for each $i = 1, \ldots, m$, we have $U_i^{-1}(x) = K_i\left(1- \sqrt{1-\sqrt{x}}\right)$ for $x \in (0,1)$, which satisfies \autoref{cond 1.3} with $\theta = 1 $.
\end{example}
\begin{figure}[t!] 
    \centering
    \includegraphics[width=1\textwidth]{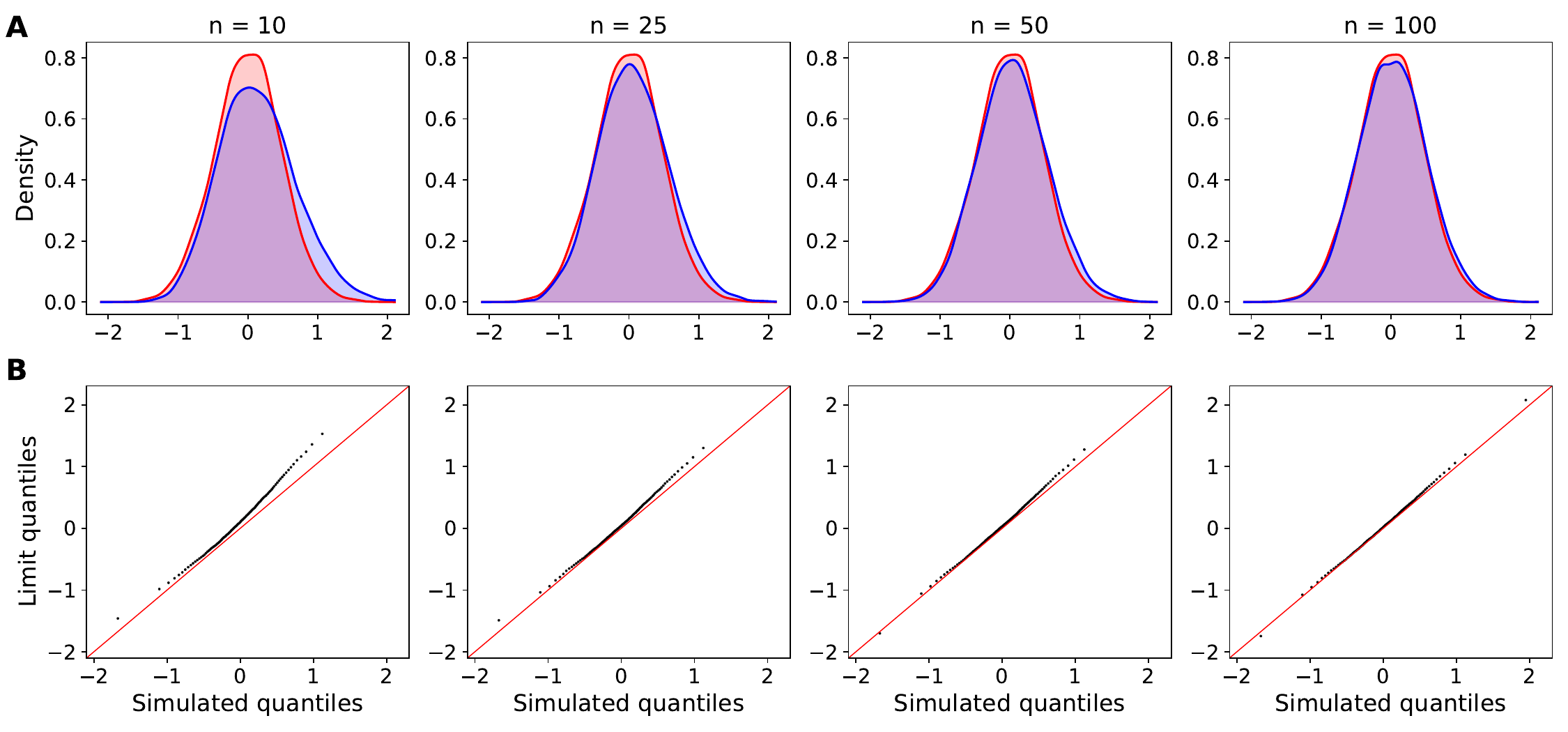} 
    \caption{\textbf{A}: Kernel density estimates of the limit distribution (red) and $\sqrt{n} \left( T_{0,n}^2  - T_{0}^2\right)$ (blue) for $n = 10,25,50,100$ in \autoref{example: convergence alternative} with $m = 6$ based on $10,\!000$ realizations each. \textbf{B}: QQ-plots of the corresponding quantiles and red dotted reference line.}\label{figure: convergence rate of test stat alternative}
\end{figure}
Similarly to the simulations under the null hypothesis, the limit distribution in \autoref{example: convergence alternative} can be simulated via Monte Carlo samples based on discretizing the integral in the limit. \autoref{figure: convergence rate of test stat alternative} shows kernel density estimates and QQ-plots for $m = 6, p = 2$ and $\beta = 0$. Plots for different hyperparameters are presented in \autoref{figure: Simulation alternative 1} and \autoref{figure: Simulation alternative 2} in \autoref{appendix: simulations}.
\subsection{SLB Barycenter Test}\label{subsec:sim:test}
We continue by applying the test $\Phi_{\beta,p}^\ast$ (see \autoref{theorem: bootstrap consistency}) to synthetic data. To this end, we use the publicly available triangulated objects dataset \citep{sumner2004deformation}. In the database, there are meshes corresponding to different classes of objects such as horses, elephants, camels and more, recall \autoref{fig: Meshes pictures} for some examples. For our experimental setup, each mesh induces an mm-space $(\mathcal{X},d, \mu)$ by letting $\mathcal{X}$ be the set of $50,\!000$ randomly sampled points, uniform over the surface of the mesh, $\mu$ the uniform measure on $\mathcal{X}$ and $d$ to be the Euclidean distance. In particular, we applied the test to the comparison of mm-spaces $(\mathcal{X}_i,d_i, \mu_i)_{1 \leq i \leq 10}$ in the following scenarios:
\begin{enumerate}[label=Scenario \Alph*, leftmargin=*]
    \item The null hypothesis $H_0$ is satisfied. In particular, we choose each mm-space according to the same mesh corresponding to a camel.\label{item: scenario A}
    \item Most mm-spaces are isomorphic, but there is an outlier. For $i=1, \ldots 9$, we choose each mm-space according to the same mesh of a camel as in the first scenario. However, $(\mathcal{X}_{10},d_{10},\mu_{10})$ is chosen corresponding to a mesh of a horse.\label{item: scenario B}
    \item All mm-spaces are slightly different. Here, we chose $(\mathcal{X}_i,d_i, \mu_i)_{1 \leq i \leq 10}$ corresponding to the 10 first meshes in the \textit{camel-gallop} class. The meshes correspond to discrete time steps of an animation of the respective animal and consequently vary slightly between time points.\label{item: scenario C}
\end{enumerate}
\begin{figure}[t!]
    \centering
    \includegraphics[width=1\linewidth]{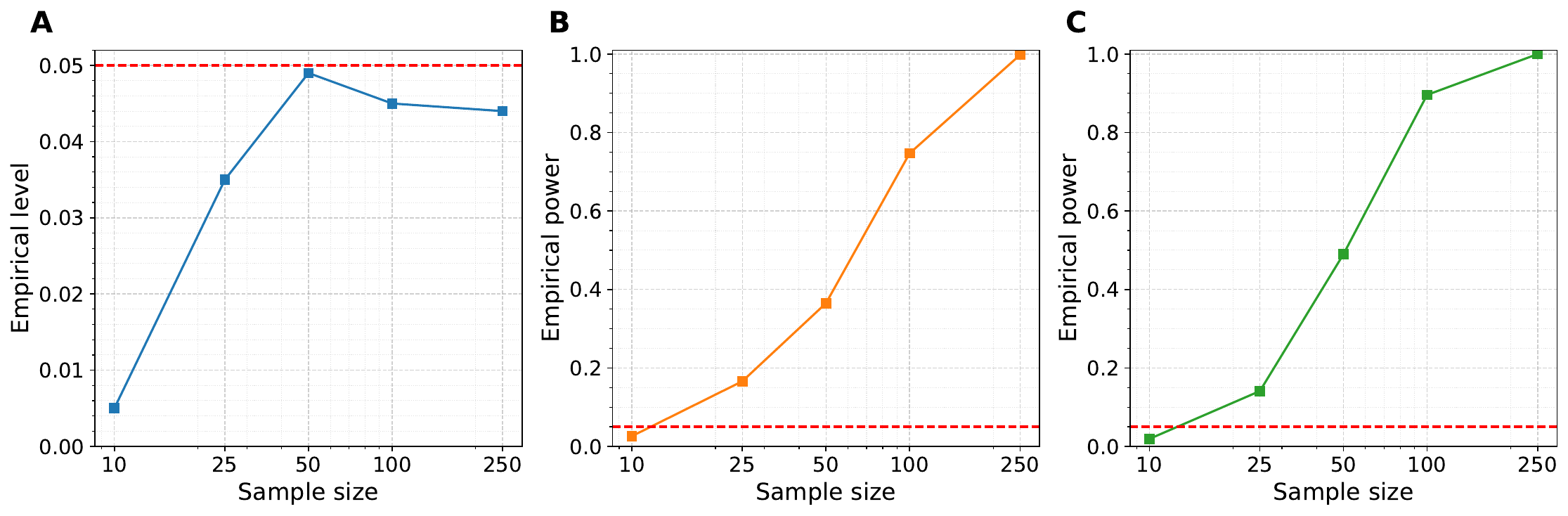}
    \caption{Empirical level/power of $\Phi_{\beta,p}^\ast$ in \hyperref[item: scenario A]{Scenarios~A--C} over $1,\!000$ test repetitions and using $\beta = 0.1$ and $p = 2$.}\label{fig: Meshes test results}
\end{figure}
The empirical rejection rates of $\Phi_{\beta,p}^\ast$ are plotted in \autoref{fig: Meshes test results}. It can be seen that under both types of alternatives, the test rejects $H_0$ reliably for sample sizes $n = 100$ or $n = 250$. In~\ref{item: scenario A}, i.e.\ when $H_0$ holds, the empirical level approaches the significance level $\alpha = 0.05$ as $n$ gets large. For low sample sizes, the empirical level stays below the significance level. For some more extensive simulations using different configurations for the parameters $p$ and $\beta$, we refer to \autoref{subsec:further-sim:test}.
\subsection{SLB Barycenter Classifier}\label{subsec:sim:classifier}
\begin{figure}[b!]
\centering
\includegraphics[width=\linewidth]{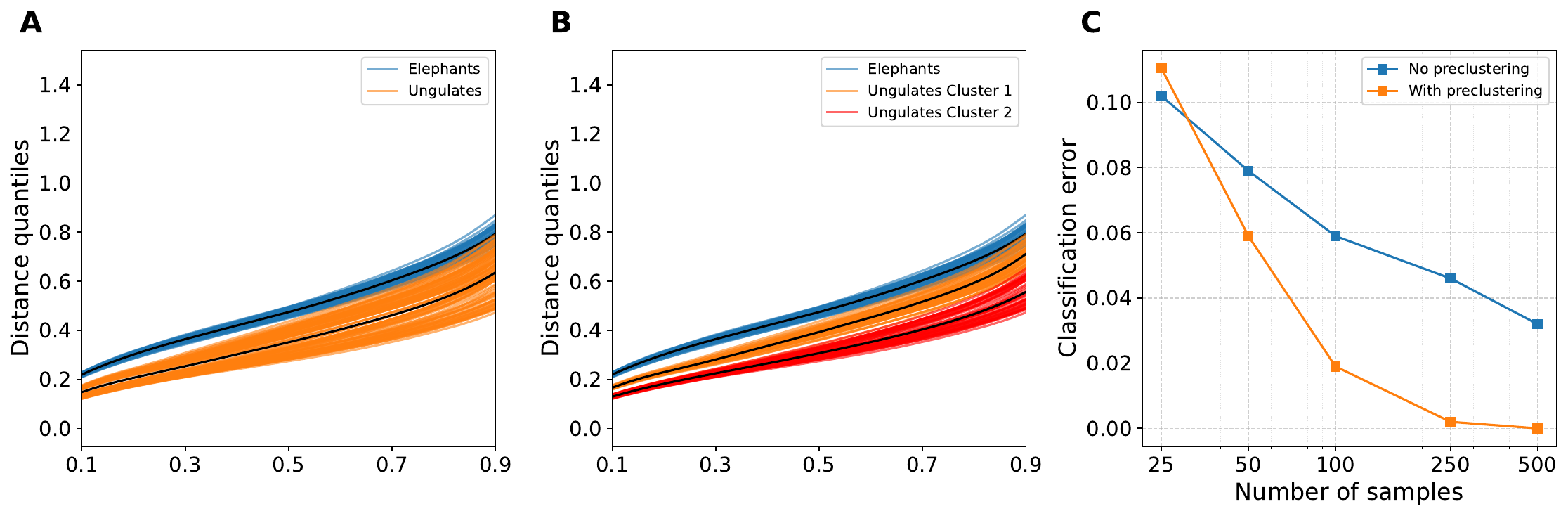}
\caption{\textbf{A:} Quantile functions of Euclidean distances between $n = 500$ randomly sampled points of meshes belonging to $\mathcal{G}_1$ (blue) and $\mathcal{G}_2$ (orange) and the respective barycenter of the quantile functions (black). \textbf{B:} The quantile functions of $\mathcal{G}_2$ have been clustered using $k$-means. The barycenters of the respective subgroups are plotted in black.  \textbf{C:} Performance of the SLB barycenter classifier with $\beta = 0.1$ and $p = 2$ for different sample sizes with and without the usage of $k$-means clustering.}\label{fig: classifier meshes performance and plots}
\end{figure}
We apply the SLB barycenter classifier from \autoref{theorem: barycenter classifier} to the triangulated objects database used in the previous paragraph. We use the classes \textit{camel-gallop}, \textit{horse-gallop} and \textit{elephant-gallop}, which each contain 48 meshes. To illustrate the potential of combining the SLB classifier with $k$-means clustering as commented in \autoref{remark: barycenter classifier}, we consider the following setup: Class $\mathcal{G}_1$ is the set of all mm-spaces corresponding to elephants and $\mathcal{G}_2$ to ungulates, i.e.\ camels and horses. Part A of \autoref{fig: classifier meshes performance and plots} shows a clear separation in the distance distributions of the classes, although $\mathcal{G}_2$ admits a notably higher variation in the distance quantile functions. The results of applying $k$-means clustering with $k = 2$ to $\mathcal{G}_2$ are shown in part B of \autoref{fig: classifier meshes performance and plots}. Notably, the $k$-means clustering recovers the original classes \textit{camel-gallop} and \textit{horse-gallop} almost perfectly. Next, we carry out the SLB classifier for $\mathcal{G}_1$ and $\mathcal{G}_2$ by performing leave-one-out classification, using different sample sizes $n$. In order to assess the impact of including the $k$-means preclustering in the procedure, we implement the SLB barycenter classifier using either no preclusters or clustering $\mathcal{G}_2$ into $k = 2$ groups before classifying the unlabeled space and compare the empirical results over $1,\!000$ repetitions. Part C of \autoref{fig: classifier meshes performance and plots} shows a significantly lower classification error when including the $k$-means, approaching an error of $0$ when $n = 500$.
\section{Further Details: Testing on Structural Protein Data}\label{sec:application} 
In this section, we provide the details about the application of the SLB barycenter test on structural protein data, which was summarized shortly in \autoref{subsec:applications}. 
\begin{figure}[b!]
\centering 
\includegraphics[width=0.66\linewidth]{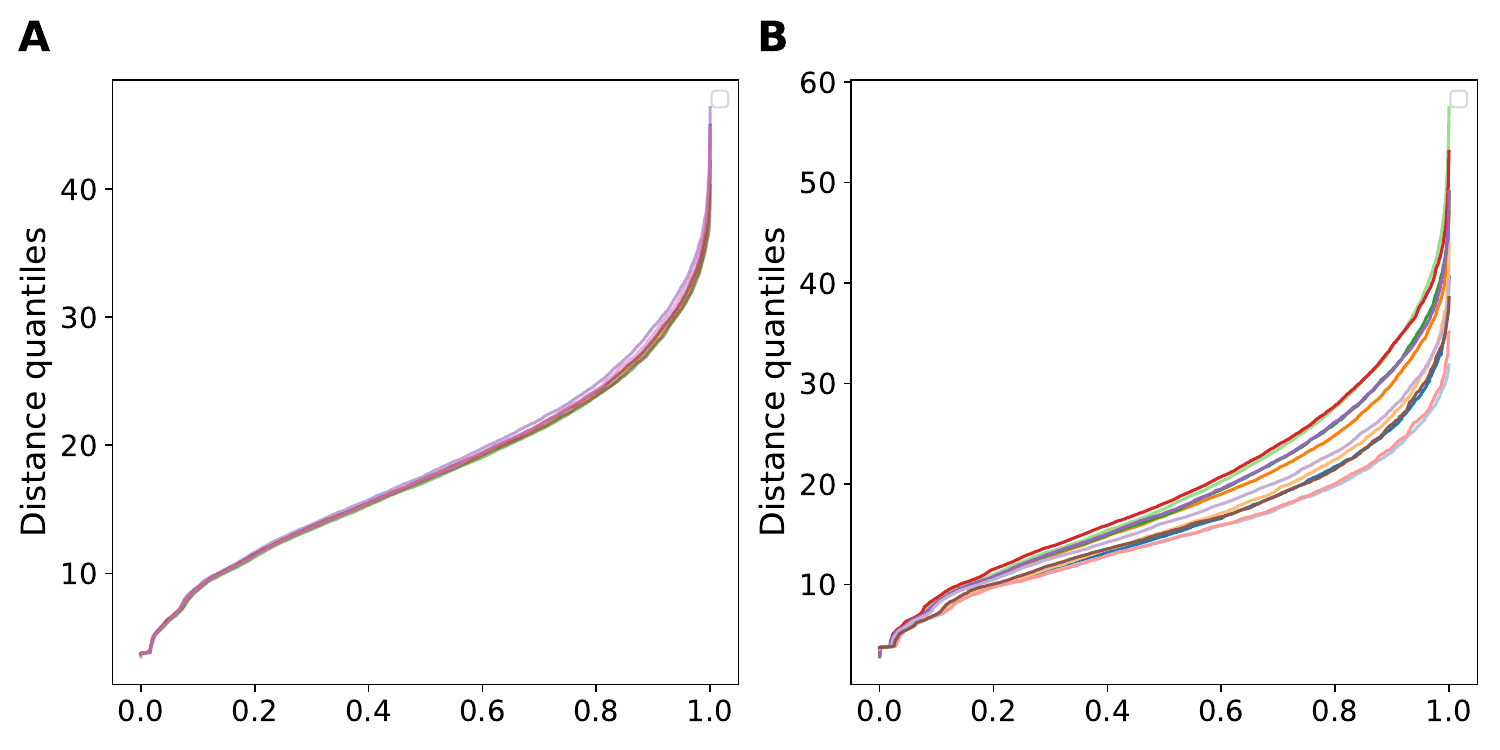}
\caption{\textbf{A:} Empirical quantile functions of Euclidean distances between C-alpha atoms of 7 Immunoglobulin heavy chain variable domains. \textbf{B:} Empirical quantile functions of Euclidean distances between C-alpha atoms of all SCOPe protein domains contained in the 95\%-identity filtered subset of family b.1.1.3.}\label{fig: distance profiles for test}
\end{figure}
\paragraph{Comparison of Immunoglobulin heavy chain variable domains} For the first application of the test $\Phi_{\beta,p}^\ast$, we compare seven Immunoglobulin heavy chain variable domains, which are protein domains responsible for binding antigen. These domains contain $117-123$ C-alpha atoms. As they are not only protein domains from the same SCOPe family (b.1.1.1), but even belong to a smaller biological subgroup of this family, we expect the domains to be structurally very similar. This intuition is reflected by the empirical quantile functions of the distances induced by the C-alpha atoms of the protein domains, see part A of \autoref{fig: distance profiles for test}. We carried out $\Phi_{\beta,p}^\ast$ in this example, using $\beta = 0.01$, $p = 2$, significance level $\alpha = 0.05$ and subsampling $n = 10,25,50,70,100$ of the C-alpha atoms. The empirical level of the test is documented in \autoref{table: power level proteins test}. It is calculated over $1,\!000$ repetitions of the test, for each of which $500$ bootstrap replications were calculated. The empirical level at no point exceeds the significance level, which shows that, in accordance with our prior knowledge about the proteins, there is no evidence against the null hypothesis for these domains.
\begin{table}[t!]
\centering
\caption{Empirical rejection rates of $\Phi_{\beta,p}^{\ast}$ by sample size.}\label{table: power level proteins test}

\begin{tabular}{lccccc}
\toprule
Experiment & $n=10$ & $n=25$ & $n=50$ & $n=70$ & $n=100$ \\
\midrule
A & 0.001 & 0.004 & 0.002 & 0 & 0 \\
B & 0.051 & 0.751 & 1 & 1 & --- \\
\bottomrule
\end{tabular}

\smallskip
\begin{minipage}{\linewidth}
\small
\noindent Notes: Empirical level is based on 1,000 repetitions of the test.
Experiment A compares immunoglobulin heavy-chain variable domains.
Experiment B uses domains from family b.1.1.3.
The symbol --- indicates that the experiment was not performed, since some of the protein domains consist of less than $n$ C-alpha atoms.
\end{minipage}
\end{table}
\paragraph{Comparison of SCOPe family b.1.1.3} Next, we apply the test to all protein domains contained in a representative subset of the SCOPe family b.1.1.3 (i.e.\ the $95\%$ sequence identity filtered subset on the SCOPe website). These domains contain $72-109$ C-alpha atoms and are part of the same superfamily as the Immunoglobulin heavy chain variable domains. In contrast to the previous example however, these domains are not filtered by any further biological subgroup. The empirical quantile functions corresponding to the distances of the C-alpha atoms of these domains are displayed in part B of \autoref{fig: distance profiles for test}. It is visible that the spread between these quantile functions is larger than in the previous example. The empirical power of $\Phi_{\beta,p}^\ast$ over $1,\!000$ repetitions of the test, each using $B = 500$ bootstrap replications, and with parameters $\beta = 0.01$, $\alpha = 0.05$ and $p = 2$ is reported in \autoref{table: power level proteins test}. Already at a small subsample size of $n = 25$, the empirical power of the test exceeds the significance level remarkably, while at $n = 50,70$, the power is one. Thus, the null hypothesis is rejected reliably. This shows the ability of the method to detect deviations in the domains to be compared, even when the protein domains are from the same SCOPe family and thus structurally related.
\section*{Declaration of Generative AI Use and Data Availability} Generative AI (GPT-5.6 Sol) was used in order to check the paper for grammatical and notational errors in the last screening before uploading. \newline
The data used in this paper is available from the SCOPe database at \url{https://scop.berkeley.edu/} and from the triangulated mesh dataset \url{https://people.csail.mit.edu/sumner/research/deftransfer/data.html}.
\bibliographystyle{asa-apalike.bst}
\bibliography{references}
\newpage
\appendix
\section*{Appendix}
\subsection*{Notation}
Throughout the appendix, we are going to use the following notation and conventions: We assume that the random elements $X_1, \ldots, X_n \sim \mu$ live on a common probability space $(\Omega, \mathcal{A}, P)$. Weak convergence in the sense of Hoffmann-J\o rgensen is denoted by $\rightsquigarrow$ \citep[chap.~1.3]{wellner2013weak}. For an arbitrary set $T$, we let $\ell^\infty(T)$ be the space of bounded functions $f: T \to \mathbb{R}$. Furthermore, $\norm{f}_T = \sup_{t \in T} \abs{f(t)}$ stands for the uniform norm. For arbitrary constants $C_{i,1}<C_{i,2},\, i = 1, \ldots,m$, the norm of the product space $\ell^\infty([C_{1,1}, C_{1,2}]) \times \ldots \times \ell^\infty([C_{m,1}, C_{m,2}])$ will be chosen and denoted by
\begin{equation} \label{eq: norm on product space}
    \norm{f}_{[C_1,C_2]} \coloneqq \sum_{i = 1}^m \norm{f_i}_{[C_{i,1}, C_{i,2}]}, \qquad f \in \ell^\infty([C_{1,1}, C_{1,2}]) \times \ldots \times \ell^\infty([C_{m,1}, C_{m,2}]).
\end{equation}
For $p \geq 1$ and probability measures $\mu,\nu \in \mathcal{P}(\mathbb{R})$ with quantile functions $F_\mu^{-1}, F_{\nu}^{-1}$, we write \begin{equation*}
\mathcal{W}_p(\mu,\nu) = \left( \int_0^1 \abs{F_\mu^{-1} (t) - F_\nu^{-1}(t)}^p \diff t \right)^{1/p}
\end{equation*}
for the $p$-Wasserstein distance between the measures, see \citet[chap.~2.2]{villani2021topics}. When $X \subset \mathbb{R}$ and $p \geq 1$, we denote by $\lambda$ the Lebesgue measure on $X$ and $L^p(X) = L^p(X, \mathcal{B}(X), \lambda)$ the space of measurable and $p$-integrable functions equipped with the norm
\begin{equation*}
    \norm{f}_{L^p(X)} = \left( \int_X \abs{f(x)}^p \diff \lambda(x) \right)^{1/p}
\end{equation*}
for any $f \in L^p(X)$. Moreover, for an integrable function $f$ on a set $X$, we often use the convention
\begin{equation*}
    \int_X f = \int_X f(x)\diff x.
\end{equation*}
For sets $A_1, A_2, \ldots \subset X$ for some arbitrary set $X$, write $A_n \nearrow A \subset X$, if the sequence is nested, i.e.\ $A_1 \subseteq A_2 \subseteq \ldots$, and if $A = \cup_{n \in \mathbb{N}} A_n$. Analogously, we write $A_n \searrow A$ if $A_1 \supseteq A_2 \supseteq \ldots$ and $\cap_{n \in \mathbb{N}} A_n = A$.
\section{Joint Weak Convergence of the Quantile and Barycenter Process}\label{section: JWC of DoD process}
Throughout the following, we assume the setting of \autoref{subsection: the proposed approach}. The goal of this section is to derive the joint weak convergence of the quantile processes $\mathbb{U}_{i,n}^{-1} = \sqrt{n}(U_{i,n}^{-1} - U_i^{-1})$, $1 \leq i \leq m$ and the barycenter process \begin{equation}\label{eq: barycenter process}
\lambda_{p,n} \coloneqq \sqrt{n} ( \Lambda_p(\Uninv) - \Lambda_p\left(\Uinv \right) ), 
\end{equation}as this is the pillar for any distributional convergence results given in \autoref{section: theory: limit distributions} and \autoref{sec: statistical methodology}. For $p \geq 2$, we will show (\autoref{lemma: properties barycenter}(a) ahead), that $\Lambda_p$ has derivative 
\begin{equation*}
    \psi_x^p: \mathbb{R}^m \to \mathbb{R}, \, y \mapsto \begin{cases}
        \frac{\sum_{i=1}^m |x_i - \Lambda_p(x)|^{p-2}y_i}{\sum_{i=1}^m |x_i - \Lambda_p(x)|^{p-2}}  &\text{if } \sum_{i=1}^m |x_i- \Lambda_p(x)|>0 \\
        \Lambda_p(y) &\text{else},
    \end{cases}
\end{equation*}
at each $x \in \mathbb{R}^m$ for which there exists some $i \neq j$ such that $x_i \neq x_j$. For $p = 1$, we need to introduce some further notation. For $l \in \lbrace 1, \ldots, m \rbrace$, let $\nu_l : \mathbb{R}^m \to \mathbb{Z}$ be defined as
\begin{equation*}
    \nu_l(x) \coloneqq \nu_{l,x} = l - \# \lbrace i \in \lbrace 1,\ldots, m \rbrace : x_i < x_{(l)} \rbrace,
\end{equation*}
i.e.\ the number of values in $x$ that are strictly smaller than its $l$-th order statistic subtracted from $l$. Moreover, we define $J_l: \mathbb{R}^m \to \mathcal{P}(\lbrace 1, \ldots, m \rbrace)$, where $\mathcal{P}(A)$ is the power set of $A$, via
\begin{equation*}
    J_l(x) \coloneqq J_{l,x} = \lbrace i \in \lbrace 1, \ldots, m \rbrace : x_i = x_{(l)} \rbrace,
\end{equation*}
i.e.\ all indices of $x$ that tie with its $l$-th order statistic. Letting $J \subseteq \lbrace 1, \ldots, m \rbrace$, we denote by $\pi^J : \mathbb{R}^m \to \mathbb{R}^{|J|}$ the coordinate projection of $x \in \mathbb{R}^m$ onto the coordinates $(x_i)_{i \in J}$. Based on all the above auxiliary maps, we can now define
\begin{equation}\label{eq: l-th order quantile process}
    \Psi_{l,x}: \mathbb{R}^m \to \mathbb{R}, \, y \mapsto \left( \pi^{J_{l,x}} (y) \right)_{(\nu_{l,x})}
\end{equation}
which is the $\nu_{l,x}$-th order statistic of $\pi^{J_{l,x}} (y)$. Note that if for some $x,y \in \mathbb{R}^m$, both have same ordering in their coordinates, i.e. $x_i < x_j$ if $y_i < y_j$, then $\Psi_{l,x}(y) = y_{(l)}$. Furthermore, it always holds that $\Psi_{l,x}(x) = x_{(l)}$. Let then
\begin{equation*}
        \psi_x^1(y) = \frac{1}{2} \left( \Psi_{\lfloor \frac{m+1}{2} \rfloor,x}(y) + \Psi_{\lceil \frac{m+1}{2} \rceil,x}(y) \right).
\end{equation*}
For notational convenience, we will write for any $p \geq 1$
\begin{equation*}
    \psi_{(f_1, \ldots, f_m)}^p(g_1, \ldots , g_m)(t) \coloneqq \psi_{(f_1(t), \ldots, f_m(t))}^p(g_1(t), \ldots , g_m(t)),
\end{equation*}
for functions $f_1,\ldots, f_m, g_1, \ldots, g_m : \mathbb{R} \to \mathbb{R}$ and $t \in \mathbb{R}$. Moreover, we will denote $\mathbb{U}_n^{-1}(t) = (\mathbb{U}_{1,n}^{-1}(t), \ldots, \mathbb{U}_{m,n}^{-1}(t))$.
\begin{theorem}[Weak convergence of the barycenter process]\label{theorem: convergence of barycenter process}
    Let $p=1$ or $p \geq 2$ and $\lambda_{p,n}$ as in \eqref{eq: barycenter process}.
    \begin{enumerate}
        \item If $\beta > 0$ and \autoref{cond 1.2} is satisfied, then  
        \begin{equation*}
            \left( \mathbb{U}_n^{-1}, \lambda_{p,n} \right) \rightsquigarrow \left(\mathbb{G}, \psi_{\Uinv}^p(\mathbb{G}) \right) \text{ in }\ell^\infty([\beta,1-\beta])^m \times L^1([\beta,1-\beta]).
        \end{equation*}
        Here $\mathbb{G} = (\mathbb{G}_1, \ldots, \mathbb{G}_m)$ is the Gaussian process defined in \autoref{theorem: Convergence under the Null}.
        \item If \autoref{cond 1.3} is satisfied with $\theta = 1$, it further holds that
        \begin{equation*}
            \left( \mathbb{U}_n^{-1}, \lambda_{p,n} \right) \rightsquigarrow \left(\mathbb{G}, \psi_{\Uinv}^p(\mathbb{G}) \right) \text{ in } L^1((0,1))^{m+1}
        \end{equation*}
        for a Gaussian process $\mathbb{G}$ as specified in the first part.
    \end{enumerate}
\end{theorem}
Before giving the proof of \autoref{theorem: convergence of barycenter process}, we provide some results on the weak convergence of the joint process $\mathbb{U}_n^{-1}$.
\subsection{Joint Weak Convergence of the Quantile Process}\label{subsec: joint convergence of u quantile processes}
\begin{lemma}\label{lemma: joint convergence inverted u process}
    Let $\beta > 0$ and assume \autoref{cond 1.2}. It then holds that 
    \begin{equation*}
        \mathbb{U}_n^{-1} = (\mathbb{U}_{1,n}^{-1}, \ldots, \mathbb{U}_{m,n}^{-1}) \rightsquigarrow \mathbb{G} \text{ in }\ell^\infty([\beta,1-\beta])^m,
    \end{equation*}
    where $\mathbb{G} = (\mathbb{G}_1, \ldots, \mathbb{G}_m)$ is the Gaussian process as defined in \autoref{theorem: Convergence under the Null}. If further, for every $i = 1, \ldots, m$, $U_i$ has compact support on an interval $[0,D_i]$ and is continuously differentiable on its support with strictly positive density on $(0,D_i)$, then we obtain 
    \begin{equation*}
        \mathbb{U}_n^{-1} \rightsquigarrow \mathbb{G} \text{ in }L^1((0,1))^m.
    \end{equation*}
\end{lemma}
In order to prove \autoref{lemma: joint convergence inverted u process}, we derive a distributional limit for the joint process $\mathbb{U}_n = (\mathbb{U}_1, \ldots, \mathbb{U}_m)$. To this end, note that for each $i=1, \ldots, m$, $U_{i,n}(t)$ (and similarly $U_i(t)$) can be rewritten as 
\begin{equation*}
    U_{i,n}(t) = \widetilde{U}_{i,n}(f_{i,t}) \coloneqq \frac{2}{n(n-1)} \sum_{1 \leq k < l \leq n} f_{i,t}(X_{k,i},X_{l,i}),
\end{equation*}
where $f_{i,t}:\mathcal{X}_i^2 \to \mathbb{R}$ is defined via $f_{i,t}(x,y) = \mathbbm{1} \!(d_i(x,y) \leq t)$. Hence, they are $U$-statistics with kernel function $f_{i,t}$ \citep[chap.~5]{serfling1980approximation} and $\mathbb{U}_{i,n}$ constitutes a $U$-process over the function class
\begin{equation} \label{eq: function class u-process}
    \mathcal{F}_i = \left\{  f_{i,t} \, : \,  t \in [C_{i,1}, C_{i,2}]   \right\},
\end{equation}
where the constants $C_{i,1}, C_{i,2}$ will be chosen depending on the context. This connection was also exploited in \citet[Lemma A.7]{weitkamp2024distribution}, in the proof of which it is shown that $\mathcal{F}_i$ is a \emph{VC-subgraph class} \citep[chap.~2.6]{wellner2013weak}, enabling the derivation of the limit distribution of $\mathbb{U}_{i,n}$. In pursuit of extending these results from weak convergence of $\mathbb{U}_{i,n}, i= 1, \ldots, m$ to weak convergence of $\mathbb{U}_n$, we further use the following characterization of joint weak convergence: Let $T_1, \ldots, T_m$ be arbitrary sets and for $i = 1, \ldots,m$ and $n \in \mathbb{N}$, let $X_{i,n} : \Omega_n \to \ell^\infty(T_i)$. Then by a straightforward extension of \citet[Theorem 1.5.4]{wellner2013weak} to product spaces of arbitrary finite dimension, it holds that 
\begin{equation*}
    \left( X_{1,n}, \ldots, X_{m,n} \right) \rightsquigarrow \left( X_1, \ldots, X_m \right) \text{ in }\ell^\infty(T_1) \times \ldots \times \ell^\infty(T_m)
\end{equation*}
for tight random elements $X_1, \ldots, X_m$, if and only if the sequences $(X_{i,n})_{n \in \N}$ are asymptotically tight for each $i = 1, \ldots, m$ and the joint finite dimensional distributions converge, i.e.\ for each set of indices $t_{i,1}, \ldots, t_{i,k_i} \in T_i$ it holds that 
\begin{equation}
\label{eq: characterization JWC}
\begin{aligned}
&\left(
X_{1,n}(t_{1,1}), \ldots, X_{1,n}(t_{1,k_1}), \ldots,
X_{m,n}(t_{m,1}), \ldots, X_{m,n}(t_{m,k_m})
\right) \\
&\qquad \rightsquigarrow
\left(
X_{1}(t_{1,1}), \ldots, X_{1}(t_{1,k_1}), \ldots,
X_{m}(t_{m,1}), \ldots, X_{m}(t_{m,k_m})
\right).
\end{aligned}
\end{equation}
Hence, weak convergence in $m$-product spaces can be characterized by tightness of the marginals and convergence of certain random vectors in the Euclidean space.
\begin{lemma}\label{lemma: joint convergence of U-processes}
    For each $1 \leq i \leq m$, let $0 \leq C_{i,1} < C_{i,2} < \infty$ and assume that $U_i$ is continuously differentiable on $[C_{i,1},C_{i,2}]$ and denote $\mathbb{U}_{i,n} = \sqrt{n}(U_{i,n}-U_i)$. Then it holds that 
    \begin{equation*}
        \mathbb{U}_n = (\mathbb{U}_{1,n}, \ldots, \mathbb{U}_{m,n}) \rightsquigarrow \mathbb{K} \text{ in }\ell^\infty([C_{1,1}, C_{1,2}]) \times \ldots \times \ell^\infty([C_{m,1}, C_{m,2}]),
    \end{equation*}
    where $\mathbb{K} = (\mathbb{K}_1, \ldots, \mathbb{K}_m)$ is an m-dimensional, centered Gaussian process with covariance structure defined by 
    \begin{equation*}
        \Cov(\mathbb{K}_i(t), \mathbb{K}_j(s)) = \gammaK{i,j}(t,s), \qquad t \in [C_{i,1}, C_{i,2}], s \in [C_{j,1},C_{j,2}],
    \end{equation*}
    with $\gammaK{i,j}(t,s)$ as defined in \autoref{theorem: Convergence under the Null}.
\end{lemma}
\begin{proof}
    By \citet[Corollary E.8]{weitkamp2024distribution} and the fact that $\mathcal{F}_i, i=1, \ldots ,m$ as defined in~\eqref{eq: function class u-process} are permissible VC-subgraph classes (see proof of \citet[Lemma A.7]{weitkamp2024distribution}), it holds that $\mathbb{U}_{i,n} \rightsquigarrow \mathbb{K}_i$ in $\ell^\infty([C_{i,1},C_{i,2}])$ for each $i=1, \ldots, m$. As permissibility of a function class implies that it is also image admissible Suslin \citep[Lemma G.6]{weitkamp2024distribution}, and since functions in $\mathcal{F}_i$ are uniformly bounded by $1$, we may apply \citet[Corollary 4.2(b)]{arcones1993limit} to conclude that
    \begin{equation} \label{eq: U process via Hoeffding}
        \left\| \mathbb{U}_{i,n}(t) - 2\sqrt{n} \left( \frac{1}{n} \sum_{j=1}^n  \mathbb{E}_{Y \sim \mu_i}\left[\mathbbm{1} \!(d_i(X_{i,j},Y) \leq t) \right] - U_i(t) \right)  \right\|_{[C_{i,1},C_{i,2}]}  = o_p(1).
    \end{equation}
    Here, $\mathbb{E}_{Y \sim \mu_i}$ denotes taking the expectation only with respect to $Y$ under the distribution $Y \sim \mu_i$. Based on~\eqref{eq: U process via Hoeffding}, define \begin{equation*}
        \widehat{\mathbb{U}}_{i,n}(t) \coloneqq 2\sqrt{n} \left(\frac{1}{n} \sum_{j=1}^n\mathbb{E}_{Y \sim \mu_i}\left[\mathbbm{1} \!(d_i(X_{i,j},Y) \leq t)\right] - U_i(t) \right).\end{equation*} 
        Instead of deriving the asymptotic distribution of $\mathbb{U}_n$ directly, we derive the asymptotic distribution of $(\widehat{\mathbb{U}}_{1,n}, \ldots, \widehat{\mathbb{U}}_{m,n})$ and use equation~\eqref{eq: U process via Hoeffding} combined with Slutzky's Lemma \citep[Lemma 1.10.2(i)]{wellner2013weak} to show that they coincide. Note that for each $i= 1, \ldots, m$, $\widehat{\mathbb{U}}_{i,n}$ converges weakly by~\eqref{eq: U process via Hoeffding} and is thus asymptotically tight. Furthermore, for each $i = 1, \ldots, m$ and a fixed $t\in[C_{i,1},C_{i,2}]$, it clearly holds that $\widehat{\mathbb{U}}_{i,n}(t)$ converges in $\mathbb{R}$ to a normal distribution by the central limit theorem. Similarly, the multivariate CLT yields for a finite number of indices $t_{i,j} \in [C_{i,1}, C_{i,2}]$ where $1\leq i \leq m$ and $1 \leq j \leq k_i$ that 
    \begin{eqnarray} \label{eq: Joint finite dimensional distribution}
        \left(\widehat{\mathbb{U}}_{1,n} (t_{1,1}), \ldots, \widehat{\mathbb{U}}_{1,n}(t_{1,k_1}), \ldots, \widehat{\mathbb{U}}_{m,n} (t_{m,1}), \ldots, \widehat{\mathbb{U}}_{m,n}(t_{m,k_m})\right)
    \end{eqnarray}
    converges to a $(\sum_{i=1}^m k_i)$-dimensional normal distribution with mean vector $0$. Let $1 \leq i,j \leq m$ and $1 \leq l_i \leq k_i$ as well as $1 \leq l_j \leq k_j$. Then the corresponding entry in the covariance matrix of the limiting distribution induced by~\eqref{eq: Joint finite dimensional distribution} is 
    \begin{equation} \label{eq: covariance of the gaussian process}
        4\Cov\left[ \mathbb{E}_{Y \sim \mu_i}\left( \mathbbm{1} \!(d_i(X_{i,1},Y) \leq t_{i,l_i}) \right),  \mathbb{E}_{Y \sim \mu_j}\left( \mathbbm{1} \!(d_j(X_{j,1},Y) \leq t_{j,l_j}) \right) \right]
    \end{equation}
    Thus, by~\eqref{eq: characterization JWC} and the preceding discussion we obtain convergence of $\widehat{\mathbb{U}}_n = (\widehat{\mathbb{U}}_{1,n}, \ldots , \widehat{\mathbb{U}}_{m,n})$ to a centered Gaussian process $\mathbb{K}$. The distribution of $\mathbb{K}$ is uniquely determined by the covariances as given in~\eqref{eq: covariance of the gaussian process}. The Lemma now follows upon realizing that by~\eqref{eq: U process via Hoeffding}
    \begin{equation*}
        \sum_{i=1}^m \norm{ \mathbb{U}_{i,n} - \widehat{\mathbb{U}}_{i,n} }_{[C_{i,1}, C_{i,2}]} = o_p(1),
    \end{equation*}
    where the left hand side is a product norm on $\ell^\infty([C_{1,1}, C_{1,2}]) \times \ldots \times \ell^\infty([C_{m,1}, C_{m,2}])$.
\end{proof}
With \autoref{lemma: joint convergence of U-processes} at our disposal, we are now ready to prove \autoref{lemma: joint convergence inverted u process}.
\begin{proof}[Proof of \autoref{lemma: joint convergence inverted u process}]
    First, let $\beta > 0$ and assume \autoref{cond 1.2} with constants $C_{i,1},C_{i,2}$ as specified there for $i = 1, \ldots, m$. Furthermore, denote by $\mathbb{D}_{1,i}$ the set of all restrictions of distribution functions to $[C_{i,1},C_{i,2}]$ and let $(\mathcal{D}([C_{i,1},C_{i,2}]),\|\cdot\|_{[C_{i,1},C_{i,2}]})$ be the space of c\`adl\`ag functions on $[C_{i,1},C_{i,2}]$  equipped with the uniform norm. For the sake of notation we denote in the following the product spaces $\mathbb{D}_1^m = \mathbb{D}_{1,1} \times \ldots \times \mathbb{D}_{1,m}$ and $\mathcal{D}^m = \mathcal{D}([C_{1,1},C_{1,2}]) \times \ldots \times \mathcal{D}([C_{m,1},C_{m,2}])$. Consider the map $\invmap: \mathbb{D}_1^m \subset \mathcal{D}^m \to \ell^\infty([\beta,1-\beta])^m$ defined via
    \begin{equation*}
        \invmap(F_1, \ldots, F_m) = (F_1^{-1}, \ldots , F_m^{-1})
    \end{equation*}
    mapping each component $F_i$ to its generalized inverse 
    \begin{equation*}
        F^{-1}:(0,1) \to \mathbb{R}, \qquad F^{-1}(t) = \inf \lbrace x \in \mathbb{R} : F(x) \geq t \rbrace.
    \end{equation*}
    Under the assumptions of \autoref{cond 1.2} and by  \citet[Lemma 3.10.24]{wellner2013weak}, we obtain that each of the $i = 1, \ldots,m$ components of this map is Hadamard differentiable at $U_i$ tangentially to $\mathcal{C}([C_{i,1},C_{i,2}])$ with derivative $F \mapsto -(F/u_i) \circ U_i^{-1}$. Thus, $\invmap$ is Hadamard differentiable at $\mathbf{U} = (U_1, \ldots , U_m)$ tangentially to $\mathcal{C}^m = \mathcal{C}([C_{1,1},C_{1,2}]) \times \ldots \times \mathcal{C}([C_{m,1},C_{m,2}])$ with Hadamard derivative 
    \begin{equation} \label{eq: def of invmap}
        D_{\mathbf{U}}\invmap(F_1, \ldots, F_m) = - \left((F_1/u_1) \circ U_1^{-1}, \ldots ,  (F_m/u_m) \circ U_m^{-1}\right).
    \end{equation}
    Combining this with the delta method \citep[Theorem 3.10.4]{wellner2013weak} and the limit law from \autoref{lemma: joint convergence of U-processes} yields that 
    \begin{equation*}
        \sqrt{n}(\invmap(\Un)-\invmap(\mathbf{U})) \rightsquigarrow D_\mathbf{U} \invmap (\mathbb{K}_1, \ldots, \mathbb{K}_m) \text{ in }\ell^\infty([\beta,1-\beta])^m.
    \end{equation*}
    Here, $\mathbb{K} = (\mathbb{K}_1, \ldots, \mathbb{K}_m)$ is the $m$-dimensional centered Gaussian process with covariance structure as defined in \autoref{lemma: joint convergence of U-processes}. As $D_\mathbf{U} \invmap$ is a linear transformation, the limit $D_\mathbf{U} \invmap(\mathbb{K})$ is also a Gaussian process \citep[Lemma 2.2]{bogachev1998gaussian} and the covariance structure follows from elementary calculations. Next, we assume that for each $i =1, \ldots, m$, the support of $U_i$ is $[0,D_i]$ and $U_i$ is continuously differentiable on $[0,D_i]$ with strictly positive derivative $u_i$ on $(0,D_i)$. We then let $\mathbb{D}_{2,i}$ be the set of distribution functions that concentrate on $(0, D_i]$ and $\mathbb{D}_2^m = \mathbb{D}_{2,1} \times \ldots \times \mathbb{D}_{2,m}$. As a consequence of \citet[Lemma A.18]{weitkamp2024distribution}, we obtain that the inversion map when considered as a map $\invmap : \mathbb{D}_2^m \subset \mathcal{D}^m \to L^1((0,1))^m$ is Hadamard-differentiable at $\mathbf{U}$ tangentially to $\mathcal{C}([0,D_1]) \times \ldots \times \mathcal{C}([0,D_m])$ with the same derivative as in equation~\eqref{eq: def of invmap}. Then the rest of the proof follows analogously by an application of the delta method.
\end{proof}
\subsection{Weak Convergence of the $l$-th Order Statistics Process}
For the proof of \autoref{theorem: convergence of barycenter process}, we require some further preparation. Recall that for $p=1$, we define $ \Lambda_1(x) = \frac{1}{2} (x_{( \lceil \frac{m+1}{2} \rceil )} + x_{( \lfloor \frac{m+1}{2} \rfloor )})$. In order to derive the distributional limit of $\lambda_{1,n}$ in~\eqref{eq: barycenter process}, we first study the behavior of the $l$-th order quantile processes
\begin{equation*}
    \lambda_n^{(l)} = \sqrt{n} \left( \left(U_{1,n}^{-1}, \ldots,U_{m,n}^{-1}\right)_{(l)} - \left(U_{1}^{-1}, \ldots,U_{m}^{-1}\right)_{(l)} \right).
\end{equation*}
To this end, recall the definition of $\Psi_{l,x}$ from \eqref{eq: l-th order quantile process}.
\begin{proposition} \label{prop: l-th quantile process}
    Let $\beta > 0$. Under \autoref{cond 1.2}, it holds that for each $l = 1, \ldots, m$
    \begin{equation*}
        \left(\mathbb{U}_n^{-1}, \lambda_n^{(l)} \right) \rightsquigarrow \left(\mathbb{G}, \Psi_{l, \Uinv}(\mathbb{G}) \right) \text{ in }\ell^\infty([\beta,1-\beta])^m \times L^1([\beta,1-\beta]),
    \end{equation*}
    where $\mathbb{G}$ is the $m$-dimensional centered Gaussian process as defined in \autoref{theorem: convergence of barycenter process}. If furthermore \autoref{cond 1.3} is satisfied with $\theta = 1$, then 
    \begin{equation*}
        \left(\mathbb{U}_n^{-1}, \lambda_n^{(l)} \right) \rightsquigarrow \left(\mathbb{G}, \Psi_{l, \Uinv}(\mathbb{G}) \right) \text{ in }
        L^1((0,1))^{m+1}.
    \end{equation*}
\end{proposition} 
For the proof of \autoref{prop: l-th quantile process}, we require some auxiliary results.
\begin{lemma}\label{lemma: J_p is finite}
    Let $p \geq 1$, assume that \autoref{cond 1.3} holds with $\theta = p$. Then $J_p(\mu^{U_i}) < \infty$ for each $i=1, \ldots, m$, where $\mu^{U_i}$ is the probability measure corresponding to $U_i$.
\end{lemma}
\begin{proof}
    Under \autoref{cond 1.3} with $\theta = p$, it holds by \citet[Corollary A.22]{bobkov2019one} that
    \begin{equation*}
        J_p(\mu^{U_1}) = \int_{-\infty}^\infty \frac{\left[ U_1(t)(1-U_1(t)) \right]^{p/2}}{u_1(t)^{p-1}} \diff t = \int_{0}^1 \left((U_1^{-1})^\prime(t)\right)^p \left( t (1-t) \right)^{p/2} \diff t.
    \end{equation*}
    Furthermore, by assumption there are constants $\gamma_{1,1}, \gamma_{1,2}> -1/p - 1/2$, such that for any $t \in (0,1)$, $(U_1^{-1})^\prime(t) \leq c_{U_1} t^{\gamma_{1,1}}(1-t)^{\gamma_{1,2}}$. Applying this inequality to the above display leads to
\begin{equation*}
    J_p(\mu^{U_1}) \leq c_{U_1}^p  \int_0^1 t^{p(\gamma_{1,1} +1/2)} (1-t)^{p(\gamma_{1,2} +1/2)} \diff t.
\end{equation*}
The bounds on $\gamma_{1,1}$ and $\gamma_{1,2}$ further imply that $p(\gamma_{1,i} + 1/2) +1 > 0$ for $i = 1,2$. Hence,
\begin{equation*}
    J_p(\mu^{U_1}) \leq c_{U_1}^p B\Bigl(p(\gamma_{1,1} + 1/2) + 1, p(\gamma_{1,2} + 1/2) + 1\Bigr) < \infty,
\end{equation*}
where $B$ denotes the Beta function.
\end{proof}
\begin{lemma} \label{lemma: Lipschitz continuity convergence to 0}
    Let $\Xi_t : \mathbb{R}^m \to \mathbb{R}$ for $t \in T \subset \mathbb{R}$ be a family of Lipschitz continuous maps with uniform Lipschitz constant and denote
    \begin{equation}
        \zeta_n(t) \coloneqq \sqrt{n} \left( \Xi_t(\Uninv) - \Xi_t(\Uinv)\right).
    \end{equation}
    Let further $(A_\delta)_{\delta > 0} \subset \mathcal{B}(T)$, where $\mathcal{B}(T)$ is the Borel sigma algebra, such that $A_\delta \searrow A$ for some $A$ of measure $0$.
    \begin{enumerate}
        \item Let $\beta >0$, $T = [\beta,1-\beta]$ and assume \autoref{cond 1.2}. Then
        \begin{equation*}
            \lim_{\delta \searrow 0} \limsup_{n \to \infty} \mathbb{P}(\|\zeta_n\mathbbm{1}_{A_\delta}\|_{L^1(T)} > \varepsilon) = 0,
        \end{equation*}
        for each $\varepsilon > 0$.
        \item Let now $T = (0,1)$ and assume \autoref{cond 1.3} with $\theta = 1$. Then
        \begin{equation*}
            \lim_{\delta \searrow 0} \limsup_{n \to \infty} \mathbb{P}(\|\zeta_n\mathbbm{1}_{A_\delta}\|_{L^1(T)} > \varepsilon) = 0,
        \end{equation*}
        for each $\varepsilon > 0$.
    \end{enumerate}
\end{lemma}
\begin{proof}
    Denote $L > 0$ the uniform Lipschitz constant of the functions $\Xi_t$. By the Lipschitz continuity assumption, we obtain
    \begin{align}
        \mathbb{P}(\left\| \zeta_n \mathbbm{1}_{A_\delta}\right\|_{L^1(T)} > \varepsilon) &\leq \mathbb{P} \left(L \sum_{i=1}^m \left\|\mathbbm{1}_{A_\delta} \mathbb{U}_{i,n}^{-1}\right\|_{L^1(T)}  > \varepsilon\right) \nonumber \\
        & \leq \sum_{i=1}^m \mathbb{P}\left(\left\|\mathbbm{1}_{A_\delta} \mathbb{U}_{i,n}^{-1} \right\|_{L^1(T)} > \frac{\varepsilon}{mL} \right). \label{eq: equation in proof lipschitz maps}
    \end{align} 
    \begin{enumerate}
        \item Let $\beta > 0 $ and $T = [\beta,1-\beta]$. By \autoref{lemma: joint convergence inverted u process}, it holds that $\mathbb{U}_n^{-1} \rightsquigarrow \mathbb{G}$ in $\ell^\infty(T)^m$ under \autoref{cond 1.2}. Hence, $\|\mathbbm{1}_{A_\delta} \mathbb{U}_{i,n}^{-1}\|_{L^1(T)}$ converges weakly to $\|\mathbbm{1}_{A_\delta}\mathbb{G}_i\|_{L^1(T)}$ by the continuous mapping theorem. As a consequence, taking $\limsup_{n \to \infty}$ on both sides of the above display and applying the Portmanteau theorem \citep[Theorem 1.3.4]{wellner2013weak} and Markov's inequality yields 
    \begin{align*}
        \limsup_{n \to \infty} \mathbb{P}(\| \zeta_n \mathbbm{1}_{A_\delta}\|_{L^1(T)} > \varepsilon) &\leq \sum_{i=1}^m \mathbb{P}\left(\| \mathbb{G}_i 1 _{A_\delta} \|_{L^1(T)} \geq \frac{\varepsilon}{mL} \right) \\
        &\leq \frac{mL}{\varepsilon}\sum_{i=1}^m \mathbb{E}\left(\| \mathbb{G}_i 1 _{A_\delta} \|_{L^1(T)}  \right).
    \end{align*}
    Note that by \citet[Theorem 2.1.1]{adler2007random}, $\mathbb{E}(\|\mathbb{G}_i\|_T) < \infty$ for each $i = 1, \ldots, m$. Hence, we may apply the dominated convergence theorem to conclude that the right hand side converges to $0$ as $\delta \searrow 0$.
    \item Let $T = (0,1)$ and assume that \autoref{cond 1.3} is satisfied with $\theta = 1$. By the second part of \autoref{lemma: joint convergence inverted u process}, it then holds that $\mathbb{U}_n^{-1}\rightsquigarrow \mathbb{G}$ in $L^1((0,1))$. Starting from equation \eqref{eq: equation in proof lipschitz maps}, we thus again apply the Portmanteau Theorem and Markov inequality to obtain
    \begin{equation} \label{eq: equation in proof of lipschitz}
        \limsup_{n \to \infty} \mathbb{P}(\| \zeta_n \mathbbm{1}_{A_\delta}\|_{L^1(T)} > \varepsilon) \leq \frac{mL}{\varepsilon}\sum_{i=1}^m \mathbb{E}\left(\| \mathbb{G}_i 1 _{A_\delta} \|_{L^1(T)}  \right).
    \end{equation}
    Now, let $Y_{i,1}, \ldots, Y_{i,n}$ be an independent copy of $X_{i,1}, \ldots, X_{i,n}$ and denote by $\Tilde{\mu}_n^{V_{i}}$ the empirical measure corresponding to the distances induced by $Y_{i,1}, \ldots, Y_{i,n}$. It is remarked in \citet[sec. 4]{bobkov2019one} that
    \begin{equation*}
        \mathbb{E}\left[ \mathcal{W}_1\left(\mu_n^{U_i}, \mu^{U_i} \right)  \right] \leq 2  \mathbb{E}\left[ \mathcal{W}_1\left(\mu_n^{U_i}, \Tilde{\mu}_n^{V_{i}} \right)  \right].
    \end{equation*}
    Using this fact and then applying \citet[Theorem 2.5 and Remark A.6]{weitkamp2024distribution} with $p=1$, we thus obtain that 
    \begin{align*}
        \mathbb{E}\left(\| \mathbb{U}_{i,n}^{-1} \|_{L^1(0,1)} \right) &\leq 2 \sqrt{n} \mathbb{E}\left[ \mathcal{W}_1\left(\mu_n^{U_i}, \Tilde{\mu}_n^{V_{i}} \right)  \right] \\
        & \leq C J_1(\mu^{U_i}),
    \end{align*}
    for some finite constant $C < \infty$. Furthermore $J_1(\mu^{U_i})$ is finite under \autoref{cond 1.3} with $\theta = 1$ by \autoref{lemma: J_p is finite}. Thus, the expectations of $\|\mathbb{U}_n^{-1}\|_{L^1((0,1))}$ are uniformly bounded. As the map $x \mapsto \abs{x}$ is continuous and bounded from below, the Portmanteau Theorem yields
    \begin{equation*}
        \mathbb{E}\left( \| \mathbb{G}_i\|_{L^1(0,1)} \right) \leq \liminf_{n \to \infty} \mathbb{E}(\| \mathbb{U}_{i,n}^{-1} \| ) < \infty.
    \end{equation*}
    Hence, we can apply the dominated convergence theorem in \eqref{eq: equation in proof of lipschitz} when letting $\delta \searrow 0$ in order to obtain the desired result.
    \end{enumerate}
\end{proof}
\begin{lemma}\label{lemma: LLN for U quantiles}
    Let \autoref{cond 1.2} be satisfied. Then it holds that
    \begin{equation*}
        \|U_{i,n}^{-1} - U_i^{-1} \|_{[\beta,1-\beta]} \convAS 0, \qquad i=1, \ldots, m.
    \end{equation*}
    If further, $U_1, \ldots, U_m$ are continuous and strictly increasing on their supports, then 
    \begin{equation*}
        \|U_{i,n}^{-1} - U_i^{-1} \|_{(0,1)} \convAS 0, \qquad i = 1, \ldots, m.
    \end{equation*}
\end{lemma}
\begin{proof}
    The proof goes along the lines of \citet{bogoya2016convergence}. For now, assume the conditions of the first part. Recall that, by the discussion preceding \autoref{lemma: joint convergence of U-processes}, we may view $U_i$ as a $U$-statistic indexed by a function class which is image admissible Suslin and uniformly bounded. Hence, an application of \citet[Theorem 3.1(i)]{arcones1993limit} yields that
    \begin{equation} \label{eq: LLN for distance kernels}
        \| \Un- \U \|_{[C_1,C_2]} \convAS 0,
    \end{equation}
    with the norm $\| \cdot \|_{[C_1,C_2]}$ as defined in~\eqref{eq: norm on product space}. We now consider a fixed $\omega \in \Omega$ such that equation~\eqref{eq: LLN for distance kernels} holds. Then, clearly for each $t \in [\beta,1-\beta]$, it holds that $\Uninv(\omega,t) \to \Uinv(t)$, due to the componentwise continuity and monotonicity of $\mathbf{U}$. Hence, $\Uninv(\omega, \cdot)$ is a sequence of componentwise monotone functions (as they are quantile functions), which converges pointwise to $\Uinv$ on the compact set $[\beta,1-\beta]$, which implies that the convergence is uniform. As this holds for every $\omega$ in a set of probability one, this also implies that the uniform convergence takes place almost surely and thus proves the first part of the lemma. The second part follows immediately by \citet[Theorem 1.1]{bogoya2016convergence}.
\end{proof}
For the derivation of \autoref{prop: l-th quantile process}, it will be convenient to introduce the smallest positive deviation (SPD) as a map $\text{SPD}: \mathbb{R}^m \to \mathbb{R}$ via
\begin{equation}\label{eq: def of SPD}
    \text{SPD}(x) = \inf \lbrace |x_i-x_j| : 1 \leq i,j \leq m, |x_i - x_j |>0  \rbrace
\end{equation}
with the convention that $\inf \emptyset = \infty$. Furthermore, we will view $\lambda_n^{(l)}$ separately on the intervals, where $\Uinv(t) = (U_1^{-1}(t), \ldots, U_m^{-1}(t))$ has the same ordering in the sense that the indices of the corresponding order statistics are the same. Thus, we define the auxiliary map $\kappa: \mathbb{R}^m \to \mathcal{P}(\lbrace 1, \ldots, m \rbrace)^m$ by
\begin{equation*}
    \kappa(x) = \left(\lbrace i = 1, \ldots, m: x_i = x_{(l)}  \rbrace \right)_{1 \leq l \leq m}.
\end{equation*}
For example, $\kappa(1,3,2,2) = (\lbrace 1 \rbrace, \lbrace 3,4 \rbrace, \lbrace 3,4 \rbrace, \lbrace 2 \rbrace)$. For each $Q \in \mathcal{P}(\lbrace 1, \ldots,m \rbrace)^m$, we then regard the process $\lambda_{n}^{(l)}$ on the set of $t$ such that $\kappa(\Uinv(t)) = Q$.
\begin{proof}[Proof of \autoref{prop: l-th quantile process}]
First, let $\beta > 0$ and assume \autoref{cond 1.2} and  $l \in \lbrace 1 , \ldots, m \rbrace$. We start by showing that
\begin{equation}\label{eq: order statistic quantile process in proof}
    \norm{\lambda_n^{(l)}- \Psi_{l, \Uinv}(\mathbb{U}_n^{-1})}_{L^1([\beta,1-\beta])} \convP 0.
\end{equation}
To this end, for each $Q \in \mathcal{P}(\lbrace 1, \ldots , m \rbrace)^m$, we define
\begin{equation*}
    C^Q = \left\{ t \in [\beta,1-\beta]: \kappa(\Uinv(t)) = Q \right\}. 
\end{equation*}
It is clear that 
\begin{equation*}
    [\beta,1-\beta] = \cup_{Q \in \mathcal{P}(\lbrace 1, \ldots, m \rbrace)} C^Q
\end{equation*}
and for $Q_1 \neq Q_2$ it holds that $C^{Q_1} \cap C^{Q_2} = \emptyset$. 
Hence it holds that 
\begin{equation} \label{eq: L^1 decomp of median}
     \norm{\lambda_n^{(l)}- \Psi_{l, \Uinv}(\mathbb{U}_n^{-1}) }_{L^1([\beta,1-\beta])} = \sum_{Q \in \mathcal{P}(\lbrace 1, \ldots,m \rbrace)^m} \norm{\lambda_n^{(l)}- \Psi_{l, \Uinv}(\mathbb{U}_n^{-1}) }_{L^1(C^Q)}
\end{equation}
Once we have proven that each of the summands above converges to $0$ in probability, \eqref{eq: order statistic quantile process in proof} follows. Let $Q_0 = \lbrace 1, \ldots,m\rbrace^m$. For any $t \in C^{Q_0}$ it holds that $U_1^{-1}(t) =  ... = U_m^{-1}(t)$, from which can be deduced that  
\begin{equation*}
    \lambda_{n}^{(l)}\mathbbm{1}_{C^{Q_0}} =\left(\mathbb{U}_n^{-1} \right)_{(l)}\mathbbm{1}_{C^{Q_0}} = \Psi_{l, \Uinv}(\mathbb{U}_n^{-1})\mathbbm{1}_{C^{Q_0}}.
\end{equation*}
Hence, we now fix some $Q \in \mathcal{P}(\lbrace 1, \ldots, m \rbrace)^m$ such that $Q \neq Q_0$. For this $Q$ and $\delta>0$, we then let
\begin{equation*}
    C_\delta^{Q} = \left\{ t \in C^Q: \text{SPD}(\Uinv(t)) > \delta \right\}.
\end{equation*}
Here, SPD is the smallest positive deviation defined in \eqref{eq: def of SPD}. Since $Q \neq Q_0$, we know that for each $t \in C^Q$ the smallest positive deviation of $\Uinv$ is greater than 0, meaning that $C_\delta^Q \nearrow C^Q$ as $\delta \searrow 0$. We then expand
    \begin{align*}
        \left(\lambda_{n}^{(l)} - \Psi_{l,\Uinv}(\mathbb{U}_n^{-1}) \right)\mathbbm{1}_{C^Q} = & \underbrace{\left(\lambda_{n}^{(l)}\mathbbm{1}_{C^Q} - \lambda_{n}^{(l)}\mathbbm{1}_{C_\delta^Q}\right)}_{\Delta_{n, \delta}^{(1)}} + \underbrace{\left(\lambda_{n}^{(l)}\mathbbm{1}_{C_\delta^Q} - \Psi_{l,\Uinv}(\mathbb{U}_n^{-1}) \mathbbm{1}_{C_\delta^Q} \right)}_{\Delta_{n,\delta}^{(2)}}
        \\& +\underbrace{\left(\Psi_{l,\Uinv}(\mathbb{U}_n^{-1})\mathbbm{1}_{C_\delta^Q} - \Psi_{l,\Uinv}(\mathbb{U}_n^{-1})\mathbbm{1}_{C^Q} \right)}_{\Delta_{n,\delta}^{(3)}}.
    \end{align*}
    By \citet[Example 2.29]{wainwright2019high}, the $l$-th order statistics map
    \begin{equation*}
        \mathrm{os}_l : \mathbb{R}^m \to \mathbb{R}, \qquad x \mapsto x_{(l)}
    \end{equation*}
    is Lipschitz continuous. Hence, we may apply \autoref{lemma: Lipschitz continuity convergence to 0}(a) to conclude that 
    \begin{equation*}
         \lim_{\delta \searrow 0} \limsup_{n \to \infty} \mathbb{P} \left( \left\|\lambda_{n}^{(l)}\mathbbm{1}_{\left(C_\delta^Q\right)^c}   \right\|_{L^1(C^Q)} > \varepsilon\right) = 0
    \end{equation*}
    for any $\varepsilon > 0$. Regarding $\Delta_{n,\delta}^{(3)}$, note that $\kappa(\Uinv(t)) = Q$ for every $t \in C^Q$. Together with Lipschitz continuity of $\mathrm{os}_l$, this implies that $\Psi_{l,\Uinv}$ is Lipschitz continuous as a map into $L^1(C^Q)$. By similar arguments as in \autoref{lemma: Lipschitz continuity convergence to 0}, we thus obtain that also
    \begin{equation*}
         \lim_{\delta \searrow 0} \limsup_{n \to \infty} \mathbb{P} \left( \left\|\Psi_{l, \Uinv}(\mathbb{U}_n^{-1})\mathbbm{1}_{\left(C_\delta^Q\right)^c}   \right\|_{L^1(C^Q)} > \varepsilon\right) = 0
    \end{equation*}
    for any $\varepsilon > 0$. Thus, only the treatment of $\Delta_{n,\delta}^{(2)}$ remains. Define the set 
    \begin{equation*}
        E_{n,\delta} = \left\{ \omega: \max_{i=1,\ldots,m}\left\| U_{i,n}^{-1}(\omega,t) - U_{i}^{-1}(t) \right\|_{C_\delta^Q}  < \frac{\delta}{2} \right\}.
    \end{equation*}
    Clearly, for $\omega \in E_{n,\delta}$ and $t \in C_\delta^Q$, we have $\lambda_n^{(l)}(\omega,t) = \Psi_{l,\Uinv}(\mathbb{U}_n^{-1}(\omega,t))$, which implies that 
    \begin{equation} \label{eq: second summand order statistic}
        \Delta_{n,\delta}^{(2)} = \left[ \lambda_{n}^{(l)} - \Psi_{l,\Uinv}(\mathbb{U}_n^{-1}) \right] \mathbbm{1}_{C_\delta^Q}\mathbbm{1}_{E_{n,\delta}^c}.
    \end{equation}
    By the LLN for $U_i^{-1}$ in \autoref{lemma: LLN for U quantiles}, for any fixed $\delta > 0$ it holds that $\mathbb{P}(E_{n,\delta}^c) \to 0$ as $n \to \infty$. Combined with~\eqref{eq: second summand order statistic}, this shows 
    \begin{equation*}
        \limsup_{n \to \infty} \mathbb{P} \left( \left\| \lambda_{n}^{(l)} \mathbbm{1}_{C_\delta^Q} - \Psi_{l,\Uinv}(\mathbb{U}_n^{-1}) \mathbbm{1}_{C_\delta^Q} \right\|_{L^1(C^Q)} > \varepsilon \right) = 0
    \end{equation*}
    Combining our findings, we have thus shown that 
    \begin{equation} \label{eq: convergence second summand order stat}
         \lim_{n \to \infty}\mathbb{P}\left(\left\| \lambda_{n}^{(l)} - \Psi_{l,\Uinv}(\mathbb{U}_n^{-1}) \right\|_{L^1(C^Q)} > \varepsilon \right) = 0
    \end{equation}
    As equation~\eqref{eq: convergence second summand order stat} holds for any choice of $Q$, combined with~\eqref{eq: L^1 decomp of median} this implies
    \begin{equation*}
        \norm{\lambda_n^{(l)}- \Psi_{l, \Uinv}(\mathbb{U}_n^{-1})}_{L^1([\beta,1-\beta])} \convP 0
    \end{equation*}
    Moreover, $\Psi_{l,\Uinv}:L^1(C^Q)^m \to L^1(C^Q)$ is continuous for each $Q \in \mathcal{P}(\lbrace 1, \ldots,m \rbrace)^m$, since $\mathrm{os}_l$ is Lipschitz continuous and the ordering of $\Uinv$ remains unchanged on $C^Q$. Denoting $\mathcal{C}^{Q,m} \coloneqq [\beta,1-\beta]^m \times C^Q$, it thus holds that
    \begin{align*}      \left(\mathbb{U}_n^{-1},\Psi_{l,\Uinv}\left(\mathbb{U}_n^{-1}\right) \right) &= \sum_{Q \in \mathcal{P}(\lbrace 1, \ldots,m \rbrace)^m} \left(\mathbb{U}_n^{-1},\Psi_{l,\Uinv}\left(\mathbb{U}_n^{-1}\right) \right)\mathbbm{1}_{C^{Q,m}} \\
    &\rightsquigarrow\left(\mathbb{G}, \Psi_{l,\Uinv}(\mathbb{G}) \right) \text{ in }\ell^\infty([\beta,1-\beta])^m \times L^1([\beta,1-\beta])
    \end{align*}
    by the continuous mapping theorem and \autoref{lemma: joint convergence inverted u process}. Note that when \autoref{cond 1.3} is satisfied with $\theta = 1$, then \autoref{lemma: Lipschitz continuity convergence to 0} holds on $(0,1)$ as well and $\mathbb{U}_n^{-1}$ converges to $\mathbb{G}$ in $L^1((0,1))^m$ by \autoref{lemma: joint convergence inverted u process}. Hence, the proof can be carried out analogously. 
\end{proof}

\subsection{Proof of \autoref{theorem: convergence of barycenter process}}
We begin by collecting some auxiliary results.

\begin{lemma}\label{lemma: properties barycenter}
    Let $D^m = \lbrace x \in \mathbb{R}^m : x_1 = \ldots =  x_m \rbrace$.
    \begin{enumerate}
        \item For $p \geq 2$, it holds that $\Lambda_p \in \mathcal{C}^1(\mathbb{R}^m \setminus D^m)$ and the derivative at a point $x \in \mathbb{R}^m \setminus D^m$ in the direction $y\in \mathbb{R}^m$ is given by 
        \begin{equation} \label{eq: def of pointwise barycenter derivative}
            (d\Lambda_p)_x(y) = \frac{\sum_{i=1}^m |x_i - \Lambda_p(x)|^{p-2}y_i}{\sum_{i=1}^m |x_i - \Lambda_p(x)|^{p-2}} \eqqcolon \psi_x^p(y).
        \end{equation}
        \item For $p \geq 1$, $\Lambda_p$ is Lipschitz continuous with constant $L=1$ w.r.t.\ the norm $\| x \|_\infty = \max_{i=1, \ldots, m} \abs{x_i}$ for $x \in \mathbb{R}^m$.
        \item For $p \geq 2$, the map $\psi_x^p: \mathbb{R}^m \to \mathbb{R}$ as stated in~\eqref{eq: def of pointwise barycenter derivative} is Lipschitz continuous for each $x \in \mathbb{R}^m \setminus D^m$ with uniform Lipschitz constant in $x$.
        \item For any compact set $K \subset \mathbb{R}^m$ satisfying $\inf \lbrace \|x - y \|: x \in K, y \in D^m\rbrace > 0$, it holds that 
        \begin{equation}
            \sup_{x \in K} \left|\Lambda_p(x+y)- \Lambda_p(x) - \psi_x^p(y) \right| = o(\|y \|).
        \end{equation}
    \end{enumerate}
\end{lemma}
\begin{proof}
    \begin{enumerate}
        \item This is a direct consequence of the implicit function theorem. A sketch of the proof is as follows: Let $g: \mathbb{R}^m \times \mathbb{R} \to \mathbb{R}$ be defined via 
        \begin{equation*}
            g(x,y) = \sum_{i=1}^m \abs{x_i-y}^p.
        \end{equation*}
        Then clearly
        \begin{equation*}
            \partial_y g(x,y) = -\sum_{i=1}^m p\abs{x_i - y}^{p-1} \sgn(x_i-y),
        \end{equation*}
        which is a continuously differentiable function as $p\geq 2$. Moreover, for each $x \in \mathbb{R}^m$ it holds that $\partial_y g(x, \Lambda_p(x)) = 0$, as $\Lambda_p(x)$ is by definition the minimizer of $g(x, \cdot)$. Furthermore, if $x \in \mathbb{R}^m \setminus D^m$, then \begin{equation*}
            \partial_y^2 g(x, \Lambda_p(x)) = \sum_{i=1}^m p(p-1) \abs{x_i - \Lambda_p(x)}^{p-2} > 0.        \end{equation*}
        Thus, at each such $x$, the assumptions of the implicit function theorem are satisfied for $\partial_y g(x, \Lambda_p(x))$ and the explicit form of the derivative is then a simple consequence of the theorem.
        \item Note that for any $a,b \in \mathbb{R}^m$, it holds by definition of $\Lambda_p$ that $a \leq b$ coordinatewise implies $\Lambda_p(a) \leq \Lambda_p(b)$. Let $x, y \in \mathbb{R}^m$ and $\delta = \left\| x - y \right\|_\infty$. Then $x - \delta \leq y \leq x + \delta$ and hence it holds that $\Lambda_p(x-\delta) \leq \Lambda_p(y) \leq \Lambda_p(x + \delta)$. Combining this with the fact that for $z \in \mathbb{R}^m$ and $t \in \mathbb{R}$, it holds that $\Lambda_p(z + t) = \Lambda_p(z) + t$, we thus obtain
        \begin{equation*}
            \Lambda_p(x) - \delta \leq \Lambda_p(y) \leq \Lambda_p(x) + \delta,
        \end{equation*}
        which yields the claim.
        \item Let $x \in \mathbb{R}^m\setminus D^m$ and $z,y \in \mathbb{R}^m$. Then by the Cauchy-Schwarz inequality
        \begin{align*}
            \abs{\psi_x^p(z-y)} &\leq \frac{\sum_{i=1}^m |x_i - \Lambda_p(x)|^{p-2}| z_i-y_i |}{\sum_{i=1}^m |x_i - \Lambda_p(x)|^{p-2}} \\
            &\leq \sqrt{\sum_{i=1}^m \left( \frac{|x_i - \Lambda_p(x)|^{p-2}}{\sum_{j=1}^m |x_j - \Lambda_p(x)|^{p-2}} \right)^2 } \sqrt{\sum_{i=1}^m |z_i-y_i|^2}.
        \end{align*}
        Now the claim follows upon realizing that the first factor of the above display is uniformly bounded in $x$ by $1$.
        \item This is a direct consequence of the first statement, since $\psi_x^p$ is the continuous derivative of $\Lambda_p$ on $K$. The rest follows by compactness of $K$ and simple analytical arguments.
    \end{enumerate}
\end{proof}

Equipped with these results, we are now ready to tackle the weak convergence of $(\mathbb{U}_n^{-1}, \lambda_{p,n})$.
\begin{proof}[Proof of \autoref{theorem: convergence of barycenter process}]
    We first provide the proof of \autoref{theorem: convergence of barycenter process}(a). 
    \par \medskip \noindent
    \textbf{Case $\mathbf{p = 1}$:} Note that 
        \begin{equation*}
            \left( \mathbb{U}_n^{-1}, \lambda_{1,n} \right) = \frac{1}{2}\left(\mathbb{U}_n^{-1}, \lambda_{n}^{ \left( \lfloor \frac{m+1}{2} \rfloor \right)}   \right) + \frac{1}{2}\left(\mathbb{U}_n^{-1}, \lambda_{n}^{ \left( \lceil \frac{m+1}{2} \rceil \right)}   \right).
        \end{equation*}
        The asymptotic distribution of both summands is a consequence of \autoref{prop: l-th quantile process}. A quick inspection of the proof of \autoref{prop: l-th quantile process} then yields that the sum converges weakly to the sum of the limiting distributions, since these derivations hinge on a continuous mapping argument used on the same sequence of random elements.
        \par \medskip \noindent
    \textbf{Case $\mathbf{p \geq 2}$}. Consider the sets
    \begin{equation*}
        C \coloneqq \left\{ t \in [\beta,1-\beta]: \sum_{i=1}^m |U_{i}^{-1}(t) - \Lambda_p(\Uinv(t))| >0 \right\}
    \end{equation*}
    and for each $\delta > 0$
    \begin{equation*}
        C_\delta \coloneqq \left\{ t \in [\beta,1-\beta]:  \text{SPD}(U_1^{-1}(t), \ldots, U_m^{-1}(t)) > \delta  \right\}.
    \end{equation*}
    Notice that by definition, $C = \lbrace t \in [\beta,1-\beta]: \text{SPD}(\Uinv(t)) > 0 \rbrace$ and thus $C_\delta \subset C$ for each $\delta >0$. Furthermore, it holds that $C_{\delta_1} \subset C_{\delta_2}$ for $\delta_2 \leq \delta_1$. Similarly to the the proof of \autoref{prop: l-th quantile process}, we first show that 
    \begin{equation*} 
        \norm{\left(\lambda_{p,n} - \psi_{\Uinv}^p(\mathbb{U}_n^{-1}) \right)\mathbbm{1}_C}_{L^1([\beta,1-\beta])} \convP 0.
    \end{equation*}
    To this end, we expand
    \begin{align*}
        \left(\lambda_{p,n} - \psi_{\Uinv}^p(\mathbb{U}_n^{-1}) \right)\mathbbm{1}_C = & \underbrace{\lambda_{p,n}\mathbbm{1}_C - \lambda_{p,n}\mathbbm{1}_{C_\delta}}_{\Delta_{n,\delta}^{(1)}} + \underbrace{\lambda_{p,n}\mathbbm{1}_{C_\delta} - \psi_{\Uinv}^p(\mathbb{U}_n^{-1})\mathbbm{1}_{C_\delta}}_{\Delta_{n,\delta}^{(2)}}
        \\& +\underbrace{\psi_{\Uinv}^p(\mathbb{U}_n^{-1})\mathbbm{1}_{C_\delta} - \psi_{\Uinv}^p(\mathbb{U}_n^{-1})\mathbbm{1}_{C}}_{\Delta_{n,\delta}^{(3)}}.
    \end{align*}
    We treat each of the summands individually. Regarding the first summand, note that $\Lambda_p$ is Lipschitz continuous by \autoref{lemma: properties barycenter}(b). Thus, the assumptions of \autoref{lemma: Lipschitz continuity convergence to 0} are fulfilled for $\Delta_{n,\delta}^{(1)}$ and
    \begin{equation*}
        \lim_{\delta \searrow 0} \limsup_{n \to \infty}\mathbb{P}(\| \lambda_{p,n} \mathbbm{1}_{C \setminus C_\delta}\|_{L^1([\beta,1-\beta])} > \varepsilon) = 0.
    \end{equation*}
    for each $\varepsilon > 0$. Next, we turn our attention to $\Delta_{n,\delta}^{(3)}$. It holds that the map $\psi_{\Uinv(t)}^p$ is Lipschitz continuous for each $t \in (\beta, 1-\beta)$ with uniform Lipschitz constant by \autoref{lemma: properties barycenter}(c). By linearity of $\psi_{\Uinv}^p$,
    \begin{equation*}
        \psi_{\Uinv}^p(\mathbb{U}_n^{-1}) = \sqrt{n} \left(\psi_{\Uinv}^p(\Uninv) - \psi_{\Uinv}^p(\Uinv) \right)
    \end{equation*}
    and thus we can again apply \autoref{lemma: Lipschitz continuity convergence to 0} to show that 
    \begin{equation*}
        \lim_{\delta \searrow 0} \limsup_{n \to \infty}\mathbb{P}(\| \psi_{\Uinv}^p(\mathbb{U}_n^{-1}) \mathbbm{1}_{C \setminus C_\delta}\|_{L^1([\beta,1-\beta])} > \varepsilon) = 0.
    \end{equation*}
    Regarding $\Delta_{n,\delta}^{(2)}$, let
    \begin{equation*}
        E_{n, \delta} \coloneqq \left\{ \omega : \max_{1 \leq i \leq m} |U_{i,n}^{-1}(\omega,t)- U_i^{-1}(t)| < \frac{1}{2} \text{SPD}(U_1^{-1}(t), \ldots, U_m^{-1}(t)), \quad  \forall t \in C_\delta \right\}.
    \end{equation*}
    For $\omega \in E_{n , \delta}, t \in C_\delta$ it holds that 
    \begin{equation} \label{eq: Uninv not all equal}
        \Uninv (\omega, t) \notin D^m \coloneqq \lbrace x \in \mathbb{R}^m: x_1 = \ldots = x_m \rbrace.
    \end{equation}
    This holds as for any $t \in C_\delta$, there is a pair $i \neq j$ such that $\left| U_i^{-1}(t) - U_j^{-1}(t) \right|= \text{SPD}\left(\Uinv(t)\right)$. Now fix this $i,j$ and let $\omega \in E_{n,\delta}$. Then,
    \begin{align*}
        \left| U_{i,n}^{-1}(\omega, t) - U_{j,n}^{-1}(\omega, t)  \right| > 0,
    \end{align*}
    which shows \eqref{eq: Uninv not all equal}.
    Furthermore, on $\left\{ \Uinv(t) : t \in C_\delta \right\}$, $\Lambda_p$ is differentiable and we can write 
    \begin{align*}
        &\left( \lambda_{p,n}- \psi_{\Uinv}^p(\mathbb{U}_n^{-1}) \right) \mathbbm{1}_{C_\delta}\mathbbm{1}_{E_{n,\delta}}  \\
        = &\sqrt{n}\left[ \Lambda_p\left(\Uinv + \left(\Uninv - \Uinv\right)\right)- \Lambda_p\left(\Uinv\right) - \psi_{\Uinv}^p\left(\Uninv - \Uinv\right) \right] \mathbbm{1}_{C_\delta}\mathbbm{1}_{E_{n,\delta}}.
    \end{align*}
    \autoref{lemma: properties barycenter}(d) now implies that 
    \begin{equation*}
        \| \left( \lambda_{p,n}- \psi_{\Uinv}^p(\mathbb{U}_n^{-1}) \right) \mathbbm{1}_{E_{n,\delta}}\|_{C_\delta} = \sqrt{n} o\! \left( O_p (1/\sqrt{n}) \right) = o_p(1).
    \end{equation*}
    On the other hand, the LLN for $U$-quantile processes from \autoref{lemma: LLN for U quantiles} implies $\mathbb{P}\left(E_{n,\delta}^c\right) \to 0$, from which we overall conclude that 
    \begin{equation*}
        \norm{\left( \lambda_{p,n}- \psi_{\Uinv}^p(\mathbb{U}_n^{-1}) \right) }_{C_\delta} \convP 0
    \end{equation*}
    for each $\delta > 0$. Combining our findings for $\Delta_{n,\delta}^{(1)}, \Delta_{n,\delta}^{(2)}$ and $\Delta_{n,\delta}^{(3)}$ yields  
    \begin{equation*}
        \limsup_{n \to \infty} \mathbb{P} \left( \norm{ \left(\lambda_{p,n} - \psi_{\Uinv}^p(\mathbb{U}_n^{-1}) \right)\mathbbm{1}_C }_{L^1([\beta,1-\beta])} > \varepsilon \right)  = 0
    \end{equation*}
    By \autoref{lemma: properties barycenter}(c), $\psi_{\Uinv}^p$ is Lipschitz continuous for each $\Uinv(t)$ where $t \in C$ with uniform Lipschitz constant. Hence, by the continuous mapping theorem and \autoref{lemma: joint convergence inverted u process}
    \begin{equation*}
        \left(\mathbb{U}_n^{-1},\psi_{\Uinv}^p(\mathbb{U}_n^{-1}) \right) \rightsquigarrow  \left( \mathbb{G}, \psi_{\Uinv}^p(\mathbb{G}) \right) \text{ in } \ell^\infty([\beta,1-\beta])^m \times L^1(C),
    \end{equation*}
    and in conjunction with our previous findings this also means that 
    \begin{equation*}
        \left(\mathbb{U}_n^{-1},\lambda_{p,n} \right) \rightsquigarrow  \left( \mathbb{G}, \psi_{\Uinv}^p(\mathbb{G}) \right) \text{ in } \ell^\infty([\beta,1-\beta])^m \times L^1(C).
    \end{equation*}
    On $C^c$, note that $U_1^{-1}(t) = \ldots = U_m^{-1}(t)$ and thus
    \begin{equation*}
        \lambda_{p,n}\mathbbm{1}_{C^c} = \sqrt{n}\Lambda_p(\Uninv-\Uinv)\mathbbm{1}_{C^c} = \Lambda_p(\mathbb{U}_n^{-1})\mathbbm{1}_{C^c}.
    \end{equation*}
    As $\Lambda_p$ is continuous, we thus conclude that 
    \begin{equation*}
        \left(\mathbb{U}_n^{-1},\lambda_{p,n} \right) \rightsquigarrow  \left( \mathbb{G}, \Lambda_p(\mathbb{G}) \right) \text{ in }\ell^\infty([\beta,1-\beta])^m \times L^1(C^c).
    \end{equation*}
    Now clearly $(\mathbb{U}_n^{-1}, \lambda_{p,n})$ as an element of $\ell^\infty([\beta,1-\beta])^m \times L^1([\beta,1-\beta])$ is a sum of the restrictions to $[\beta,1-\beta]^m \times C$ and $[\beta,1-\beta]^m \times C^c$, which implies the statement. \\
    \par \medskip \noindent
    The proof of \autoref{theorem: convergence of barycenter process}(b) can be carried out in exact analogy to the one of part (a), by substituting $[\beta,1-\beta]$ with $(0,1)$ and using the respective parts of \autoref{lemma: joint convergence inverted u process} and \autoref{lemma: Lipschitz continuity convergence to 0} for the case $p \geq 2$ as well as the second part of \autoref{prop: l-th quantile process} for $p = 1$.
\end{proof}
\newpage

\section{Deferred Proofs of the Main Results}\label{appendix: proof of main results}
\subsection{Proof of \autoref{prop: SLB bounds GW}}\label{proof:prop:barycenter-bound} 
Define the set $\mathcal{Q}$ to be the set of all quantile functions belonging to some probability measure on the real line, and further recall the definition of the set of all quantile functions induced by distances of an mm-space $\tilde{\mathcal{Q}}$ from \autoref{rem: comments SLB approach}. Then it holds that  
\begin{align*}
    T_{GW}^p &= \inf_{\mathcal{X} \in \mathcal{G}_w} \frac{1}{m} \sum_{i=1}^m \mathcal{GW}_p^p(\mathcal{X}_i, \mathcal{X}) \\
     &\geq \inf_{\mathcal{X} \in \mathcal{G}_w} \frac{1}{m} \sum_{i=1}^m \mathcal{SLB}_p^p(\mathcal{X}_i, \mathcal{X}) \\
     &= \inf_{U^{-1} \in \tilde{\mathcal{Q}}} \frac{1}{m} \sum_{i=1}^{m} \int_0^1 \abs{U_i^{-1}(t) - U^{-1}(t)}^p \diff t \\
     &\geq \inf_{U^{-1} \in \mathcal{Q}} \frac{1}{m} \sum_{i=1}^{m} \int_0^1 \abs{U_i^{-1}(t) - U^{-1}(t)}^p \diff t = T^p.
\end{align*}
Here, it was used in the first inequality that the SLB bounds the GW distance from below, and in the last step that $\tilde{\mathcal{Q}} \subset \mathcal{Q}$. \qed 
\subsection{Proof of \autoref{theorem: finite sample bound the alternative}}\label{sec: proof finite sample bound alternative}\label{proof:thm:finitesample-bias-alt}
In the following, we denote by $D = \max_{i=1, \ldots,m} \diam(\mathcal{X}_i)$. Note that the diameter of an mm-space is always finite as it is by definition compact. Thus, it holds that for each $t \in (\beta,1-\beta)$,
\begin{equation*}
    \abs{U_{i,n}^{-1}(t) - U_i^{-1}(t)} \leq D, \qquad \abs{\Lambda_p(\Uninv(t)) - \Lambda_p(\Uinv(t))} \leq D,
\end{equation*}
where the second inequality is a consequence of the Lipschitz continuity of $\Lambda_p$, see \autoref{lemma: properties barycenter}. Hence, for $p>1$, the mean value theorem yields, for any fixed $t \in (\beta,1-\beta)$
\begin{align*}
    &\Biggl| \abs{U_{i,n}^{-1}(t) - \Lambda_p\left( \Uninv(t) \right)}^p - \abs{U_{i}^{-1}(t) - \Lambda_p\left( \Uinv(t) \right)}^p \Biggr|\\
    \leq &p (2D)^{p-1} \abs{U_{i,n}^{-1}(t) - U_i^{-1}(t) - \left( \Lambda_p\left(\Uninv(t)\right) -\Lambda_p\left(\Uinv(t)\right) \right)} \\
    \leq & p (2D)^{p-1} \left[ \abs{U_{i,n}^{-1}(t) - U_i^{-1}(t)}+ \abs{ \Lambda_p\left(\Uninv(t)\right) -\Lambda_p\left(\Uinv(t)\right) }\right].
\end{align*}
While for $p = 1$ the mean value theorem is not applicable due to the lack of differentiability of the absolute value at $0$, we want to remark that the above display still holds true due to the reverse triangle inequality. Hence, applying the above inequality and in a second step using Lipschitz continuity of $\Lambda_p$ (see \autoref{lemma: properties barycenter}), we obtain
\begin{align}
    \abs{T_{\beta,n}^p -T_\beta^p} &\leq \frac{p(2D)^{p-1}}{m}\sum_{i=1}^m \int_\beta^{1-\beta} \left[ \abs{U_{i,n}^{-1} - U_i^{-1}}+ \abs{ \Lambda_p\left(\Uninv\right) -\Lambda_p\left(\Uinv\right)  }\right] \nonumber \\
    &\leq \frac{p(2D)^{p-1}}{m}\sum_{i=1}^m \int_0^{1} \left[ \abs{U_{i,n}^{-1} - U_i^{-1}}+ \max_{1 \leq j \leq m} \abs{U_{j,n}^{-1} -U_j^{-1}}\right] \nonumber \\  
    & \leq \frac{p(2D)^{p-1}(m+1)}{m}\sum_{i=1}^m \int_0^{1}  \abs{U_{i,n}^{-1} - U_i^{-1}}     \label{eq: Inequality finite sample bias}
\end{align}
Note that~\eqref{eq: Inequality finite sample bias} can be viewed as a sum of $1$-Wasserstein distances between the empirical and underlying distance-to-distance distributions, and as such we can write 
\begin{equation} \label{eq: finite sample bias wasserstein}
    \abs{T_{\beta,n}^p -T_\beta^p} \leq \frac{p(2D)^{p-1}(m+1)}{m}\sum_{i=1}^m \int_{0}^{D}  \abs{U_{i,n}(t) - U_i(t)} \diff t,
\end{equation}
by \citet[Remark 2.19(iii)]{villani2021topics} and since $\abs{U_{i,n}(t) - U_i(t)} = 0$ for $t<0$ and $t > D$. By the inequality $(x_1 + \ldots + x_m)^q \leq m^{q-1}(x_1^q + \ldots + x_m^q)$ for $q \geq 1$ and $x_1, \ldots, x_m \in \mathbb{R}_{\geq 0}$ and Jensen's inequality,~\eqref{eq: finite sample bias wasserstein} further implies 
\begin{equation*}
    \abs{T_{\beta,n}^p -T_\beta^p}^q \leq \left(p2^{p-1}(m+1) \right)^q D^{pq-1} m^{-1} \sum_{i=1}^m \int_{0}^{D}  \abs{U_{i,n}(t) - U_i(t)}^q \diff t.
\end{equation*}
Taking the expectation on both sides of the above display and in a second step applying Tonelli's theorem leads to 
\begin{equation} \label{eq: bound finite sample risk}
    \mathbb{E} \left[\abs{T_{\beta,n}^p -T_\beta^p}^q \right] \leq \left(p2^{p-1}(m+1) \right)^q D^{pq-1} m^{-1} \sum_{i=1}^m \int_{0}^{D}  \mathbb{E} \left[ \abs{U_{i,n}(t) - U_i(t)}^q \right] \diff t.
\end{equation}
We continue by bounding the expectations in the above integrals in a pointwise manner. To this end, note that for each $t \in \mathbb{R}$, $U_{i,n}(t)$ is a $U$-statistic with kernel function $h_t(x,y) = \mathbbm{1} \!(d_i(x,y) \leq t)$. Hence, Hoeffding‘s inequality for $U$-statistics \citep{Hoeffding01031963} and the fact that $\abs{h_t(x,y)} \leq 1$ implies that 
\begin{equation*}
    \mathbb{P}\left( \abs{U_{i,n}(t) - U_i(t)} \geq x \right) \leq 2 \exp\left( - 2 x^{2} \lfloor n/2 \rfloor \right), \qquad x > 0.
\end{equation*}
Here, $\lfloor y \rfloor$ denotes the largest integer less or equal to $y \in \mathbb{R}$. Hence, we can bound the expectation of the $q$-th power of the $U$-statistic by writing
\begin{align*}
    \mathbb{E} \left[ \abs{U_{i,n}(t) - U_i(t)}^q \right] &= q \int_{0}^{\infty} x^{q-1} \mathbb{P}\left( \abs{U_{i,n}(t) - U_i(t)} \geq x \right) \diff x \\
    &\leq 2q \int_{0}^{\infty} x^{q-1} \exp\left( - 2 x^{2} \lfloor n/2 \rfloor \right) \diff x \\
    &=  q 2^{-q/2}  \lfloor n/2 \rfloor^{-q/2}  \int_{0}^{\infty} x^{q/2 - 1} \exp\left( - x \right) \diff x \\
    &=  2^{1-q/2} \lfloor n/2 \rfloor^{-q/2} \Gamma(1 + q/2),
\end{align*}
where $\Gamma$ denotes the Gamma-function. Plugging in this moment bound into~\eqref{eq: bound finite sample risk}, it follows that 
\begin{equation} \label{eq: second inequality finite sample bias}
    \mathbb{E} \left[\abs{T_{\beta,n}^p -T_\beta^p}^q \right] \leq 2^{qp - 3q/2 + 1} \left(p(m+1) \right)^q D^{pq} \Gamma(1 + q/2) \lfloor n/2 \rfloor ^{-q/2},
\end{equation}
which concludes the proof.
\qed
\subsection{Proof of \autoref{theorem: sample bias of DoD}}\label{proof:thm:finite-sample-bias} 
   Denote by $Y_1, \ldots, Y_n \sim \mu_1$ a sample independent from $\lbrace X_{i,j} | 1 \leq i \leq m; 1 \leq j \leq n \rbrace$. Define the cdf of the distribution of distances of this sample as
    \begin{equation*}
        V_n(t) = \frac{2}{n(n-1)}\sum_{i<j} 1(d_1(Y_i, Y_j) \leq t ).
   \end{equation*}
    Furthermore, denote the quantile function corresponding to $V_n$ by $V_n^{-1}$. By definition, it holds that
    \begin{align*}
        T_{\beta, n}^p &= \underset{V^{-1} \in \mathcal{Q}}{\inf} \frac{1}{m} \sum_{i=1}^m \int_\beta^{1-\beta} \left| U_{i,n}^{-1}(t) - V^{-1}(t)\right|^p \diff t\\
        &\leq \frac{1}{m} \sum_{i=1}^m \int_0^1 \left| U_{i,n}^{-1}(t) - V_n^{-1}(t)\right|^p \diff t,
    \end{align*}
    as $V_n^{-1}$ is clearly contained in the set of quantile functions. By Jensen's inequality, it thus follows that 
    \begin{equation} \label{eq: finite sample bound null}
        \left( T_{\beta, n}^p \right)^q \leq \frac{1}{m} \sum_{i=1}^m \int_0^1 \left| U_{i,n}^{-1}(t) - V_n^{-1}(t)\right|^{pq} \diff t.
    \end{equation}
    Note that $J_{pq}(\mu^{U_1}) < \infty$ by assumption and $\mu^{U_1} = \ldots = \mu^{U_m} = \mu^{V}$, where $\mu^V$ is the measure corresponding to $V$, the underlying distance-to-distance distribution function of $V_n$. Further, a careful consideration of \citet[Theorem 2.5 and Remark A.6]{weitkamp2024distribution} yields that
    \begin{equation*}
        \mathbb{E}\left(\int_0^1 \abs{U_{i,n}^{-1}(t) - V_n^{-1}(t)}^{pq} \diff t \right) \leq \left( \frac{10 pq}{\sqrt{n+2}} \right)^{pq} J_{pq}(\mu^{U_1}).
    \end{equation*}
    Combining this with~\eqref{eq: finite sample bound null}, we have thus shown
    \begin{equation*}
        \mathbb{E}\left[\left(T_{\beta, n}^p \right)^q \right] \leq \left( \frac{10 pq}{\sqrt{n+2}} \right)^{pq} J_{pq}(\mu^{U_1}),
    \end{equation*} 
    which finishes the proof. \qed
\subsection{Convergence under the Null Hypothesis}\label{proof:thm:convergence-null}
In this section, we will present the proof for \autoref{theorem: Convergence under the Null}. The two parts of the theorem will be proven individually, since the case $\beta = 0$ requires more careful consideration. Thus, we start with the case $\beta > 0$.
\subsubsection*{Proof of \autoref{theorem: Convergence under the Null}(a)}
    For $\beta >0$ and under \autoref{cond 1.2} we have by the first part of \autoref{lemma: joint convergence inverted u process} that 
    \begin{equation*}
        \mathbb{U}_n^{-1}  \rightsquigarrow \mathbb{G} = \left( \mathbb{G}_1 , \ldots, \mathbb{G}_m \right) \text{ in }\ell^\infty\left([\beta,1-\beta]\right)^m.
    \end{equation*}
    Now, define the map $\phi: \ell^\infty([\beta,1-\beta])^m \to \mathbb{R}$ via
    \begin{equation*}
        \phi(f_1, \ldots , f_m) = \frac{1}{m} \sum_{i=1}^m \int_\beta^{1-\beta} \left|f_i(x) - \Lambda_p(f_1(x),\ldots, f_m(x)) \right|^p \diff x.
    \end{equation*}
    As the barycenter operator $\Lambda_p$ is continuous by \autoref{lemma: properties barycenter}, the continuity of $\phi$ is readily verified. Furthermore, we can rewrite
    \begin{align*}
         &\frac{n^{p/2}}{m} \sum_{i=1}^m \int_\beta^{1-\beta} \left| U_{i,n}^{-1} - \Lambda_p\left(\Uninv \right)\right|^p \\
         = &\frac{1}{m} \sum_{i=1}^m \int_\beta^{1-\beta} \left|  \sqrt{n}\left(U_{i,n}^{-1} - U_i^{-1} \right) - \sqrt{n} \left( \Lambda_p\left(\Uninv \right) - U_i^{-1} \right) \right|^p \\
         = &\frac{1}{m} \sum_{i=1}^m \int_\beta^{1-\beta} \left|  \sqrt{n}\left(U_{i,n}^{-1} - U_i^{-1} \right) -\Lambda_p\left(\sqrt{n}(U_{1,n}^{-1} - U_1^{-1}),\ldots,\sqrt{n} ( U_{m,n}^{-1} - U_m^{-1}) \right)  \right|^p.
    \end{align*}
    In the second line, we used the fact that $T_\beta^p = 0$, which implies that $U_1^{-1} = \ldots = U_m^{-1}$ on the interval $[\beta,1-\beta]$. Further, in the last equality, it was used that for $x \in \mathbb{R}^m$ and $t \in \mathbb{R}$, it holds that $\Lambda_p(x) -t = \Lambda_p(x_1-t, \ldots,x_m-t)$. Note that the above display can be rewritten in terms of $\phi$ which yields
    \begin{equation*}
        \frac{n^{p/2}}{m} \sum_{i=1}^m \int_\beta^{1-\beta} \left| U_{i,n}^{-1} - \Lambda_p\left(\Uninv \right)\right|^p = \phi \left( 
        \mathbb{U}_n^{-1}\right)
    \end{equation*}
    Now, an application of the continuous mapping theorem leads to
    \begin{equation*}
        \frac{n^{p/2}}{m} \sum_{i=1}^m \int_\beta^{1-\beta} \left| U_{i,n}^{-1} - \Lambda_p(U_{1,n}^{-1},\ldots,U_{m,n}^{-1})\right|^p \rightsquigarrow \phi \left( \mathbb{G}_1, \ldots , \mathbb{G}_m \right),
    \end{equation*}
    which is precisely the limiting distribution stated in the first part of \autoref{theorem: Convergence under the Null}. \qed
\subsubsection*{Proof of \autoref{theorem: Convergence under the Null}(b)}
We follow a similar approach as \citet{weitkamp2024distribution}. Define the map
\begin{equation} \label{eq: def of Xi_n}
    \Xi_n(\beta) \coloneqq \frac{n^{p/2}}{m} \sum_{i=1}^m \int_\beta^{1-\beta} \left| U_{i,n}^{-1} - \Lambda_p(\Uninv)\right|^p
\end{equation}
for $\beta \in [0,1/2]$. In particular, $\Xi_n$ can be seen as a random map into $(\mathcal{C}([0,1/2]), \| \cdot \|_{[0,1/2]})$. This is due to the fact that a metric measure space is compact, which means that the corresponding distance-to-distance distributions have bounded support. Hence, the integrand above is bounded and an application of the dominated convergence theorem with respect to $\beta$ yields the continuity of $\Xi_n: [0,1/2] \to \mathbb{R}$. 
\begin{lemma}\label{lemma: tightness of xi_n}
    Assume that \autoref{cond 1.3} is satisfied with $\theta = p$ and that $\mu^{U_1} = \ldots = \mu^{U_m}$. Then, the sequence $(\Xi_n)_{n \in \mathbb{N}}$ is a tight sequence of random elements in $(\mathcal{C}([0, 1/2]), \|\cdot \|_{[0, 1/2]})$.
\end{lemma}
\begin{proof}
    In order to show measurability of $\Xi_n$, fix $n \in \mathbb{N}$ and $\beta \in [0,1/2]$. By left continuity of the quantile function and the continuity of the pointwise barycenter $\Lambda_p$ (see \autoref{lemma: properties barycenter}), \citet[Proposition 2.33 and Proposition 2.43]{capasso2021introduction} implies that
    \begin{equation*}
        Z_n = \frac{1}{m} \sum_{i=1}^m \abs{U_{i,n}^{-1} - \Lambda_p\left( \Uninv \right)}
    \end{equation*}
    as a function $[\beta,1-\beta] \times \Omega \to \mathbb{R}$ is $\mathcal{B}([\beta,1-\beta]) \otimes \mathcal{A} - \mathcal{B}(\mathbb{R})$ measurable. Then, the measurability of $\Xi_n(\beta)$ for arbitrary $\beta \in [0,1/2)$ follows by an application of \citet[Theorem 18.3]{billingsley1979probability}. The fact that $\Xi_n$ is measurable as a map into $(\mathcal{C}([0,1/2]), \| \cdot \|_{[0,1/2]})$ then follows directly by the measurability of its coordinates $\Xi_n(\beta)$. \\
    The tightness of $\Xi_n$ follows by an application of \citet[Theorem 7.3]{billingsley1999convergence} once we have shown that
    \begin{enumerate}
        \item The sequence $(\Xi_n(0))_{n \in \mathbb{N}}$ is measurable and tight.
        \item Defining the modulus of continuity $\omega(\Xi_n,\delta) = \sup_{\abs{x-y} \leq \delta} \abs{\Xi_n(x)- \Xi_n(y)}$, it holds that \begin{equation} \label{eq: modulus of continuity}
            \lim_{\delta \searrow 0} \limsup_{n \to \infty} \mathbb{P}(\omega(\Xi_n, \delta) > \varepsilon) = 0, \qquad \varepsilon > 0.
        \end{equation}
    \end{enumerate}
    The measurability of $\Xi_n(0)$ has already been confirmed at the beginning of this proof. Further, by \autoref{lemma: J_p is finite}, $J_p(\mu^{U_1})$ is finite. Hence, we may apply \autoref{theorem: sample bias of DoD}, by which
    \begin{equation*} 
        \mathbb{E}(\abs{\Xi_n(0)}) \leq n^{p/2}\left( \frac{10p}{\sqrt{n+2}} \right)^p J_p(\mu^{U_1}).
    \end{equation*}
    The right hand side of the above display is uniformly bounded in $n \in \mathbb{N}$, which proves the tightness of $\Xi_n$. Thus, only~\eqref{eq: modulus of continuity} remains to be shown. Let $x,y \in [0,1/2]$ and assume without loss of generality that $x< y$. Furthermore, let $Y_1, \ldots, Y_n \iid \mu_1$ be independent of $\lbrace X_{i,j}: i = 1, \ldots, m; j = 1, \ldots, n \rbrace$ and denote by $V_n$ the empirical distribution function of the distances $\lbrace d_1(Y_i, Y_j) \rbrace_{1 \leq i < j \leq n}$. Then 
    \begin{align*}
        \abs{\Xi_n(x) - \Xi_n(y)} &= \frac{n^{p/2}}{m}\sum_{i=1}^m \int_x^y \abs{U_{i,n}^{-1} - \Lambda_p\left(\Uninv \right)}^p + \int_{1-y}^{1-x} \abs{U_{i,n}^{-1} - \Lambda_p\left(\Uninv \right)}^p \\
        &\leq \frac{n^{p/2}}{m}\sum_{i=1}^m \int_x^y \abs{U_{i,n}^{-1} - V_n^{-1}}^p + \int_{1-y}^{1-x} \abs{U_{i,n}^{-1} - V_n^{-1}}^p \\
        &\eqqcolon \frac{n^{p/2}}{m} \sum_{i=1}^m \omega_{i,n}(x,y).
    \end{align*}
    Here, the inequality in the second line is a consequence of the fact that by definition, the barycenter operator $\Lambda_p$ is the pointwise minimizer of the integrand. Using these findings, it follows that 
    \begin{align*}
        \mathbb{P} \left( \omega(\Xi_n,\delta) > \varepsilon \right) &\leq \mathbb{P}\left( \frac{1}{m} \sum_{i=1}^m \sup_{\abs{x-y} \leq \delta} \omega_{i,n}(x,y) > \varepsilon\right) \\
        & \leq \sum_{i=1}^m \mathbb{P}\left( \sup_{\abs{x-y} \leq \delta} \omega_{i,n}(x,y) > \varepsilon \right) \\
        &= \sum_{i=1}^m \mathbb{P} \left( \omega(\Xi_{i,n}, \delta) > \varepsilon \right),
    \end{align*}
    where $\Xi_{i,n}:[0,1/2] \to \mathbb{R}$ is defined as the random map
    \begin{equation*}
        \Xi_{i,n}(\beta) = n^{p/2} \int_\beta^{1-\beta} \abs{U_{i,n}^{-1} - V_n^{-1}}^p.
    \end{equation*}
    By a straightforward extension of \citet[Lemma A.13]{weitkamp2024distribution} to general $p \geq 1$, it holds that for each $i = 1, \ldots, m$
    \begin{equation*}
        \lim_{\delta \searrow 0} \limsup_{n \to \infty} \mathbb{P} \left( \omega(\Xi_{i,n}, \delta) > \varepsilon \right) = 0,
    \end{equation*}
    when \autoref{cond 1.3} is met with $\theta = p$. Hence, we conclude that
    \begin{align*}
        \lim_{\delta \searrow 0} \limsup_{n \to \infty} \mathbb{P} \left( \omega(\Xi_n,\delta) > \varepsilon \right) \leq \sum_{i=1}^m \lim_{\delta \searrow 0} \limsup_{n \to \infty} \mathbb{P} \left( \omega(\Xi_{i,n}, \delta) > \varepsilon \right) = 0.
    \end{align*}
\end{proof}
Equipped with this result, we now give the proof of the second part of \autoref{theorem: Convergence under the Null}.
\begin{proof}[Proof of  \autoref{theorem: Convergence under the Null}(b)]
    As stated in \autoref{lemma: tightness of xi_n}, $(\Xi_n)_{n \in \mathbb{N}}$ is a tight sequence of random elements in $(\mathcal{C}([0,1/2]), \| \cdot \|_{[0,1/2]})$. Hence, by Prohorov's theorem \citep[Theorem 1.3.9]{wellner2013weak}, there is a weakly convergent subsequence $(\Xi_{a_k})_{k \in \mathbb{N}}$, the tight limit of which we denote by $\Xi$. Note that a slight adjustment of the proof of \autoref{theorem: Convergence under the Null}(a) yields that, for any collection of indices $t_1, \ldots, t_k \in (0,1/2)$, 
    \begin{equation*}
        (\Xi(t_1), \ldots, \Xi(t_k)) \eqD (X(t_1), \ldots, X(t_k)),
    \end{equation*}
    where $X(t_i) = 1/m \sum_{i=1}^m \int_\beta^{1-\beta} | \mathbb{G}_i^{-1} - \Lambda_p(\mathbb{G}_1^{-1},\ldots,\mathbb{G}_m^{-1})|^p$. By \citet[Lemma 1.5.3]{wellner2013weak}, this uniquely determines the distribution of $\Xi$ on $(0,1)$. Further, tightness of $\Xi$ implies separability \citep[Lemma 1.3.2]{wellner2013weak}, and hence we may apply the Skorohod representation theorem \citep[Theorem 6.7]{billingsley1999convergence} in order to find a probability space $(\Tilde{\Omega}, \Tilde{\mathcal{A}}, \Tilde{\mathbb{P}})$ and random elements $(\Tilde{\Xi}_{a_k})_{k \in \mathbb{N}}, \Tilde{\Xi}$ on this probability space such that $\Tilde{\Xi}_{a_k} \stackrel{D}{=} \Xi_{a_k}$ for each $k \in \mathbb{N}$ as well as $\Tilde{\Xi} \stackrel{D}{=} \Xi$ with the property that $\Tilde{\Xi}_{a_k}(\omega) \to \Tilde{\Xi}(\omega)$ in $(\mathcal{C}([0,1/2]), \| \cdot \|_{[0,1/2]})$ for almost every $\omega \in \Tilde{\Omega}$. From now on we will investigate only the Skorohod representation of the random elements, and hence we will drop the tilde for the ease of notation. Note that $\Xi_{a_k}(\omega) \to \Xi(\omega)$ for each $\omega \in \Omega$ in $(\mathcal{C}([0,1/2]), \| \cdot \|_{[0,1/2]})$ particularly implies that $\Xi_{a_k}(\beta)(\omega) \to \Xi(\beta)(\omega)$ pointwise for each $\beta \in [0,1/2]$ and almost every $\omega \in \Omega$. Combining these facts with part one of \autoref{theorem: Convergence under the Null}, it holds for any $\beta > 0$ and almost every $\omega \in \Omega$ that 
    \begin{equation} \label{eq: pointwise convergence of Xi_n}
         \frac{a_k^{p/2}}{m} \sum_{i=1}^m \int_\beta^{1-\beta} \left| U_{i,a_k}^{-1}(t) - \Lambda_p(\Uinv_{a_k}(t))\right|^p\diff t \to \frac{1}{m} \sum_{i=1}^m \int_\beta^{1-\beta} \left| \mathbb{G}_i(t) -\Lambda_p(\mathbb{G}(t))\right|^p \diff t.
    \end{equation}
    Furthermore, by definition, $\Xi_{a_k}(\cdot)$ is monotonically decreasing for each fixed $\omega$ and $\Xi_{a_k}(0) \geq 0$. Together with~\eqref{eq: pointwise convergence of Xi_n} this implies
    \begin{align*}
        \Xi(0) = \lim_{k \to \infty} \Xi_{a_k}(0) \geq \lim_{k \to \infty} \Xi_{a_k}(\beta) = \frac{1}{m} \sum_{i=1}^m \int_\beta^{1-\beta} \left| \mathbb{G}_i(t) -\Lambda_p(\mathbb{G}(t))\right|^p \diff t
    \end{align*}
    for each $\beta \in (0,1/2]$. Taking the limit of $\beta \searrow 0$, we thus conclude that for each $\omega \in \Omega$
    \begin{equation*}
        \Xi(0) \geq \frac{1}{m} \sum_{i=1}^m \int_0^{1}\left| \mathbb{G}_i(t) -\Lambda_p(\mathbb{G}(t))\right|^p  \diff t,
    \end{equation*}
    by an application of the monotone convergence theorem. We will now show that the set of elements in $\Omega$ for which the above inequality becomes strict is a null-set. To this end, consider a specific $\omega \in \Omega$ s.t.
    \begin{equation} \label{eq: strict ineq}
        \lim_{k \to \infty} \frac{a_k^{p/2}}{m} \sum_{i=1}^m \int_0^{1} \left| U_{i,a_k}^{-1} - \Lambda_p(\Uinv_{a_k})\right|^p = \Xi(0) > \frac{1}{m} \sum_{i=1}^m \int_0^{1} \left| \mathbb{G}_i -\Lambda_p(\mathbb{G})\right|^p.
    \end{equation}
    Combining~\eqref{eq: pointwise convergence of Xi_n} and~\eqref{eq: strict ineq} yields, for any $\beta > 0$
    \begin{equation} \label{eq: inquality in proof beta=0}
         \lim_{k \to \infty} \frac{a_k^{p/2}}{m} \sum_{i=1}^m \int_{[0, \beta] \cup [1-\beta, 1]} \left| U_{i,a_k}^{-1} - \Lambda_p(\Uinv_{a_k})\right|^p  > \frac{1}{m} \sum_{i=1}^m \int_{[0, \beta] \cup [1-\beta, 1]} \left| \mathbb{G}_i -\Lambda_p(\mathbb{G})\right|^p.
    \end{equation}
    More specifically, due to the fact that~\eqref{eq: pointwise convergence of Xi_n} holds, the difference in~\eqref{eq: inquality in proof beta=0} is constant in $\beta$, i.e.\ we can choose $\Delta > 0 $ depending only on $\omega$ such that 
    \begin{equation*}
        \lim_{k \to \infty} \frac{a_k^{p/2}}{m} \sum_{i=1}^m \int_{[0, \beta] \cup [1-\beta, 1]} \left| U_{i,a_k}^{-1} - \Lambda_p(\Uinv_{a_k})\right|^p  = \frac{1}{m} \sum_{i=1}^m \int_{[0, \beta] \cup [1-\beta, 1]} \left| \mathbb{G}_i -\Lambda_p(\mathbb{G})\right|^p + \Delta,
    \end{equation*}
    for each $\beta > 0$. Hence
    \begin{equation*}
        \lim_{k \to \infty} \frac{a_k^{p/2}}{m} \sum_{i=1}^m \int_{[0, \beta] \cup [1-\beta, 1]} \left| U_{i,a_k}^{-1} - \Lambda_p(\Uinv_{a_k})\right|^p \geq \Delta,
    \end{equation*}
    independently of the choice $\beta \in (0,1/2)$. Taking the infimum on both sides leads to
    \begin{equation*}
        \inf_{0< \beta \leq 1/2} \lim_{k \to \infty} \frac{a_k^{p/2}}{m} \sum_{i=1}^m \int_{[0, \beta] \cup [1-\beta, 1]} \left| U_{i,a_k}^{-1} - \Lambda_p(\Uinv_{a_k})\right|^p \geq \Delta.
    \end{equation*}
    The term that the infimum is taken over is monotonically increasing in $\beta$ and as such one can write
    \begin{align*}
         \mathbb{I}\coloneqq &\lim_{\beta \searrow 0 } \liminf_{k \to \infty} \frac{a_k^{p/2}}{m} \sum_{i=1}^m \int_{[0, \beta] \cup [1-\beta, 1]} \left| U_{i,a_k}^{-1} - \Lambda_p(\Uinv_{a_k})\right|^p \\
         = &\inf_{0< \beta \leq 1/2} \lim_{k \to \infty} \frac{a_k^{p/2}}{m} \sum_{i=1}^m \int_{[0, \beta] \cup [1-\beta, 1]} \left| U_{i,a_k}^{-1} - \Lambda_p(\Uinv_{a_k})\right|^p \geq \Delta.
    \end{align*}
   Summarizing the results obtained thus far, we have established that if~\eqref{eq: strict ineq} holds for some $\omega \in \Omega$, then it is contained in the set
    \begin{equation} \label{eq: set inclusion}
         A \coloneqq \lbrace \omega \in \Omega : \mathbb{I}(\omega) >0 \rbrace.
    \end{equation}
    We proceed by showing that $A$ is a null set. Note that
    \begin{align*}
        \mathbb{E}(\mathbb{I})=&\mathbb{E}\left[ \lim_{\beta \searrow 0 } \liminf_{k \to \infty} \frac{a_k^{p/2}}{m} \sum_{i=1}^m \int_{[0, \beta] \cup [1-\beta, 1]} \left| U_{i,a_k}^{-1} - \Lambda_p(\Uinv_{a_k})\right|^p \right] \\
        = &  \lim_{\beta \searrow 0 }\mathbb{E}\left[ \liminf_{k \to \infty} \frac{a_k^{p/2}}{m} \sum_{i=1}^m \int_{[0, \beta] \cup [1-\beta, 1]} \left| U_{i,a_k}^{-1} - \Lambda_p(\Uinv_{a_k})\right|^p \right] \\
        \leq &\lim_{\beta \searrow 0 }\liminf_{k \to \infty} \mathbb{E}\left[ \frac{a_k^{p/2}}{m} \sum_{i=1}^m \int_{[0, \beta] \cup [1-\beta, 1]} \left| U_{i,a_k}^{-1} - \Lambda_p(\Uinv_{a_k})\right|^p \right]
    \end{align*}
   by an application of the monotone convergence theorem and Fatou's lemma. Now, once again consider $Y_1, \ldots, Y_{a_k} \iid \mu_1$ independent of $\lbrace X_{i,j}: i = 1, \ldots, m; j = 1, \ldots, {a_k} \rbrace$ and denote by $V_{a_k}$ the empirical distribution function of the distances $\lbrace d_1(Y_i, Y_j) \rbrace_{1 \leq i < j \leq a_k}$. Then
    \begin{align*}
        \mathbb{E}\left[ \frac{a_k^{p/2}}{m} \sum_{i=1}^m \int_{[0, \beta] \cup [1-\beta, 1]} \left| U_{i,a_k}^{-1} - \Lambda_p(\Uinv_{a_k})\right|^p \right] &\leq \mathbb{E}\left[ \frac{a_k^{p/2}}{m} \sum_{i=1}^m \int_{[0, \beta] \cup [1-\beta, 1]} \left| U_{i,a_k}^{-1} - V_{a_k}^{-1}\right|^p \right] \\
        & \leq g_{a_k}(\beta),
    \end{align*}
    where $g_{n}(\beta)$ can be chosen such that $\lim_{\beta \searrow 0} \limsup_{n \to \infty} g_n(\beta) = 0$ by \citet[Lemma A.12 and Remark A.16]{weitkamp2024distribution}. Thus, we have derived
    \begin{equation*}
        \mathbb{E}(\mathbb{I}) \leq \lim_{\beta \searrow 0 }\liminf_{k \to \infty} \mathbb{E}\left[ \frac{a_k^{p/2}}{m} \sum_{i=1}^m \int_{[0, \beta] \cup [1-\beta, 1]} \left| U_{i,a_k}^{-1} - \Lambda_p(U_{a_k}^{-1})\right|^p \right] = 0.
    \end{equation*} 
    Hence, $A$ is a null set, from which we conclude together with~\eqref{eq: set inclusion} that for almost every $\omega \in \Omega$, the equality
    \begin{equation*}
        \Xi(0) = \frac{1}{m} \sum_{i=1}^m \int_0^{1} |\mathbb{G}_i(t) - \Lambda_p(\mathbb{G}(t))|^p \diff t
    \end{equation*}
    holds. To be more precise and to reintroduce the notation of the Skorohod representation, we have shown that 
    \begin{equation*}
        \Tilde{\Xi}(0) \stackrel{D}{=} \frac{1}{m} \sum_{i=1}^m \int_0^{1} |\mathbb{G}_i(t) - \Lambda_p(\mathbb{G}(t))|^p \diff t
    \end{equation*}
    and due to the fact that $\Tilde{\Xi}(0) \stackrel{D}{=} \Xi(0)$ this implies that 
    \begin{equation*}
        \Xi_{a_k}(0) \rightsquigarrow \frac{1}{m} \sum_{i=1}^m \int_0^{1} |\mathbb{G}_i(t) - \Lambda_p(\mathbb{G}(t))|^p \diff t.
    \end{equation*}
    As the subsequence was chosen arbitrarily, the same limit holds for every convergent subsequence, which concludes the proof.
\end{proof}
\subsection{Convergence under the Alternative}\label{proof:thm:convergence-alt}
In this section, we provide the proof for \autoref{theorem: Convergence under the alt}. The proof is again split between the first and second part of the theorem. 
\subsubsection*{Proof of \autoref{theorem: Convergence under the alt}(a)}
In order to provide a distributional limit for $\sqrt{n}\left(T_{\beta,n}^p - T_\beta^p \right)$, we subdivide the interval $(\beta,1-\beta)$ into different subsets and treat each of the subsets individually. To this end, set for each $i=1,\ldots, m$
    \begin{equation*}
        C_{i, \delta}^> \coloneqq \left\{ t \in (\beta,1-\beta): |U_i^{-1}(t) - \Lambda_p(\Uinv(t))| > \delta \right\}.
    \end{equation*}
    In order to ease the notation, in the following we will write
    \begin{equation*}
        \Delta_i \coloneqq U_i^{-1} - \Lambda_p\left(\Uinv \right) 
    \end{equation*} and let $\Delta_{i,n} \coloneqq  U_{i,n}^{-1} - \Lambda_p(\Uninv)$ be the empirical counterpart.
\begin{lemma}\label{lemma: convergence alternative > subset}
    Let $\beta >0$, $p = 1$ or $p \geq 2$ and $\delta > 0$. Then under \autoref{cond 1.2} it holds that
    \begin{equation*}
        \frac{\sqrt{n}}{m}\sum_{i=1}^m \int_{C_{i, \delta}^>} |\Delta_{i,n}|^p - |\Delta_i|^p \\
        \rightsquigarrow \frac{1}{m} \sum_{i=1}^m \int_{C_{i, \delta}^>}  p  |\Delta_i|^{p-1}(\mathbb{G}_i - \psi_{\Uinv}^p(\mathbb{G}))\sgn(\Delta_i).
    \end{equation*}
    where $\left( \mathbb{G}, \psi_{\Uinv}^p(\mathbb{G}) \right)$ is the Gaussian process defined in \autoref{theorem: convergence of barycenter process}.
\end{lemma}
\begin{proof}
   We begin by defining
    \begin{equation*}
        E_{n, \delta} \coloneqq \left\{ \omega: \left\| \Delta_{i,n}(\omega, \cdot)- \Delta_i(\cdot) \right\|_{C_{i, \delta}^>} < \frac{\delta}{2} \text{ for }i = 1, \ldots, m\right\}.
    \end{equation*}
    By a Taylor expansion with Lagrange remainder on the integrand, we can write
    \begin{align}
        \left(\frac{\sqrt{n}}{m}\sum_{i=1}^m \int_{C_{i, \delta}^>} |\Delta_{i,n}|^p - |\Delta_i|^p \right)\mathbbm{1} \!_{E_{n,\delta}}& = \left(\frac{\sqrt{n}}{m}\sum_{i=1}^m \int_{C_{i, \delta}^>} p|\Delta_i|^{p-1} \sgn(\Delta_i)(\Delta_{i,n} - \Delta_i) \right)\mathbbm{1} \!_{E_{n,\delta}} \label{eq: equation in proof delta superset} \\
        &+ \left(\frac{\sqrt{n}}{m}\sum_{i=1}^m \int_{C_{i, \delta}^>} \frac{g(\zeta_i)}{2} (\Delta_{i,n} - \Delta_i)^2 \right)\mathbbm{1} \!_{E_{n,\delta}}, \nonumber
    \end{align}
    where $g: \mathbb{R}\setminus\lbrace0\rbrace \to \mathbb{R}$ is given by $g(x) = p(p-1)|x|^{p-2}$ and \begin{equation*}
        \zeta_i(\omega,t) \in \left(\min \lbrace |\Delta_{i,n}(\omega,t)|,| \Delta_i(t)| \rbrace, \max \lbrace |\Delta_{i,n}(\omega,t)|, |\Delta_i(t)| \rbrace \right).
    \end{equation*}
    We now establish that the second part of the sum in~\eqref{eq: equation in proof delta superset} converges to $0$ in probability. By the construction of $E_{n, \delta}$ and the definition of $C_{i,\delta}^>$, there is a constant $K \in \mathbb{R}$ such that \begin{equation} \label{eq: max and sup less than constant}
        \max_{i=1,\ldots,m} \sup_{t \in C_{i,\delta}^>} |g(\zeta_i(\omega,t))| < K
    \end{equation}
    for any $\omega \in E_{n,\delta}$. Furthermore, for each $1 \leq i \leq m$,
    \begin{align}
        \sqrt{n}\int_{C_{i,\delta}^>} (\Delta_{i,n}-\Delta_i)^2 &\leq \| \Delta_{i,n}-\Delta_i \|_{[\beta,1-\beta]} \sqrt{n}\int_{C_{i,\delta}^>}\abs{\Delta_{i,n}-\Delta_i} \label{eq: o_p(1) term in proof} \\
        &= o_p(1)O_p(1) = o_p(1) \nonumber,
    \end{align}
    where the first factor is $o_p(1)$ due to the LLN given in \autoref{lemma: LLN for U quantiles} and the integral is tight, since $\sqrt{n}(\Delta_{i,n}-\Delta_i)$ converges weakly in $L^1([\beta,1-\beta])$ by \autoref{theorem: convergence of barycenter process}. Putting together~\eqref{eq: max and sup less than constant} and~\eqref{eq: o_p(1) term in proof}, we can thus conclude that 
    \begin{equation*}
        \left(\frac{\sqrt{n}}{m}\sum_{i=1}^m \int_{C_{i, \delta}^>} \frac{g(\zeta_i)}{2} (\Delta_{i,n} - \Delta_i)^2 \right)\mathbbm{1} \!_{E_{n,\delta}} = o_p(1).
    \end{equation*}
    Next, we treat the first summand of~\eqref{eq: equation in proof delta superset}. Define the map $\phi: \ell^\infty([\beta,1-\beta])^{m} \times L^1([\beta,1-\beta]) \to \mathbb{R}$ via
    \begin{equation*}
        \phi(f_1,\ldots,f_{m+1}) = \frac{1}{m} \sum_{i=1}^m \int_{C_{i,\delta}^>} p |\Delta_i|^{p-1}\sgn(\Delta_i)(f_i-f_{m+1}).
    \end{equation*}
    This map is clearly continuous. Furthermore, \begin{equation} \label{eq: limit law in proof alternative}
        \left( \mathbb{U}_n^{-1},\lambda_{p,n} \right)\mathbbm{1} \!_{E_{n,\delta}} \rightsquigarrow \left(\mathbb{G}, \psi_{\Uinv}^p(\mathbb{G})\right) \text{ in }\ell^\infty([\beta,1-\beta])^{m} \times L^1([\beta,1-\beta])
    \end{equation} 
    by \autoref{theorem: convergence of barycenter process} and the LLN from \autoref{lemma: LLN for U quantiles}, which implies that $\mathbb{P} \left(E_{n,\delta} \right) \to 1$. We may apply continuous mapping theorem, using the limit law~\eqref{eq: limit law in proof alternative} with the map $\phi$ to obtain that 
    \begin{equation*}
        \left(\frac{\sqrt{n}}{m}\sum_{i=1}^m \int_{C_{i, \delta}^>} |\Delta_{i,n}|^p - |\Delta_i|^p \right)\mathbbm{1} \!_{E_{n,\delta}} \rightsquigarrow \frac{1}{m} \sum_{i=1}^m \int_{C_{i, \delta}^>}  p  |\Delta_i|^{p-1}(\mathbb{G}_i - \psi_{\Uinv}^p(\mathbb{G}))\sgn(\Delta_i).
    \end{equation*}
    Since $\mathbb{P} \left(E_{n,\delta} \right) \to 1$, the statement follows upon an application of Slutzky's Lemma \citep[Lemma 1.10.2(i)]{wellner2013weak}.
\end{proof}
Next, we define the sets 
\begin{equation*}
    C_i^= = \left\{ t \in (\beta,1-\beta):U_i^{-1}(t) = \Lambda_p\left(\Uinv(t)\right) \right\}
\end{equation*}
for $i = 1, \ldots,m$. On these sets, we distinguish the analysis for $p \geq 2$ and $p = 1$.
\begin{lemma}\label{lemma: convergence alternative on = subset p=2}
    Let $p \geq 2$ and $1 \leq i \leq m$. For $\beta >0$ and under \autoref{cond 1.2}, it holds that 
        \begin{equation*}
            \sqrt{n} \int_{C_i^=} \left( \abs{\Delta_{i,n}}^p - \abs{\Delta_i}^p \right) = o_p(1).
        \end{equation*}
\end{lemma}
\begin{proof}
    As the integral is restricted to $C_i^=$ and $p \geq 2$, we can rewrite
    \begin{align} \label{eq: 1st equality C= set}
        \abs{ \sqrt{n} \int_{C_i^=} \left( \abs{\Delta_{i,n}}^p - \abs{\Delta_i}^p \right)} = \abs{ \sqrt{n} \int_{C_{i}^=}\abs{\Delta_{i,n}}^p }
        \leq \norm{\Delta_{i,n}}_{C_i^=}^{p-1} \int_{C_{i}^=} \abs{ \sqrt{n}\Delta_{i,n}}.
    \end{align}
    Now it holds that 
    \begin{equation} \label{eq: second equality C= set}
        \norm{\Delta_{i,n}}_{C_i^=} = \norm{U_{i,n}^{-1} - U_i^{-1} - \Lambda_p\left(\Uninv\right)+ \Lambda_p\left(\Uinv\right)}_{C_i^=} = o_p(1),
    \end{equation}
    by the first part of \autoref{lemma: LLN for U quantiles}, which holds under \autoref{cond 1.2}. Moreover, we may apply \autoref{theorem: convergence of barycenter process} and thus on $C_i^=$ we observe
    \begin{equation*}
        \sqrt{n}\Delta_{i,n} = \sqrt{n}\left( \Delta_{i,n} - \Delta_i \right)\rightsquigarrow \mathbb{G}_i- \psi_{\Uinv}^p(\mathbb{G}) \text{ in }L^1(C_i^=).
    \end{equation*}
    Thus it follows that 
    \begin{equation} \label{eq: third equality C= set}
        \int_{C_i^=} \abs{\sqrt{n}\Delta_{i,n}} = O_p(1).
    \end{equation}
    Putting together~\eqref{eq: 1st equality C= set},~\eqref{eq: second equality C= set} and~\eqref{eq: third equality C= set}, we conclude 
    \begin{equation*}
        \sqrt{n} \int_{C_i^=} \left( \abs{\Delta_{i,n}}^p - \abs{\Delta_i}^p \right) = o_p(1)O_p(1) = o_p(1).
    \end{equation*}
\end{proof}
\begin{lemma}\label{lemma: convergence alternative on = subset p=1}
    Let $p=1$. For $\beta>0$ and under \autoref{cond 1.2}
    \begin{equation*}
        \frac{\sqrt{n}}{m}\sum_{i=1}^m \int_{C_i^=} \abs{\Delta_{i,n}} - \abs{\Delta_i}\rightsquigarrow \frac{1}{m}\sum_{i=1}^m\int_{C_i^=} \abs{\mathbb{G}_i - \psi_{\Uinv}^1(\mathbb{G})},
    \end{equation*}
    where $\mathbb{G}$ is the Gaussian process from \autoref{lemma: joint convergence inverted u process}.
\end{lemma}
\begin{proof}
    By \autoref{theorem: convergence of barycenter process},
    \begin{equation*}
        \sqrt{n}\Delta_{i,n}\rightsquigarrow \mathbb{G}_i- \psi_{\Uinv}^1(\mathbb{G}) \text{ in }L^1(C_i^=)
    \end{equation*}
    jointly for $i = 1, \ldots,m$. Hence, applying the continuous mapping theorem using $\phi: L^1(C_1^=) \times \ldots \times L^1(C_m^=) \to \mathbb{R}$, 
    \begin{equation*}
        \phi(f_1, \ldots, f_m) = \frac{1}{m} \sum_{i=1}^m \int_{C_i^=} |f_i|
    \end{equation*}
    yields the result.
\end{proof}
Finally, it remains to analyze the sets 
\begin{equation*}
    C_{i,\delta}^\leq = \left\{  t \in (\beta,1-\beta) : 0 < \abs{U_i^{-1}(t) - \Lambda_p\left(\Uinv(t)\right)} \leq \delta \right\}
\end{equation*}
for $i = 1, \ldots, m$. To this end, define the random variable
\begin{equation} \label{eq: lower bound Cdelta alternative}
        \Xi_{\delta,n}^p =\frac{\sqrt{n}}{m}\sum_{i=1}^m \int_{C_{i, \delta}^\leq} p\max\left\{ \abs{\Delta_{i,n}}^{p-1}, \abs{\Delta_i}^{p-1} \right\} \abs{\Delta_{i,n} - \Delta_i}.
    \end{equation}
\begin{lemma}\label{lemma: convergence alternative on < subset}
    Let $\delta > 0$ and $p \geq 1$. Further let $\Xi_{\delta,n}^p$ as defined in~\eqref{eq: lower bound Cdelta alternative}. Then for $\beta>0$ and under \autoref{cond 1.2}, we have that pointwise in $\omega$
    \begin{equation} \label{eq: omega-wise boundedness}
        - \Xi_{\delta,n}^p \leq \frac{\sqrt{n}}{m}\sum_{i=1}^m \int_{C_{i, \delta}^\leq} |\Delta_{i,n}|^p - |\Delta_i|^p \leq \Xi_{\delta,n}^p,
    \end{equation}
    and furthermore
    \begin{equation} \label{eq: limit of omega wise bound}
        \Xi_{\delta,n}^p \rightsquigarrow \frac{1}{m}\sum_{i=1}^m \int_{C_{i, \delta}^\leq} p\abs{\Delta_i}^{p-1} \abs{\mathbb{G}_i - \psi_{\Uinv}^p(\mathbb{G})}.
    \end{equation}
\end{lemma}
\begin{proof}
    First, consider $p \geq 2$. Then by the mean value theorem we have that for each $\omega$ and $t \in C_{i,\delta}^\leq$
    \begin{equation*}
        \abs{|\Delta_{i,n}|^p - |\Delta_i|^p} \leq p\max\left\{ \abs{\Delta_{i,n}}^{p-1}, \abs{\Delta_i}^{p-1} \right\} \abs{\Delta_{i,n} - \Delta_i},
    \end{equation*}
    which implies~\eqref{eq: omega-wise boundedness}. For $p =1$,~\eqref{eq: omega-wise boundedness} follows by the reverse triangle inequality. Hence, it remains to prove~\eqref{eq: limit of omega wise bound}. Under \autoref{cond 1.2} and for $\beta > 0$, \autoref{theorem: convergence of barycenter process} holds and thus 
    \begin{equation*}
        \sqrt{n} \left( \Delta_{i,n} - \Delta_i \right)\rightsquigarrow \mathbb{G}_i - \psi_{\Uinv}^p(\mathbb{G}) \text{ in }L^1([\beta,1-\beta])
    \end{equation*}
    jointly for $i = 1, \ldots,m$. Define the map $\phi:L^1([\beta,1-\beta])^m \to \mathbb{R}$ via
    \begin{equation*}
        \phi(f_1, \ldots, f_m) = \frac{1}{m} \sum_{i=1}^m \int_{C_{i,\delta}^\leq} p\abs{\Delta_i}^{p-1} \abs{f_i}.
    \end{equation*}
    Clearly, this map is continuous and thus
    \begin{equation*}
        \phi\left[\sqrt{n} \left( \Delta_{1,n} - \Delta_1 \right), \ldots, \sqrt{n} \left( \Delta_{m,n} - \Delta_m \right) \right] \rightsquigarrow \phi\left[ \mathbb{G}_1 - \psi_{\Uinv}^p(\mathbb{G}), \ldots, \mathbb{G}_m - \psi_{\Uinv}^p(\mathbb{G}) \right].
    \end{equation*}
    Finally, realize that  
    \begin{eqnarray*}
        \abs{\Xi_{\delta,n}^p - \phi\left[\sqrt{n} \left( \Delta_{1,n} - \Delta_1 \right), \ldots, \sqrt{n} \left( \Delta_{m,n} - \Delta_m \right) \right]} \\
        \leq \frac{\sqrt{n}}{m}\sum_{i=1}^m \Bigl\| \abs{\Delta_{i,n}}^{p-1} - \abs{\Delta_i}^{p-1} \Bigr\|_{C_{i,\delta}^\leq} \int_{C_{i, \delta}^\leq} p \abs{\Delta_{i,n} - \Delta_i},
    \end{eqnarray*}
    which converges to $0$ in probability by \autoref{lemma: LLN for U quantiles}.
\end{proof}
With these results at our disposal, we are now ready to show \autoref{theorem: Convergence under the alt}(a)
\begin{proof}[Proof of \autoref{theorem: Convergence under the alt}(a)]
    Throughout this proof, write $X_{\beta,n}^p = \sqrt{n}(T_{\beta,n}^p - T_\beta^p)$. Further, recall the definition of $\Xi_{\delta,n}^p$ in~\eqref{eq: lower bound Cdelta alternative}. \autoref{lemma: convergence alternative on < subset} shows that $\omega$-wise
    \begin{equation}\label{eq: sandwich in proof alternative}
        \begin{aligned}
        \Xi_{\delta,n}^{p, \leq} \coloneqq -\Xi_{\delta,n}^p +\frac{\sqrt{n}}{m}\sum_{i=1}^m \int_{C_{i, \delta}^> \cup C_i^=}  |\Delta_{i,n}|^p - |\Delta_i|^p  \leq X_{\beta,n}^p   \\
        \leq \Xi_{\delta,n}^p +\frac{\sqrt{n}}{m}\sum_{i=1}^m \int_{C_{i, \delta}^> \cup C_i^=} |\Delta_{i,n}|^p - |\Delta_i|^p \eqqcolon \Xi_{\delta,n}^{p, \geq}. 
        \end{aligned}
    \end{equation}
    We first restrict ourselves to $p \geq 2$. Then a combination of the limit laws found in \autoref{lemma: convergence alternative > subset}, \autoref{lemma: convergence alternative on = subset p=2} and \autoref{lemma: convergence alternative on < subset} yields
    \begin{equation} \label{eq: lower limit law alternative proof}
        \Xi_{\delta,n}^{p, \leq} \rightsquigarrow \Xi_\delta^{p, \leq}
    \end{equation}
    where 
    \begin{equation}\label{eq: upper limit law alternative proof}
        \Xi_\delta^{p, \leq} = \frac{1}{m} \sum_{i=1}^m \left[ \int_{C_{i,\delta}^\leq} -p\abs{\Delta_i}^{p-1} \abs{\mathbb{G}_i - \psi_{\Uinv}^p(\mathbb{G})} + \int_{C_{i,\delta}^>}   p  |\Delta_i|^{p-1}(\mathbb{G}_i - \psi_{\Uinv}^p(\mathbb{G}))\sgn(\Delta_i)\right].
    \end{equation}
    Likewise, it holds that 
    \begin{equation*}
        \Xi_{\delta,n}^{p, \geq} \rightsquigarrow \Xi_{\delta}^{p, \geq},
    \end{equation*}
    where 
    \begin{equation*}
        \Xi_\delta^{p, \geq } = \frac{1}{m} \sum_{i=1}^m \left[ \int_{C_{i,\delta}^\leq} p\abs{\Delta_i}^{p-1} \abs{\mathbb{G}_i - \psi_{\Uinv}^p(\mathbb{G})} + \int_{C_{i,\delta}^>}   p  |\Delta_i|^{p-1}(\mathbb{G}_i - \psi_{\Uinv}^p(\mathbb{G}))\sgn(\Delta_i)\right].
    \end{equation*}
    A repeated application of the Portmanteau theorem, together with the sandwich inequality~\eqref{eq: sandwich in proof alternative} and limit laws~\eqref{eq: lower limit law alternative proof} and~\eqref{eq: upper limit law alternative proof} results in
    \begin{equation} \label{eq: limit law sandwich alternative}
        \mathbb{P}(\Xi_\delta^{p, \leq } < t) \leq \liminf_{n \to \infty} \mathbb{P}(X_{\beta,n}^p \leq t) \leq \limsup_{n \to \infty} \mathbb{P}(X_{\beta,n}^p \leq t) \leq  \mathbb{P}(\Xi_\delta^{p, \geq } \leq t).
    \end{equation}
    We now prove that as $\delta \searrow 0$, the distributions of $\Xi_\delta^{p, \geq}$ and $\Xi_\delta^{p, \leq}$ agree. To this end, note that $\omega$-wise
    \begin{equation} \label{eq: bound of upper bound for dominated conv theorem}
        \abs{\Xi_\delta^{p, \geq}} \leq \frac{1}{m}\sum_{i=1}^m \int_{\beta}^{1-\beta}p\abs{\Delta_i}^{p-1} \abs{\mathbb{G}_i - \psi_{\Uinv}^p(\mathbb{G})}
    \end{equation}
    and that furthermore 
    \begin{align*}
        &\mathbb{E} \left[ \frac{1}{m}\sum_{i=1}^m \int_{\beta}^{1-\beta}p\abs{\Delta_i}^{p-1} \abs{\mathbb{G}_i - \psi_{\Uinv}^p(\mathbb{G})} \right] \\= &\frac{1}{m}\sum_{i=1}^m \int_{\beta}^{1-\beta}p\abs{\Delta_i}^{p-1} \mathbb{E} \left[ \abs{\mathbb{G}_i - \psi_{\Uinv}^p(\mathbb{G})} \right] < \infty
    \end{align*}
    by Tonelli's theorem and due to the fact that $\mathbb{G}$ is a Gaussian process and thus has normally distributed finite dimensional distributions. The above equality together with~\eqref{eq: bound of upper bound for dominated conv theorem} implies that we may use the dominated convergence theorem to conclude that for $\delta_n = 1/n$
    \begin{equation*}
        \lim_{n \to \infty} \Xi_{\delta_n}^{p,\geq} = \frac{1}{m} \sum_{i=1}^m  \int_{C_{i,0}^>}   p  |\Delta_i|^{p-1}(\mathbb{G}_i - \psi_{\Uinv}^p(\mathbb{G}))\sgn(\Delta_i)
    \end{equation*}
    holds $\omega$-wise. Here, we have used that $C_{i,\delta}^\leq \searrow \emptyset$ and $C_{i,\delta}^> \nearrow C_{i,0}^>$ as $\delta \searrow 0$. Moreover, note that the bound from~\eqref{eq: bound of upper bound for dominated conv theorem} also holds for $\abs{\Xi_\delta^{p, \leq}}$ and thus another application of the dominated convergence theorem yields 
    \begin{equation*}
        \lim_{n \to \infty} \Xi_{\delta_n}^{p,\leq} = \frac{1}{m} \sum_{i=1}^m  \int_{C_{i,0}^>}   p  |\Delta_i|^{p-1}(\mathbb{G}_i - \psi_{\Uinv}^p(\mathbb{G}))\sgn(\Delta_i).
    \end{equation*}
    Hence we have shown that 
    \begin{equation*}
        \lim_{n \to \infty} \Xi_{\delta_n}^{p,\leq} \eqD \frac{1}{m} \sum_{i=1}^m  \int_{\beta}^{1-\beta}   p  |\Delta_i|^{p-1}(\mathbb{G}_i - \psi_{\Uinv}^p(\mathbb{G}))\sgn(\Delta_i) \eqD \lim_{n \to \infty} \Xi_{\delta_n}^{p,\geq}
    \end{equation*}
    and in combination with~\eqref{eq: limit law sandwich alternative} and the continuity of $\Xi_{\delta}^{p, \leq}$ it follows that
    \begin{equation} 
        \sqrt{n} \left( T_{\beta,n}^p - T_\beta^p \right) \rightsquigarrow \frac{1}{m} \sum_{i=1}^m  \int_{\beta}^{1-\beta}   p  |\Delta_i|^{p-1}(\mathbb{G}_i - \psi_{\Uinv}^p(\mathbb{G}))\sgn(\Delta_i).
    \end{equation} 
The proof for $p \geq 2$ is then concluded upon noticing that $\sum_{i=1}^m p \abs{\Delta_i(t)}^{p-1} \sgn(\Delta_i(t)) = 0$ for any $t\in (0,1)$, by the first order optimality condition of $\Lambda_p$.\newline
In the case $p = 1$, note that by \autoref{lemma: convergence alternative > subset}, \autoref{lemma: convergence alternative on = subset p=1} and \autoref{lemma: convergence alternative on < subset} it holds that 
\begin{equation*}
        \Xi_{\delta,n}^{1, \leq} \rightsquigarrow \Xi_\delta^{1, \leq} \text{ and } \Xi_{\delta,n}^{1, \geq} \rightsquigarrow \Xi_\delta^{1, \geq},
    \end{equation*}
    where 
    \begin{equation*}
        \begin{aligned}
        \Xi_\delta^{1, \leq} = \frac{1}{m} \sum_{i=1}^m \Biggl[ \int_{C_{i,\delta}^\leq} - \abs{\mathbb{G}_i - \psi_{\Uinv}^1(\mathbb{G})} + \int_{C_{i,\delta}^>}   &(\mathbb{G}_i - \psi_{\Uinv}^1(\mathbb{G}))\sgn(\Delta_i) \\
        &+  \int_{C_{i}^=}  \abs{\mathbb{G}_i - \psi_{\Uinv}^1(\mathbb{G})}\Biggr],
        \end{aligned}
    \end{equation*}
    and
    \begin{equation*}
        \begin{aligned}
        \Xi_\delta^{1, \geq} = \frac{1}{m} \sum_{i=1}^m \Biggl[ \int_{C_{i,\delta}^\leq} \abs{\mathbb{G}_i - \psi_{\Uinv}^1(\mathbb{G})} + \int_{C_{i,\delta}^>}   &(\mathbb{G}_i - \psi_{\Uinv}^1(\mathbb{G}))\sgn(\Delta_i)\\
        &+  \int_{C_{i}^=}  \abs{\mathbb{G}_i - \psi_{\Uinv}^1(\mathbb{G})}\Biggr].
        \end{aligned}
    \end{equation*}
    Further, by the dominated convergence theorem
    \begin{equation}
        \lim_{n \to \infty} \Xi_{\delta_n}^{1,\leq} \eqD \frac{1}{m} \sum_{i=1}^m  \int_{C_{i,0}^>}     (\mathbb{G}_i - \psi_{\Uinv}^1(\mathbb{G}))\sgn(\Delta_i) + \int_{C_i^=} \abs{\mathbb{G}_i - \psi_{\Uinv}^1(\mathbb{G})} \eqD \lim_{n \to \infty} \Xi_{\delta_n}^{1,\geq}.
    \end{equation}
    Hence, the rest of the proof can be argued analogously as for $p \geq 2$.
\end{proof}
\subsubsection*{Proof of \autoref{theorem: Convergence under the alt}(b)}
Note that in the setting of \autoref{theorem: Convergence under the alt}(b), it is implied by \autoref{theorem: convergence of barycenter process} that 
\begin{equation*}
    (\mathbb{U}_n^{-1}, \lambda_{p,n}) \rightsquigarrow (\mathbb{G}, \psi_{\Uinv}^p(\mathbb{G})) \text{ in }L^1((0,1))^{m+1},
\end{equation*}
for $p = 1$ and $p \geq 2$. An inspection of the proof of \autoref{theorem: Convergence under the alt}(a) reveals that at any point in the proof, it is sufficient that the joint process $(\mathbb{U}_n^{-1}, \lambda_{p,n})$ converges in $L^1((0,1))^{m+1}$ instead of $\ell^\infty((0,1))^m \times L^1((0,1))$. As such, the proof can be carried out in exact analogy.
\subsection{Proof of \autoref{theorem: local alternatives}}\label{proof:thm:local-alternatives}  
Recall the notation from \autoref{setting: local alternatives} and denote by $U_i$ the distribution function of the distance distribution induced by $\mu_i$. We first show weak convergence of the process $\sqrt{n}(V_{i,n} - U_i)$ in the space of bounded functions. In what follows, for any sequence of probability measures $\lambda_n$ on $\mathcal{X}_i$, we denote by $\overset{\lambda_n}{\rightsquigarrow}$ weak convergence under the assumption that the random variables $(X_{i,1}^n, \ldots, X_{i,n}^n)$ are distributed according to $\lambda_n^n = \bigotimes_{k = 1}^n \lambda_n$.
\begin{lemma}\label{lemma: convergene of perturbed u process}
    Assume that for each $i = 1, \ldots,m$, $U_i$ is continuously differentiable on some interval $[C_{i,1},C_{i,2}]$ and that $h_i$ is a measurable map satisfying 
    \begin{equation} \label{eq: hellinger derivative appendix}
        \int_{\mathcal{X}_i} \left[ \sqrt{n} \left( \sqrt{g_{i,n}(x)} - \sqrt{g_{i}(x)} \right) - \frac{1}{2}h_i(x) \sqrt{g_{i}(x)} \right]^2 \diff \rho_i(x) \overset{n \to \infty}{\longrightarrow} 0.
    \end{equation}
    Then
    \begin{equation*}
        \sqrt{n}(V_{i,n} - U_i) \overset{\nu_{i,n}}{\rightsquigarrow} \mathbb{K}_i  + 2\widetilde{D}_i \quad \text{in }\ell^\infty([C_{i,1},C_{i,2}]),
    \end{equation*}
    where $C_{i,1},C_{i,2}$ are the constants given in \autoref{cond 1.2}, $\mathbb{K}_i$ is a centered and continuous Gaussian process with covariance operator $\gammaK{i,i}$ as defined in \autoref{lemma: joint convergence of U-processes} and $\widetilde{D}_i$ is a drift given by
    \begin{equation*}
        \widetilde{D}_i(t) = \int_{\mathcal{X}_i} \int_{\mathcal{X}_i} \mathbbm{1} \!(d_i(x,y) \leq t)h_i(y) \diff \mu_i(x) \diff \mu_i(y).
    \end{equation*}
\end{lemma}
\begin{proof}
First, note that for each $i = 1, \ldots, m$, by \citet[Lemma 3.11.11]{wellner2013weak} the condition~\eqref{eq: hellinger derivative appendix} implies $\int h_i \diff \mu_i = 0$, $\int h_i^2 \diff \mu_i < \infty$ and
\begin{equation} \label{eq: log likelihood expansion}
    \Gamma_n(\nu_{i,n}, \mu_i) \coloneqq \sum_{j=1}^n \log \frac{g_{i,n}}{g_{i}}\left(X_{i,j}^n\right) = \frac{1}{\sqrt{n}} \sum_{j=1}^n h_i\left(X_{i,j}^n\right) - \frac{1}{2} \int h_i^2 \diff \mu_i + R_n\left(X_{i,1}^n, \ldots, X_{i,n}^n \right),
\end{equation}
where $R_n \coloneqq  R_n\bigl(X_{i,1}^n, \ldots, X_{i,n}^n \bigr)$ converges to $0$ in probability, under both $\nu_{i,n}^n$ and $\mu_i^n$ (meaning, under the assumption that either $X_i^n = (X_{i,1}^n, \ldots, X_{i,n}^n) \sim \nu_{i,n}^n$ or $X_i^n \sim \mu_i^n$). Also, we set
\begin{equation*}
    \frac{g_{i,n}}{g_{i}}(x) = \begin{cases}
        \frac{g_{i,n}}{g_{i}}(x)  &\text{if }g_{i}(x)> 0, \\
        1  &\text{if }g_{i}(x) = g_{i,n}(x) = 0, \\
        \infty  &\text{if }g_{i,n}(x)>g_{i}(x) = 0, \\
    \end{cases}
\end{equation*}
such that the likelihood ratios are well-defined. Assuming that $X_i^n \sim \mu_i^n$ for each $n \in \N$, the centered statistic $\sqrt{n}(V_{i,n}(t)- U_i(t))$ can be viewed through the lens of U-statistics using the distance kernel $f_{i,t}(x,y) = \mathbbm{1} \! (d_i(x,y) \leq t)$ for $t \in [C_{i,1},C_{i,2}]$. Hence, by analogous arguments as given in the derivation of~\eqref{eq: U process via Hoeffding} in the proof of \autoref{lemma: joint convergence of U-processes}, it holds that
\begin{equation} \label{eq: fidi of U-process decomposition}
    \sqrt{n}(V_{i,n}(t)- U_i(t)) = 2 \sqrt{n} \left( \frac{1}{n}\sum_{j=1}^n \mathbb{E}_{Y \sim \mu_i} \left[ \mathbbm{1} \! (d_i(X_{i,j}^n,Y) \leq t) \right] - U_i(t) \right) + o_p(1)
\end{equation}
under $\mu_i^n$ and for any fixed $t \in [C_{i,1}, C_{i,2}]$. For notational ease, denote in the following 
\begin{equation*}
    \phi_{i,t}(x) = \mathbb{E}_{Y \sim \mu_i} \left[ \mathbbm{1} \! (d_i(x,Y) \leq t) \right]
\end{equation*}
Putting together~\eqref{eq: log likelihood expansion} and~\eqref{eq: fidi of U-process decomposition} and using the notation $\mu (f) = \int f \diff \mu$, an application of the central limit theorem yields that for each $t \in \mathbb{R}$,
\begin{equation} \label{eq: tightness of likelihood ratio}
    \left(\Gamma_n(\nu_{i,n}, \mu_i), \sqrt{n}(V_{i,n}(t)- U_i(t)) \right) \overset{\mu_i}{\rightsquigarrow} \mathcal{N}\left(  \begin{pmatrix}
-\frac{1}{2} \mu_i(h_i^2)\\
0 
\end{pmatrix}, \begin{pmatrix}
\mu_i(h_i^2) & 2\mu_i (h_i \phi_{i,t})\\
2\mu_i (h_i \phi_{i,t}) & \gammaK{i,i}(t,t) 
\end{pmatrix} \right)
\end{equation} 
where $\gammaK{i,i}$ is the covariance operator as defined in \autoref{lemma: joint convergence of U-processes}. By Le Cam's third lemma \citep[Appendix A.9, Lemma 3]{bickel1993efficient}, it thus holds that 
\begin{equation} \label{eq: fidi under local alternative}
    \sqrt{n}(V_{i,n}(t)- U_i(t)) \overset{\nu_{i,n}}{\rightsquigarrow} \mathcal{N} \left( 2\mu_i( h_i \phi_{i,t}), \gammaK{i,i}(t,t) \right)
\end{equation}
for each $i = 1, \ldots, m$ and $t \in \R$. Note that the above statement trivially also holds for any finite collection of indices $t_1,\ldots, t_k$. Next, we consider the processes $\sqrt{n}(V_{i,n} - U_i)$ in $\ell^\infty([C_{i,1},C_{i,2}])$. \autoref{lemma: joint convergence of U-processes} implies that 
\begin{equation*}
    \sqrt{n}(V_{i,n}- U_i) \overset{\mu_{i}}{\rightsquigarrow} \mathbb{K}_i \qquad \text{in }\ell^\infty([C_{i,1},C_{i,2}]).
\end{equation*}
Thus, $\sqrt{n}(V_{i,n} - U_i)$ is asymptotically tight in $\ell^\infty([C_{i,1},C_{i,2}])$ under $\mu_i^n$ by \citet[Lemma 1.3.8]{wellner2013weak}. Furthermore,~\eqref{eq: tightness of likelihood ratio} implies that $\Gamma_n(\nu_{i,n}, \mu_i)$ is asymptotically tight under $\mu_i^n$. Hence, $(\sqrt{n}(V_{i,n} - U_i), \Gamma_n(\nu_{i,n}, \mu_i))$ is jointly asymptotically tight in $\ell^\infty([C_{i,1},C_{i,2}]) \times \bar{\mathbb{R}}$ under $\mu_i^n$ by an application of \citet[Lemma 1.4.3]{wellner2013weak}. Consider now any subsequence $(\sqrt{n}(V_{i,n_k} - U_i), \Gamma_{n_k}(\nu_{i,n_k}, \mu_i))$ which converges weakly in $\ell^\infty([C_{i,1},C_{i,2}]) \times \bar{\R}$ under $\mu_i^n$. Such a subsequence exists by Prohorov's lemma \citep[Theorem 1.3.9]{wellner2013weak}. Then \citet[Theorem 3.11.7]{wellner2013weak} yields that $\sqrt{n}(V_{i,n_k} - U_i)$ is asymptotically tight in $\ell^\infty([C_{i,1},C_{i,2}])$ under $\nu_{i,n}^n$. The theorem is applicable, since $\nu_{i,n}^n$ is contiguous with respect to $\mu_i^n$, which is a consequence of the fact that the log-likelihood ratio $\Gamma_n(\nu_{i,n}, \mu_i)$ converges weakly under $\mu_i^n$ and its limit is normally distributed with mean $-1/2 \mu_i(h_i^2)$ and variance $\mu_i(h_i^2)$, see \citet[Example 3.11.6]{wellner2013weak}. Combining the asymptotic tightness of the subsequence $\sqrt{n}(V_{i,n_k} - U_i)$ under $\nu_{i,n}^n$ with convergence of the finite dimensional distributions~\eqref{eq: fidi under local alternative} yields that 
\begin{equation*}
    \sqrt{n}(V_{i,n_k} - U_i) \overset{\nu_{i,n_k}}{\rightsquigarrow} \mathbb{K}_i + 2 \widetilde{D}_i \text{ in }\ell^\infty([C_{i,1},C_{i,2}])
\end{equation*}
by \citet[Theorem 1.5.4]{wellner2013weak}. As the limit distribution in the above display is independent of the specific converging subsequence chosen, the claim follows.
\end{proof}
\begin{lemma} \label{lemma: inverted u process contiguous alternatives}
    Let $\beta >0$ and assume that \autoref{cond 1.2} is satisfied for $U_1, \ldots, U_m$ for some constants $C_{i,1},C_{i,2}$, $i =1, \ldots, m$. Further, suppose that~\eqref{eq: hellinger derivative appendix} is satisfied. Then it holds that for each $i = 1, \ldots,m$
    \begin{equation*}
        \sqrt{n}\left(V_{i,n}^{-1} - U_i^{-1}\right) \rightsquigarrow \mathbb{G}_i + 2 D_i \quad \text{in }\ell^\infty([\beta,1-\beta]),
    \end{equation*}
    where $\mathbb{G}_i$ is a centered Gaussian process with covariance operator $\gammaG{i,i}$ as stated in \autoref{lemma: joint convergence inverted u process} and $D_i$ is a drift term given by 
    \begin{equation*}
        D_i(t) = - \frac{\widetilde{D}_i\left(U_i^{-1}(t) \right)}{u_i\left( U_i^{-1}(t) \right)}.
    \end{equation*}
    Here, $\widetilde{D}_i$ is the drift term given in \autoref{lemma: convergene of perturbed u process}.
\end{lemma}
\begin{proof}
    The arguments are analogous to the ones given in the proof of \autoref{lemma: joint convergence inverted u process}, with the only exception that the inversion map and delta method are not applied on the joint processes, but instead on each marginal $\sqrt{n}(V_{i,n} - U_i)$. The only additional fact which needs to be shown is that the limit distribution of $\sqrt{n}(V_{i,n} - U_i)$ given in \autoref{lemma: convergene of perturbed u process} is also continuous. Since $\mathbb{K}_i$ is continuous for each $i = 1, \ldots, m$, this is reduced to showing that the drift term $\widetilde{D}_i$ is continuous. But this is a consequence of the dominated convergence theorem, since $\abs{\mathbbm{1} \!(d_i(x,y) \leq t) h_i(y)} \leq \abs{h_i(y)}$, which is integrable with respect to $\mu_i$ (see proof of \autoref{lemma: convergene of perturbed u process}), combined with the fact that $\mu_i \otimes \mu_i (\lbrace (x,y) \in \mathcal{X}_i^2 : d_i(x,y) = t \rbrace) = 0$ for any $t \in \mathbb{R}$, by continuity of $U_i$.
\end{proof}
Given the weak convergence of the inverted processes from the previous lemma, we are now ready to prove \autoref{theorem: local alternatives}.
\begin{proof}[Proof of \autoref{theorem: local alternatives}]
    Note that we are in the setting of \autoref{lemma: inverted u process contiguous alternatives} and thus 
    \begin{equation*}
        \sqrt{n}\left(V_{i,n}^{-1} - U_i^{-1}\right) \rightsquigarrow \mathbb{G}_i + 2 D_i \quad \text{in }\ell^\infty([\beta,1-\beta])
    \end{equation*}
    for each $i = 1, \ldots,m$. As the processes are independent (recall \autoref{setting: local alternatives}), this also implies (denoting $\mathbf{D} = (D_1, \ldots, D_m)$) 
    \begin{equation*}
        \mathbb{V}_n^{-1} \coloneqq \sqrt{n}\left(V_{1,n}^{-1} - U_1^{-1}, \ldots ,V_{m,n}^{-1} - U_m^{-1}\right) \rightsquigarrow  \mathbb{G} + 2\mathbf{D} \text{ in }\ell^\infty([\beta,1-\beta])^m.
    \end{equation*}
    As in the proof of \autoref{theorem: Convergence under the Null}, define the (continuous) map $\phi: \ell^\infty([\beta,1-\beta])^m \to \mathbb{R}$ via
    \begin{equation*}
        \phi(f_1, \ldots , f_m) = \frac{1}{m} \sum_{i=1}^m \int_\beta^{1-\beta} \left|f_i(x) - \Lambda_p(f_1(x),\ldots, f_m(x)) \right|^p \diff x.
    \end{equation*}
    Due to the assumption that $\mu^{U_1} = \ldots = \mu^{U_m}$, we thus obtain by an application of the continuous mapping theorem
    \begin{align*}
        n^{p/2}\widetilde{T}_{\beta,n}^p &=  \frac{n^{p/2}}{m} \sum_{i=1}^m \int_\beta^{1-\beta} \left| V_{i,n}^{-1} -U_i^{-1} - \Lambda_p(V_{1,n}^{-1}- U_1^{-1},\ldots,V_{m,n}^{-1} - U_m^{-1})\right|^p \\
        &= \phi(\mathbb{V}_n^{-1}) \rightsquigarrow \phi\left( \mathbb{G}_1 + 2 D_1, \ldots, \mathbb{G}_m + 2 D_m \right),
    \end{align*}
    which concludes the proof.
\end{proof}
\subsection{Proof of \autoref{theorem: barycenter classifier}}\label{proof:thm:barycenter-classifier} 
The first step is to realize that 
\begin{align}
    \mathbb{P} \left( \widehat{i}^\star  \neq i^\star \right) &= \mathbb{P}\Bigl( \exists j \neq i^\star : \widehat{\gamma}_{j,n} \leq \widehat{\gamma}_{i^\star,n} \Bigr) \nonumber \\
    & \leq \sum_{j \neq i^\star} \mathbb{P} \left( \widehat{\gamma}_{j,n} \leq \widehat{\gamma}_{i^\star,n}  \right) \nonumber \\
    & = \sum_{j \neq i^\star} \mathbb{P} \left( \widehat{\gamma}_{i^\star} - \gamma_{i^\star} - (\widehat{\gamma}_j - \gamma_j)  \geq \gamma_j - \gamma_{i^\star} \right) \nonumber \\
    & \leq \sum_{j \neq i^\star } \frac{\mathbb{E} \left[ \abs{\widehat{\gamma}_{i^\star,n} - \gamma_{i^\star}} \right] + \mathbb{E} \left[ \abs{\widehat{\gamma}_{j,n} - \gamma_j} \right]}{\gamma_j - \gamma_{i^\star}}. \label{eq: classifier decomposition}
\end{align}
Here, in the last line we have used the Markov inequality, which is applicable since by assumption $\gamma_j - \gamma_{i^\star} >0$ for each $j \neq i^\star$. It remains to bound the expectations in the above display. To ease the notation, fix some $j \in \lbrace 1, \ldots, K \rbrace$ and write, for each $k = 1, \ldots , m_j$,
\begin{equation*}
    n_k \coloneqq n_{j,k}, \qquad U_k^{-1} \coloneqq U_{j,k}^{-1}, \qquad U_{k,n_k}^{-1} \coloneqq U_{(j,k), n_{j,k}}^{-1}.
\end{equation*}
Then, by analogous arguments as the ones used in the derivations of~\eqref{eq: finite sample bias wasserstein} in the proof of Theorem 2.2, it then holds that 
\begin{equation} \label{eq: classifier inequality}
    \abs{\widehat{\gamma}_j - \gamma_j} \leq p(2D)^{p-1} \left[ \int_0^D \abs{V_{n_V}(t) - V(t)} \diff t + \sum_{k=1}^{m_j} \int_0^D \abs{U_{k,n_k}(t) - U_k(t)} \diff t \right].
\end{equation}
As described in the discussion following~\eqref{eq: bound finite sample risk}, $V_{n_V}(t)$ constitutes a $U$-statistic for every $t \in \mathbb{R}$. Thus by \citet[chap.~5.2 Lemma A]{serfling1980approximation} 
\begin{equation} \label{eq: variance bound U-stat}
    \Var(V_{n_V}(t)) \leq \frac{2}{n_V}.
\end{equation}
By Tonelli's theorem, Jensen's inequality and~\eqref{eq: variance bound U-stat}, it follows that 
\begin{align} 
    \mathbb{E}\left[ \int_0^D \abs{V_{n_V}(t) - V(t)} \diff t \right] &=  \int_0^D \mathbb{E}\left[ \abs{V_{n_V}(t) - V(t)} \right] \diff t \nonumber \\
    &\leq \int_0^D \Var(V_{n_V}(t) - V(t))^{1/2} \diff t \nonumber \\
    & \leq D \left( \frac{2}{n_V} \right)^{1/2}.
    \label{eq: classifier second ineq}
\end{align}
Analogous arguments lead to
\begin{equation} \label{eq: classifier third ineq}
    \mathbb{E}\left[ \int_0^1 \abs{U_{k,n_k}^{-1} - U_k^{-1}} \right] \leq D \left( \frac{2}{n_k} \right)^{1/2}.
\end{equation}
Thus, a combination of~\eqref{eq: classifier inequality},~\eqref{eq: classifier second ineq} and~\eqref{eq: classifier third ineq} leads to 
\begin{equation*}
    \mathbb{E}( \abs{\widehat{\gamma}_j - \gamma_j}) \leq 2^{p+1/2} pD^{p} \left( n_V^{-1/2} + \sum_{k=1}^{m_j} n_k^{-1/2} \right).
\end{equation*}
Then it only remains to apply the above bound to~\eqref{eq: classifier decomposition}. \qed

\newpage

\section{Bootstrap Consistency}\label{appendix: bootstrap} 
This section is devoted to the proof of \autoref{theorem: bootstrap consistency}. In order to deal with measurability issues, we use theory for non-measurable random maps developed by \citet[chap.~1]{wellner2013weak}. Throughout this section, we denote for any random map $X: \Omega \to \mathbb{D}$ mapping from a probability space $(\Omega, \mathcal{A}, \mathbb{P})$ to a metric space $(\mathbb{D},d)$, its measurable majorant by $X^\star$ and the respective measurable minorant by $X_\star$. Furthermore, we denote by $\mathbb{E}^\star$ and $\mathbb{E}_\star$ the outer and inner expectations, respectively. Similarly, the outer and inner probabilities are $\mathbb{P}^\star$ and $\mathbb{P}_\star$. Convergence in outer probability of random maps $X_n$ to a random element $X$ is denoted by $X_n \outerconvP X$ and outer almost sure convergence via $X_n \outerconvAS X$. For an introduction to these concepts, we refer the reader to \citet[chap.~1]{wellner2013weak}. It is also noteworthy that we use the notation $\star$ for everything related to outer expectations and measurable majorants, while we use $\ast$ for bootstrap-related objects.\\
Proving the consistency of the bootstrap requires a notion of conditional weak convergence. In general, consider random maps $Z_n, n \in \mathbb{N}$ into a metric space $(\mathbb{D},d)$ and let $Y_n = Y_n(\mathcal{Z}_n, M_n)$ be random maps in $(\mathbb{D},d)$ based on the observations $\mathcal{Z}_n = \lbrace Z_1, \ldots, Z_n\rbrace$ and some random element $M_n$ (which, in the bootstrap framework, is the random element responsible for the resampling). Furthermore, define the set 
\begin{equation*}
    BL_{1}(\mathbb{D}) = \left\{ f: \mathbb{D} \to \mathbb{R}: \sup_{x \in \mathbb{D}} \abs{f(x)} \leq 1 \text{ and }\abs{f(x)-f(y)} \leq d(x,y) \text{ for }x,y\in \mathbb{D} \right\}
\end{equation*}
of bounded Lipschitz functions with Lipschitz constant $1$. In many cases, we will just write $BL_1$, as the underlying space $\mathbb{D}$ will be clear from context. We then say that $Y_n$ converges in distribution conditionally outer almost surely if 
\begin{equation} \label{eq: definition BL convergence conv part}
    \sup_{h \in BL_1} \abs{\mathbb{E}_M^\star(h(Y_n)) - \mathbb{E}(h(Y))} \outerconvAS 0
\end{equation}
where $\mathbb{E}_M$ denotes the expectation only with respect to the resampling weights $M_n$, and if 
\begin{equation} \label{eq: definition BL convergence measurability part}
    \mathbb{E}_M\left[ h(Y_n)^\star \right] - \mathbb{E}_M \left[ h(Y_n)_\star \right] \outerconvAS 0 
\end{equation}
for every $h \in BL_1$. Moreover, $h(Y_n)^\star$ and $h(Y)_\star$ denote the measurable majorant and minorant with respect to both the observations $\mathcal{Z}_n$ and $M_n$. Overall,~\eqref{eq: definition BL convergence conv part} and~\eqref{eq: definition BL convergence measurability part} will be denoted by $Y_n \condconvAS Y$. We use the notation $Y_n \condconvP Y$ to denote conditional weak convergence in outer probability, i.e.\ we replace outer almost sure convergence by convergence in outer probability in the definition. Note that the bootstrap U-statistic $U_{i,n}^\ast$, as defined in \autoref{subsec: bootstrap}, does indeed fit the previously introduced framework, as for a multinomially distributed $M_n$ with parameters $n$ and (probabilities) $(1/n, \ldots, 1/n)$, it is possible to write
\begin{align*}
    U_{i,n}^\ast\left(t; \left\{ X_{1,i}, \ldots, X_{n,i} \right\},M_n \right) = \frac{2}{n(n-1)} \Bigg(& \sum_{1 \leq k < l  \leq n} M_{n,k}M_{n,l} \mathbbm{1} \!\left( d_i(X_{i,k}, X_{i,l}) \leq t) \right)\\
    & + \sum_{1\leq k \leq n } \frac{M_{n,k}(M_{n,k}-1)}{2} \mathbbm{1} \!\left( d_i(X_{i,k}, X_{i,k}) \leq t) \right) \Bigg).
\end{align*} 
\subsection{Bootstrap Consistency of Quantile Process}
We first give an analogous result to \autoref{lemma: joint convergence inverted u process} for the bootstrap quantile process
\begin{equation*}
    \mathbb{U}_{n}^{\ast, -1} = \sqrt{n} \left(  U_{1,n}^{\ast,-1} - U_{1,n}^{-1} , \ldots , U_{m,n}^{\ast, -1} - U_{m,n}^{-1} \right),
\end{equation*}
where $U_{i,n}^{\ast,-1}$ denotes the quantile function of $U_{i,n}^\ast$. 
\begin{lemma}\label{lemma: inverted bootstrap u process}
    Let $\beta \in (0,1/2)$ and assume \autoref{cond 1.2}. Then
    \begin{equation*}
         \mathbb{U}_{n}^{\ast,-1} \condconvP \mathbb{G} \coloneqq (\mathbb{G}_1, \ldots, \mathbb{G}_m) \text{ in }\ell^\infty([\beta, 1-\beta])^m
    \end{equation*}
    where $\mathbb{G}$ is the centered Gaussian process that follows the same distribution as in \autoref{lemma: joint convergence inverted u process}.
\end{lemma}
In order to prove \autoref{lemma: inverted bootstrap u process}, we first provide a distributional limit for the process
\begin{equation*}
    \mathbb{U}_{n}^{\ast, V}(t) = \sqrt{n} \left(  U_{1,n}^\ast(t) - V_{1,n}(t), \ldots ,  U_{m,n}^\ast(t) - V_{m,n}(t)\right),
\end{equation*}
where we define
\begin{equation*}
    V_{i,n}(t) = \frac{1}{n^2} \sum_{k,l = 1}^n \mathbbm{1} \! \left( d_i(X_{i,k}, X_{i,l}) \leq t\right).
\end{equation*}
Note that $V_{i,n}$ only differs from $U_{i,n}$ by the fact that it sums over all possible combinations of indices instead of excluding symmetric indices. \newline
In order to derive a convergence result for $\mathbb{U}_n^{\ast,V}$, we require the concept of covering numbers. For a pseudometric space $(\mathbb{D},d)$, we can define its $\varepsilon$-covering number by
\begin{equation*}
    N(\varepsilon,\mathbb{D},d)=\min\left\{ n \in \mathbb{N}| \exists x_1, \ldots, x_n \in \mathbb{D} \text{ such that }\mathbb{D} \subset \cup_{i=1}^n B_\varepsilon(x_i) \right\}.    
\end{equation*}
This notion can be extended to classes of functions. If $\mathcal{F}$ is a class of functions $f:\mathcal{X} \to \mathbb{R}$ and $\mu$ is a measure on $\mathcal{X}$, then
\begin{equation*}
    N_2(\varepsilon,\mathcal{F}, \mu) = N(\varepsilon, \mathcal{F}, \| \cdot \|_{L^2(\mu)})
\end{equation*}
is the related covering number of the class. Furthermore, we say that $F$ is an envelope of a function class $\mathcal{F}$ if $\abs{f(x)} \leq F(x)$ for each $f \in \mathcal{F}$ and $x \in \mathcal{X}$.
\begin{lemma}\label{lemma: U process bootstrap limit for U^V}
    Assume that for each $i = 1, \ldots,m$, there are constants $C_{i,1},C_{i,2}$ such that $U_i$ is continuously differentiable on $[C_{i,1},C_{i,2}]$. Then, as $n \to \infty$, it holds that 
    \begin{equation*}
        \mathbb{U}_{n}^{\ast, V} \condconvP \mathbb{K} \coloneqq (\mathbb{K}_1, \ldots, \mathbb{K}_m) \text{ in }\ell^\infty([C_{1,1},C_{1,2}]) \times \ldots \times \ell^\infty([C_{m,1},C_{m,2}]),
    \end{equation*}
    where $\mathbb{K}$ is the Gaussian process defined in \autoref{lemma: joint convergence of U-processes}.
\end{lemma}
\begin{proof}
    Recall that for each $i=1, \ldots,m$, the process $\mathbb{U}_{i,n} = \sqrt{n}(U_{i,n} - U_i)$ constitutes a $U$-process indexed over the function class~\eqref{eq: function class u-process}, see also the discussion preceding the equation. By \citet[Theorem 2.1]{arcones1994u}, it holds that, conditionally on a sample $X_1,X_2, \ldots$,
    \begin{equation} \label{eq: marginal bootstrap convergence V}
        \mathbb{U}_{i,n}^{\ast, V} = \sqrt{n}(U_{i,n}^\ast - V_{i,n}) \rightsquigarrow \mathbb{K}_i \text{ in }\ell^\infty([C_{i,1},C_{i,2}]),
    \end{equation}
    for almost every sample $X_1,X_2, \ldots$, where $\mathbb{K}_i$ is the $i$-th marginal of $\mathbb{K}$ defined in \autoref{lemma: joint convergence of U-processes}, if we can show the following conditions:
    \begin{enumerate}
        \item There is a function $\lambda: (0, \infty) \to [0,\infty)$ with $\int_0^\infty \lambda(x)dx < \infty$ such that for each probability measure $\nu$ on $\mathcal{X}_i$ it holds $\int F^2 \diff \nu < \infty$ for an envelope function $F$ of $\mathcal{F}_i$ and such that 
        \begin{equation} \label{eq: entropy condition}
            \sqrt{\log N_2\left(x \norm{F}_{L^2(\nu)}, \mathcal{F}_i, \nu  \right) } \leq \lambda(x)
        \end{equation}
        for any $x > 0$. 
        \item For any integers $i_1,i_2$ it holds that 
        \begin{equation} \label{eq: second requirement bootstrap u process}
            \mathbb{E} \left( \abs{F(X_{i_1},X_{i_2})}^{\abs{\left\{ i_1,i_2 \right\}}} \right) < \infty.
        \end{equation} 
        \item We have convergence of
        \begin{equation} \label{eq: thirds requiremtn bootstrap u process}
            \sqrt{n}(U_{i,n} - U_i) \rightsquigarrow \mathbb{K}_i \text{ in }\ell^\infty([C_{i,1},C_{i,2}]).
        \end{equation}
    \end{enumerate}
    Regarding the first requirement, note that $F \equiv 1$ is an envelope function of $\mathcal{F}_i$, which immediately yields $\int F^2 \diff \nu = 1$ for any probability measure $\nu$. The condition~\eqref{eq: entropy condition} follows by \citet[Theorem 2.6.7]{wellner2013weak} since $\mathcal{F}_i$ is a VC class (see proof of \citet[Lemma A.7]{weitkamp2024distribution}). The second requirement~\eqref{eq: second requirement bootstrap u process} is also a simple consequence of the fact that $F \equiv 1$ is an envelope of $\mathcal{F}_i$. Lastly, due to the continuous differentiability of $U_i$ on the respective intervals, the third requirement~\eqref{eq: thirds requiremtn bootstrap u process} is given by \autoref{lemma: joint convergence of U-processes}. Hence,~\eqref{eq: marginal bootstrap convergence V} holds for each $i=1, \ldots,m$ outer almost surely, which implies that $\mathbb{U}_{i,n}^{\ast, V}$ is outer almost surely asymptotically tight for each $i=1, \ldots, m$. Further, Theorem 2.4, Corollary 2.6, Remark 2.7 and Remark 2.10]{arcones1992bootstrap} guarantees that the finite dimensional distributions of $\mathbb{U}_{n}^{\ast, V}$ converge weakly to the finite dimensional distributions of $\mathbb{K}$, conditioned on $X_1, X_2, \ldots$ for almost every such sample. Overall, by the characterization of joint weak convergence given in~\eqref{eq: characterization JWC} and the preceding discussion, we obtain
    \begin{equation*}
        \mathbb{U}_{n}^{\ast, V} \rightsquigarrow \mathbb{K}  \text{ in }\ell^\infty([C_{1,1},C_{1,2}]) \times \ldots \times \ell^\infty([C_{m,1},C_{m,2}]),
    \end{equation*}
    conditionally on $X_1, X_2, \ldots$ for outer almost every sequence of observations. As the bounded Lipschitz metric metrizes weak convergence \citep[chap.~1.12]{wellner2013weak}, we thus obtain 
    \begin{equation*}
        \sup_{h \in BL_1} \left| \mathbb{E}_M(h(\mathbb{U}_{n}^{\ast, V})) - \mathbb{E}(h(\mathbb{K})) \right| \outerconvAS 0,
    \end{equation*}
    which proves part~\eqref{eq: definition BL convergence conv part} of the definition of outer almost sure conditional weak convergence. Note that we need not take the outer expectation $\mathbb{E}_M^\star$ above. This is due to the fact that the resampling weights $M_n$ take their values in the set
    \begin{equation} \label{eq: sigma algebra resampling weights}
        \mathbb{M}_n \coloneqq \left\{ (M_{n,1}, \ldots, M_{n,n}) \in \mathbb{N}^n : \sum_{i=1}^n M_{n,i} = n \right\}
    \end{equation}
    equipped with the power-set $\sigma$-algebra. As any map is measurable with respect to the power-set $\sigma$-algebra, this implies the measurability of $h(\mathbb{U}_{n}^{\ast, V})$ with respect to $M_n$ and thus $\mathbb{E}_M^\star(h(\mathbb{U}_{n}^{\ast, V})) = \mathbb{E}_M(h(\mathbb{U}_{n}^{\ast, V}))$.\newline 
    It remains to be shown that part~\eqref{eq: definition BL convergence measurability part} of the definition of conditional convergence holds. In order to see this, we first show unconditional weak convergence of $\mathbb{U}_{i,n}^{\ast, V} \coloneqq \sqrt{n}\left( U_{i,n}^\ast - V_{i,n} \right)$ to $\mathbb{K}_i$ holds. To this end, note that for any function $h\in BL_1$,
    \begin{align*}
        &\left| \mathbb{E}^\star \left[ h(\mathbb{U}_{i,n}^{\ast, V}) \right] - \mathbb{E}\left[h(\mathbb{K}_i) \right] \right| \\
        \leq &\left| \mathbb{E}^\star \left[ h(\mathbb{U}_{i,n}^{\ast, V}) \right] -  \mathbb{E}_X^\star \mathbb{E}_M \left[ h\left(\mathbb{U}_{i,n}^{\ast, V}\right) \right] \right| + \left| \mathbb{E}_X^\star \mathbb{E}_M \left[ h \left(\mathbb{U}_{i,n}^{\ast, V} \right) \right] - \mathbb{E} \left[ h(\mathbb{K}_i) \right] \right| \\
        \eqqcolon &\mathbb{I} + \mathbb{II}.
    \end{align*}
    By definition, the map $h$ is bounded, which means that the outer expectation is finite and we can write $\mathbb{E}^\star(h(\mathbb{U}_{i,n}^{\ast, V})) = \mathbb{E}(h(\mathbb{U}_{i,n}^{\ast, V})^\star)$. Moreover, by an application of Fubini's theorem for outer expectations \citep[Lemma 1.2.7]{wellner2013weak}, it holds that
    \begin{equation*}
        \mathbb{E} \left[ h \left(  \mathbb{U}_{i,n}^{\ast, V} \right)_\star \right] \leq \mathbb{E}_X^\star \mathbb{E}_M \left[ h \left(\mathbb{U}_{i,n}^{\ast, V} \right) \right] \leq  \mathbb{E} \left[ h\left(\mathbb{U}_{i,n}^{\ast, V}\right)^\star \right].
    \end{equation*}
    Thus, we obtain
    \begin{align*}
        \mathbb{I} \leq \mathbb{E} \left[ h\left(\mathbb{U}_{i,n}^{\ast, V}\right)^\star \right] - \mathbb{E}\left[h\left(\mathbb{U}_{i,n}^{\ast, V}\right)_\star\right] \leq \mathbb{E}_X^\star\left[ \mathbb{E}_M \left[ h\left(\mathbb{U}_{i,n}^{\ast, V}\right)^\star\right] - \mathbb{E}_M\left[h\left(\mathbb{U}_{i,n}^{\ast, V}\right)_\star \right] \right].
    \end{align*}
    Since $\mathbb{U}_{i,n}^{\ast, V} \condconvAS \mathbb{K}$ as shown in the discussion surrounding~\eqref{eq: marginal bootstrap convergence V}, the term inside the expectation $\mathbb{E}_X^\star$ converges to $0$ outer almost surely for each $h \in BL_1$ by~\eqref{eq: definition BL convergence measurability part}. Hence, $\mathbb{I}$ converges to $0$ for each $h \in BL_1$. Regarding the second summand $\mathbb{II}$, note that by an application of Jensen's inequality,
    \begin{align*}
        \mathbb{II} = \left| \mathbb{E}_X^\star \left[ \mathbb{E}_M \left[ h\left(\mathbb{U}_{i,n}^{\ast, V}\right)\right] - \mathbb{E}\left[h(\mathbb{K}_i) \right] \right] \right| \leq \mathbb{E}_X^\star \left[  \left| \mathbb{E}_M \left[ h\left(\mathbb{U}_{i,n}^{\ast, V}\right)\right] - \mathbb{E}\left[h(\mathbb{K}_i) \right] \right| \right].
    \end{align*}
    Once again, the term inside the expectation $\mathbb{E}_X^\star$ converges to 0 outer almost surely for every $h \in BL_1$ due to the fact that $\mathbb{U}_{i,n}^{\ast, V} \condconvAS \mathbb{K}_i$. Thus, also term $\mathbb{II}$ converges to 0 for every $h \in BL_1$. Hence, it follows that
    \begin{equation*}
        \left| \mathbb{E}^\star \left[ h(\mathbb{U}_{i,n}^{\ast, V}) \right] - \mathbb{E}\left[ h(\mathbb{K}_i) \right] \right| \leq \mathbb{I} + \mathbb{II} \to 0
    \end{equation*}
    for every $h \in BL_1$, implying that $\mathbb{U}_{i,n}^{\ast, V} \rightsquigarrow \mathbb{K}_i$ unconditionally. Weak convergence implies asymptotic measurability and asymptotic tightness. Therefore, for every $i=1, \ldots, m$, $\mathbb{U}_{i,n}^{\ast, V}$ is asymptotically measurable and asymptotically tight. Moreover, this implies that the joint process $\mathbb{U}_{n}^{\ast, V}$ is asymptotically measurable by \citet[Lemma 1.3.8]{wellner2013weak}, i.e.\
    \begin{equation*}
         \mathbb{E}_X \left[ \mathbb{E}_M \left[ h\left(\mathbb{U}_{n}^{\ast, V}\right)^\star \right] - \mathbb{E}_M\left[ h\left(\mathbb{U}_{n}^{\ast, V}\right)_\star \right] \right] = \mathbb{E} \left[ h\left(\mathbb{U}_{n}^{\ast, V}\right)^\star \right] - \mathbb{E}\left[ h\left(\mathbb{U}_{n}^{\ast, V}\right)_\star \right] \to 0
    \end{equation*}
    for every continuous bounded function $h$ and thus also for every $h \in BL_1$. Combined with the fact that by Markov's inequality
    \begin{equation*}
        \mathbb{P}^\star\left[ \left| \mathbb{E}_M \left[ h\left(\mathbb{U}_{n}^{\ast, V}\right)^\star \right] - \mathbb{E}_M\left[h\left(\mathbb{U}_{n}^{\ast, V}\right)_\star \right] \right| > \varepsilon \right] \leq \varepsilon^{-1}\mathbb{E}_X^\star \left[ \mathbb{E}_M \left[ h\left(\mathbb{U}_{n}^{\ast, V}\right)^\star \right] - \mathbb{E}_M\left[ h\left(\mathbb{U}_{n}^{\ast, V}\right)_\star\right] \right],
    \end{equation*}
    we thus conclude that 
    \begin{equation*}
        \mathbb{E}_M(h(\mathbb{U}_{n}^{\ast, V})^\star) - \mathbb{E}_M(h(\mathbb{U}_{n}^{\ast, V})_\star) \to 0
    \end{equation*}
    in outer probability.
\end{proof}
In a next step, the distributional convergence of $\mathbb{U}_n^{\ast,V}$ is related to the convergence of the process
\begin{equation*}
    \mathbb{U}_n^\ast = \sqrt{n} \left( U_{1,n}^\ast - U_{1,n}, \ldots, U_{m,n}^\ast - U_{m,n}  \right).
\end{equation*}
\begin{lemma}\label{lemma: convergence of bootstrap u process}
    Assume that for each $i = 1, \ldots,m$, there are constants $C_{i,1},C_{i,2}$ such that $U_i$ is continuously differentiable on $[C_{i,1},C_{i,2}]$. Then 
    \begin{equation*}
        \mathbb{U}_{n}^{\ast} \condconvP \mathbb{K} \coloneqq (\mathbb{K}_1, \ldots, \mathbb{K}_m) \text{ in }\ell^\infty([C_{1,1},C_{1,2}]) \times \ldots \times \ell^\infty([C_{m,1},C_{m,2}])
    \end{equation*}
    where $\mathbb{K}$ is the distributional limit of $\mathbb{U}_{n}^{\ast, V}$ derived in \autoref{lemma: U process bootstrap limit for U^V}.
\end{lemma}
\begin{proof}
    This proof essentially follows \citet[Theorem F.4]{weitkamp2024distribution}, slightly adjusted to fit the multidimensional process framework. First, we will check that 
    \begin{equation} \label{eq: U bootstrap process first condition}
        \sup_{h \in BL_1} \abs{\mathbb{E}_M^\star \left[ h\left(\mathbb{U}_{n}^{\ast}\right) \right] - \mathbb{E}\left[ h \left( \mathbb{K} \right)\right]} \outerconvP 0.
    \end{equation}
    In order to see this, we relate the above term to $\mathbb{U}_{n}^{\ast, V}$ as follows.
    \begin{align} \label{eq: split into two BL1 norms}
        \sup_{h \in BL_1}  \abs{\mathbb{E}_M^\star \left[ h\left(\mathbb{U}_{n}^{\ast}\right) \right] - \mathbb{E}\left[ h \left( \mathbb{K} \right)\right]} \leq &\sup_{h \in BL_1}  \abs{\mathbb{E}_M^\star \left[ h\left(\mathbb{U}_{n}^{\ast}\right) \right] - \mathbb{E}_M^\star\left[ h \left( \mathbb{U}_n^{\ast,V} \right)\right]} \\
        + &\sup_{h \in BL_1} \abs{\mathbb{E}_M^\star\left[ h \left( \mathbb{U}_n^{\ast,V} \right)\right] - \mathbb{E} \left[ h\left(\mathbb{K} \right) \right]} \nonumber
    \end{align}
    by the triangle inequality. In a first step, consider the first summand of~\eqref{eq: split into two BL1 norms}. As the norm is taken over Lipschitz functions with constant $1$, it holds that
    \begin{align}
        \sup_{h \in BL_1} \abs{\mathbb{E}_M^\star \left[ h \left(\mathbb{U}_{n}^{\ast} \right) \right] - \mathbb{E}_M^\star\left[ h\left(\mathbb{U}_{n}^{\ast, V}\right) \right]} &\leq  \mathbb{E}_M^\star \left[ \norm{ \mathbb{U}_{n}^\ast - \mathbb{U}_{n}^{\ast, V} }_{[C_1,C_2]} \right] \nonumber \\
        &=\sum_{i=1}^m\mathbb{E}_M^\star \left[ \norm{ \mathbb{U}_{i,n}^\ast - \mathbb{U}_{i,n}^{\ast, V} }_{[C_{i,1},C_{i,2}]} \right]     \label{eq: difference U and U^V}   
    \end{align}
    where $\mathbb{U}_{i,n}^\ast$ and $\mathbb{U}_{i,n}^{\ast, V}$ denote the $i$-th component of the respective processes $\mathbb{U}_{n}^{\ast}$ and $\mathbb{U}_{n}^{\ast, V}$ and $\norm{\cdot}_{[C_1,C_2]}$ is defined in~\eqref{eq: norm on product space}. We now focus on controlling the term inside the norm in~\eqref{eq: difference U and U^V}. Firstly, note that since $d_i$ is a metric, it holds that $\mathbbm{1} \! \left( d_i(X_{i,j},X_{i,j}) \leq t \right) = 1$ for each $t \geq 0$ and thus one can rewrite 
    \begin{align*}
        V_{i,n}(t) &= n^{-2} \sum_{j,k = 1}^n \mathbbm{1} \!\left( d_i \left( X_{i,j}, X_{i,k} \right) \leq t \right) \\
        &=\frac{1}{n} + n^{-2} \sum_{j\neq k} \mathbbm{1} \!\left( d_i \left(X_{i,k}, X_{i,j}\right) \leq t \right) \\
        &= \frac{1}{n}+\frac{n-1}{n}U_{i,n}(t)
    \end{align*}
    where the last equation is a consequence of the symmetry of a metric. Thus, we can bound the difference between the two processes as follows:
    \begin{align}
        \abs{\mathbb{U}_{i,n}^\ast(t) - \mathbb{U}_{i,n}^{\ast, V}(t)} &= \sqrt{n}\left|U_{i,n}(t) - V_{i,n}(t) \right| \nonumber \\
        &=\sqrt{n} \left| \frac{1}{n} U_{i,n}(t) - \frac{1}{n} \right| \nonumber \\
        &\leq \frac{1}{\sqrt{n}}\Big|U_{i,n}(t)\Big| + \frac{1}{\sqrt{n}} \label{eq: bound difference U and V process}
    \end{align}
    Clearly, $U_{i,n}(t) \leq 1$ for each $t \in \mathbb{R}$ by the fact that the kernel function of the $U$-statistic is bounded by $1$. Overall, the above considerations yield that 
    \begin{equation*}
        \sum_{i=1}^m\mathbb{E}_M^\star \left[ \norm{ \mathbb{U}_{i,n}^\ast - \mathbb{U}_{i,n}^{\ast, V} }_{[C_{i,1},C_{i,2}]} \right] \leq \frac{2m}{\sqrt{n}},
    \end{equation*}
    which converges to 0. Combined with~\eqref{eq: difference U and U^V}, we have thus proven that the first summand in~\eqref{eq: split into two BL1 norms} converges to $0$. Further, the second summand converges to $0$ as well, due to the fact that $\mathbb{U}_{n}^{\ast, V} \condconvP \mathbb{K}$. This concludes the proof of~\eqref{eq: U bootstrap process first condition}. \newline
    In a second step, we show that 
    \begin{equation*}
        \mathbb{E}_M \left[ h\left(\mathbb{U}_{n}^{\ast}\right)^\star \right] - \mathbb{E}_M \left[ h \left(\mathbb{U}_{n}^{\ast}\right)_\star\right] \outerconvP 0,
    \end{equation*}
    for each $h \in BL_1$. To this end, note that
    \begin{align}
        \mathbb{E}_M\left[h\left(\mathbb{U}_{n}^{\ast}\right)^\star\right] &= \mathbb{E}_M\left[ \left( h\left(\mathbb{U}_{n}^{\ast} \right) - h\left(\mathbb{U}_{n}^{\ast, V}\right) + h\left(\mathbb{U}_{n}^{\ast, V}\right) \right)^\star\right] \nonumber \\
        &\leq \mathbb{E}_M\left[ \left( h\left(\mathbb{U}_{n}^{\ast} \right) - h\left(\mathbb{U}_{n}^{\ast, V}\right) \right)^\star\right] + \mathbb{E}_M\left(h\left(\mathbb{U}_{n}^{\ast, V}\right)^\star\right). \label{eq: second condition bootstrap u process}
    \end{align}
    Since $h$ is Lipschitz continuous with Lipschitz constant 1, analogous considerations as in the derivation of~\eqref{eq: bound difference U and V process} yield
    \begin{align*}
        \mathbb{E}_M\left[ \left( h\left(\mathbb{U}_{n}^{\ast} \right) - h\left(\mathbb{U}_{n}^{\ast, V}\right) \right)^\star\right] &\leq \mathbb{E}_M\left[ \left(\norm{ U_{n}^\ast - U_{n}^{\ast,V} }_{[C_1,C_2]}\right)^\star\right] \\
        &\leq \frac{2m}{\sqrt{n}}.
    \end{align*}
    Combined with~\eqref{eq: second condition bootstrap u process}, this leads to
    \begin{equation} \label{eq: bound first part measurability cond}
        \mathbb{E}_M \left[ h \left(\mathbb{U}_{n}^{\ast}\right)^\star \right] \leq \mathbb{E}_M\left[h\left(\mathbb{U}_{n}^{\ast, V}\right)^\star\right] + \frac{2m}{\sqrt{n}}.
    \end{equation}
    Note that it also holds that \begin{equation*}
        -h(\mathbb{U}_{n}^{\ast})_\star \leq - \left(h\left(\mathbb{U}_{n}^{\ast} \right) - h\left(\mathbb{U}_{n}^{\ast, V}\right) \right)_\star - h\left(\mathbb{U}_{n}^{\ast, V}\right)_\star.
    \end{equation*}
    Again, it can be argued analogously to~\eqref{eq: bound difference U and V process} that
    \begin{equation} \label{eq: bound negative expectation bootstrap u process}
        - \mathbb{E}_M \left[ h\left(\mathbb{U}_{n}^{\ast}\right)_\star \right] \leq \frac{2m}{\sqrt{n}} - \mathbb{E}_M \left[ h\left(\mathbb{U}_{n}^{\ast, V}\right)_\star \right].
    \end{equation}
    Now, it remains to combine~\eqref{eq: bound first part measurability cond} and~\eqref{eq: bound negative expectation bootstrap u process}, which results in
    \begin{align*}
        \mathbb{E}_M\left[h\left(\mathbb{U}_{n}^{\ast}\right)^\star\right] - \mathbb{E}_M \left[ h\left(\mathbb{U}_{n}^{\ast}\right)_\star \right] \leq o(1) + \mathbb{E}_M\left[h\left(\mathbb{U}_{n}^{\ast, V}\right)^\star\right] - \mathbb{E}_M \left[ h\left(\mathbb{U}_{n}^{\ast, V}\right)_\star \right].
    \end{align*}
    The right hand side of the above equation converges to 0 in outer probability as $n \to \infty$, since $\mathbb{U}_{n}^{\ast, V} \condconvP \mathbb{K}$.
\end{proof}
Now we are equipped to provide the proof of \autoref{lemma: inverted bootstrap u process}.
\begin{proof}[Proof of \autoref{lemma: inverted bootstrap u process}]
    Note that we are in the correct setting to apply \autoref{lemma: convergence of bootstrap u process}, i.e.\ under the given conditions it holds that 
    \begin{equation*}
        \mathbb{U}_{n}^{\ast} \condconvP \mathbb{K} \coloneqq (\mathbb{K}_1, \ldots, \mathbb{K}_m) \text{ in }\ell^\infty([C_{1,1},C_{1,2}]) \times \ldots \times \ell^\infty([C_{m,1},C_{m,2}]).
    \end{equation*}
    Then, the rest of the proof is carried out analogously to the proof of \autoref{lemma: joint convergence inverted u process}, using the conditional delta method \citep[Theorem 12.1]{kosorok2008introduction} in conjunction with the inversion map $\invmap : \mathbb{D}_1^m \to \ell^\infty([\beta,1-\beta])^m$ defined in the proof of \autoref{lemma: joint convergence inverted u process}. The only additional requirement for applying the conditional delta method is that the map  $h(\mathbb{U}_n^\ast)$ is measurable with respect to the resampling weights $M_n$ for every bounded and continuous $h$ outer almost surely. But this was already remarked in the discussion surrounding~\eqref{eq: sigma algebra resampling weights} in the proof of \autoref{lemma: U process bootstrap limit for U^V}.
\end{proof}
\subsection*{Proof of \autoref{theorem: bootstrap consistency}}
    By \autoref{lemma: inverted bootstrap u process}, it follows that $\mathbb{U}_{n}^{\ast,-1} \condconvP \mathbb{G}$. The rest follows by the conditional continuous mapping theorem \citep[Proposition 10.7]{kosorok2008introduction} applied with the map $\phi : \ell^\infty([\beta, 1-\beta])^m \to \mathbb{R}$ defined via 
    \begin{equation*}
        \phi (f_1, \ldots, f_m) = \frac{1}{m} \sum_{i=1}^m \int_\beta^{1-\beta}|f_i -  \Lambda_p(f_1, \ldots, f_m)|^p,
    \end{equation*}
    which is clearly continuous. Also, $\ell^\infty((\beta, 1-\beta))^m$ is a Banach space with respect to the uniform norm. Thus, in order to apply the conditional continuous mapping theorem, it remains to be shown that $h\left(\mathbb{U}_{n}^{\ast,-1} \right)$ is measurable with respect to the resampling weights $M_n$ for every $h \in \mathcal{C}_b$ outer almost surely, where $\mathcal{C}_b$ is the set of bounded and continuous functions. This is the case for equivalent reasons as in the discussion surrounding~\eqref{eq: sigma algebra resampling weights} in the proof of \autoref{lemma: inverted bootstrap u process}. Thus we conclude that 
    \begin{align*}
        \Xi_{n}^\ast = \phi \left( \mathbb{U}_{n}^{\ast,-1} \right)\condconvP \frac{1}{m} \sum_{i=1}^m \int_\beta^{1-\beta} \left| \mathbb{G}_i(t) - \Lambda_p\left(\mathbb{G}_1(t), \ldots, \mathbb{G}_m(t) \right) \right|^p \diff t.
    \end{align*}
    Hence, we have also shown that the quantiles of the asymptotic distribution of $\Xi_{n}^\ast$ and $n^{p/2}T_{\beta,n}^p$ agree, which implies \autoref{theorem: bootstrap consistency}. \qed
\clearpage
\newpage
\section{Further Simulations}\label{appendix: simulations}
\begin{figure}[h!]
    \centering
    \begin{subfigure}{\textwidth}
        \centering
        \includegraphics[width=1\textwidth]{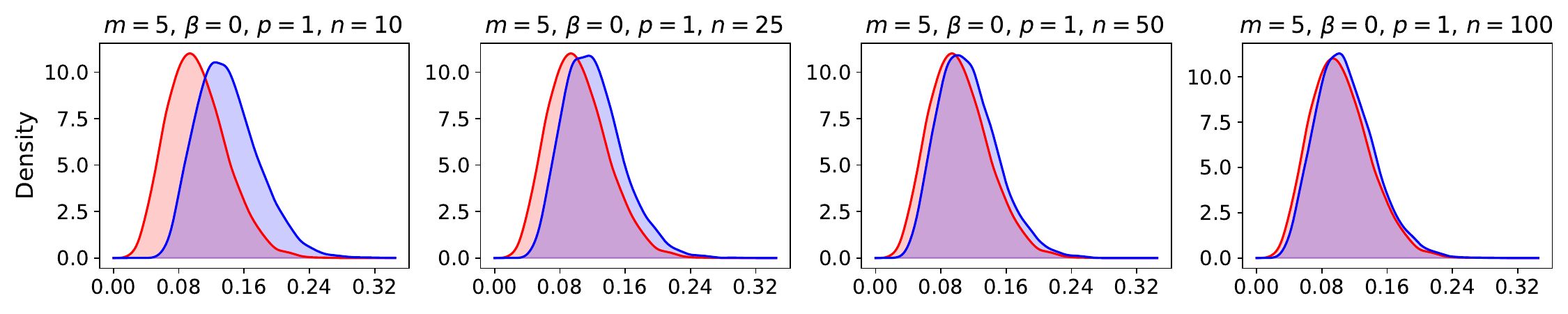}
    \end{subfigure}

    \begin{subfigure}{\textwidth}
        \centering
        \includegraphics[width=1\textwidth]{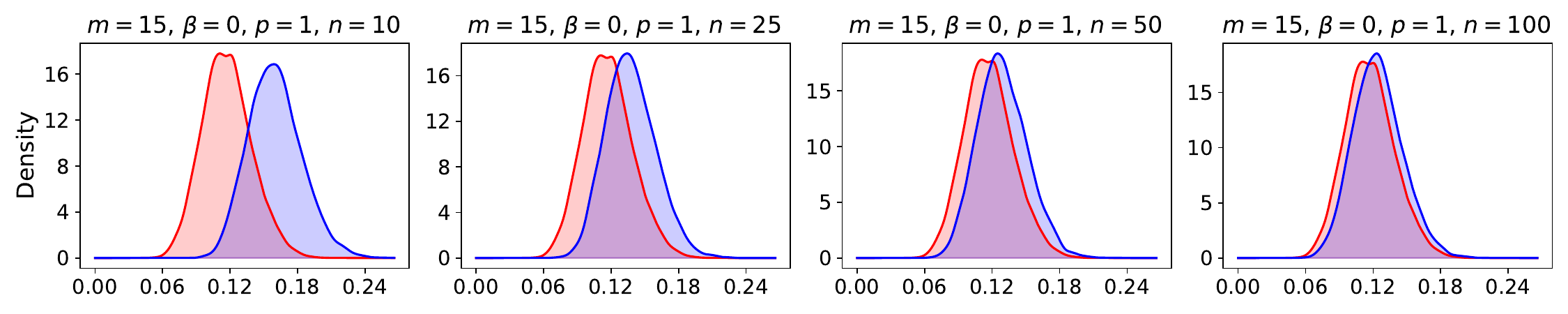}
    \end{subfigure}
    \begin{subfigure}{\textwidth}
        \centering
        \includegraphics[width=1\textwidth]{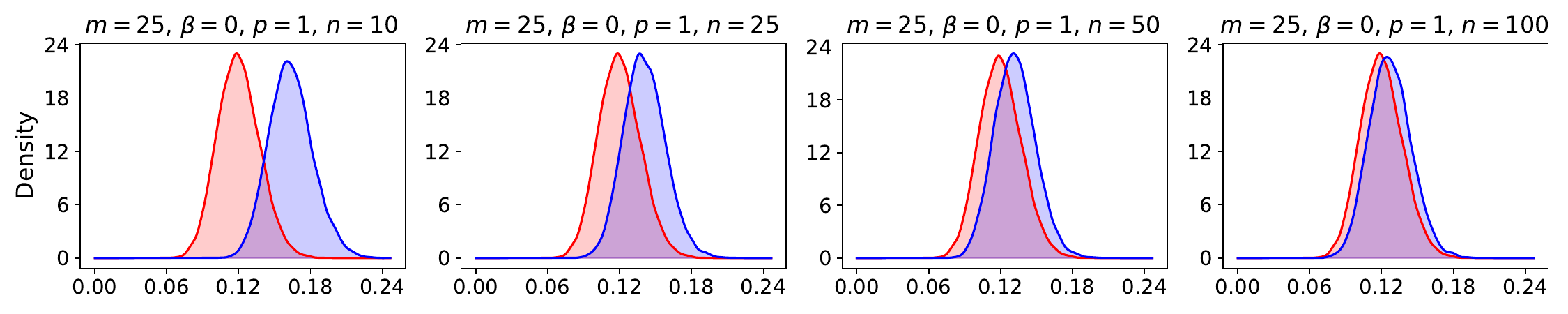}
    \end{subfigure}
    \begin{subfigure}{\textwidth}
        \centering
        \includegraphics[width=1\textwidth]{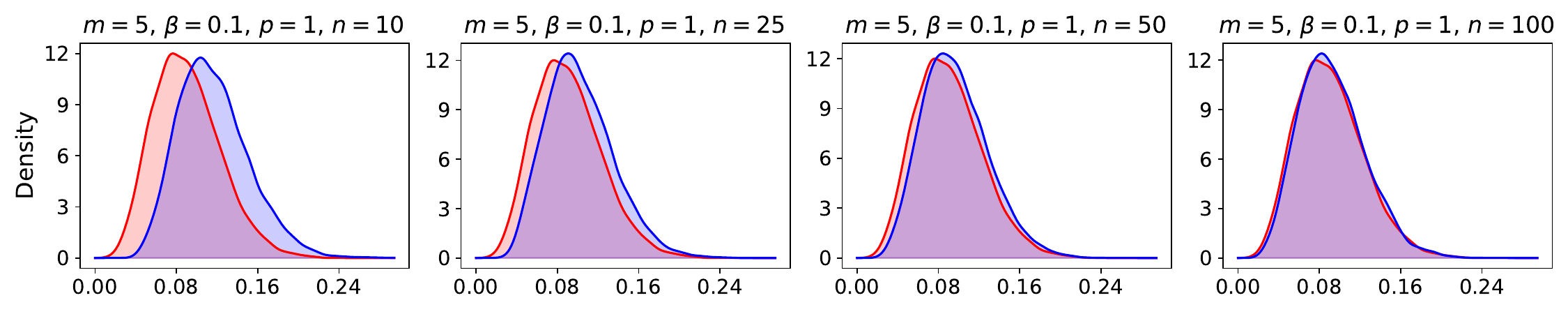}
    \end{subfigure}
    \begin{subfigure}{\textwidth}
        \centering
        \includegraphics[width=1\textwidth]{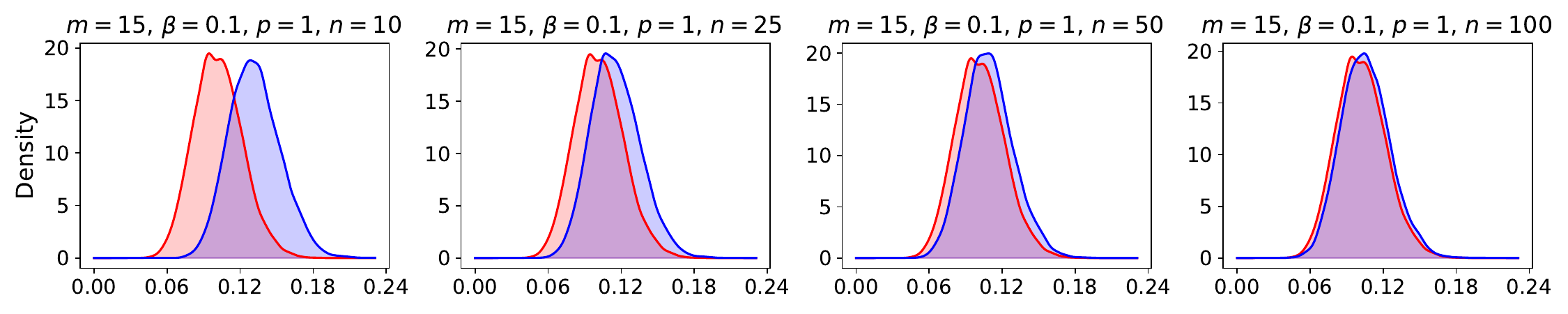}
    \end{subfigure}
    \begin{subfigure}{\textwidth}
        \centering
        \includegraphics[width=1\textwidth]{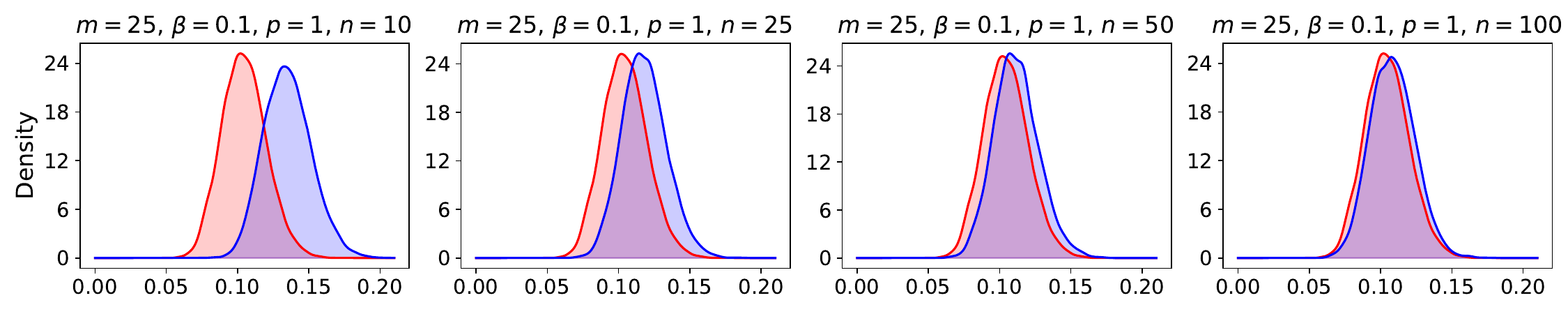}
    \end{subfigure}

    \caption{Kernel density estimates of $10,\!000$ realizations of $n^{p/2} T_{\beta,n}^p$ (blue) and limiting distribution in \autoref{theorem: Convergence under the Null} (red) for $n = 10,25,50,100$ using \autoref{example: rate of covnergence} with different hyperparameters.}\label{figure: Simulation Nullhypothese 1}
\end{figure}
\newpage
\begin{figure}[h!]
    \centering
    \begin{subfigure}{\textwidth}
        \centering
        \includegraphics[width=1\textwidth]{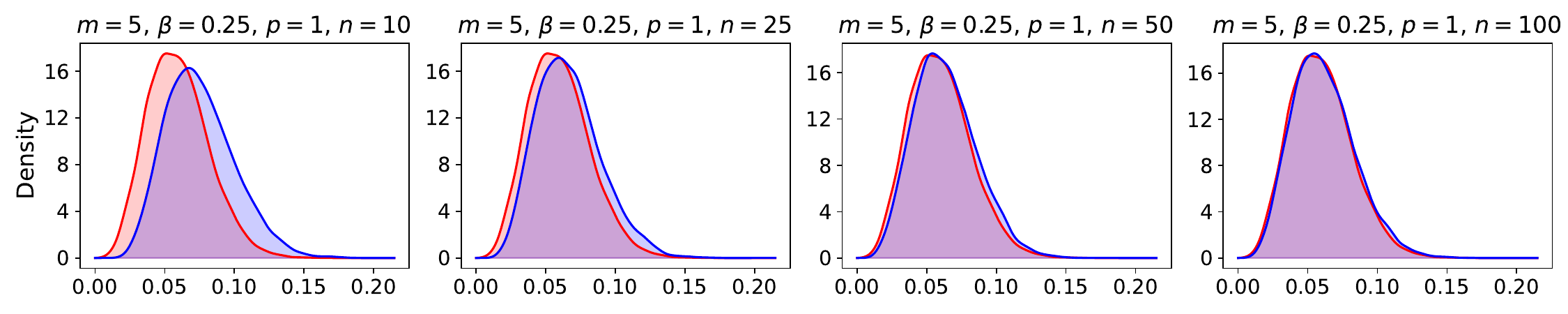}
    \end{subfigure}

    \begin{subfigure}{\textwidth}
        \centering
        \includegraphics[width=1\textwidth]{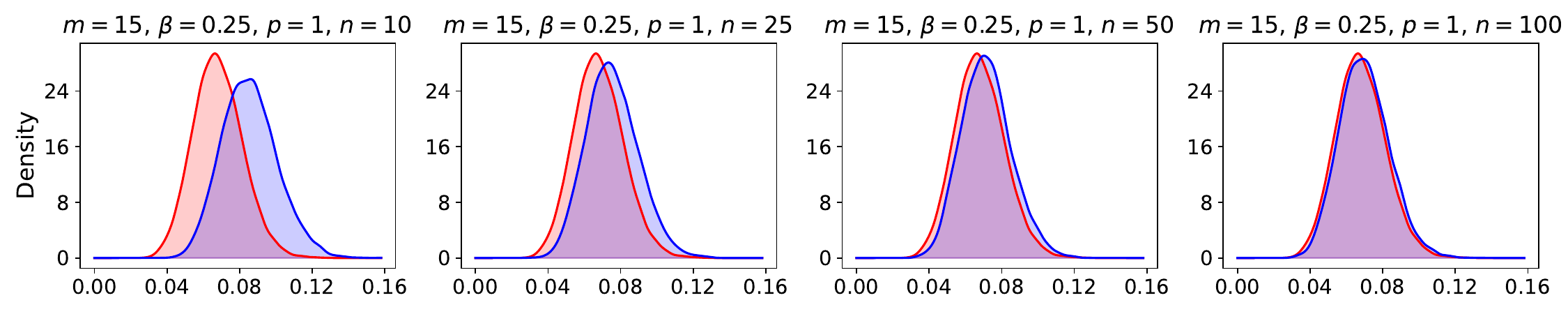}
    \end{subfigure}
    \begin{subfigure}{\textwidth}
        \centering
        \includegraphics[width=1\textwidth]{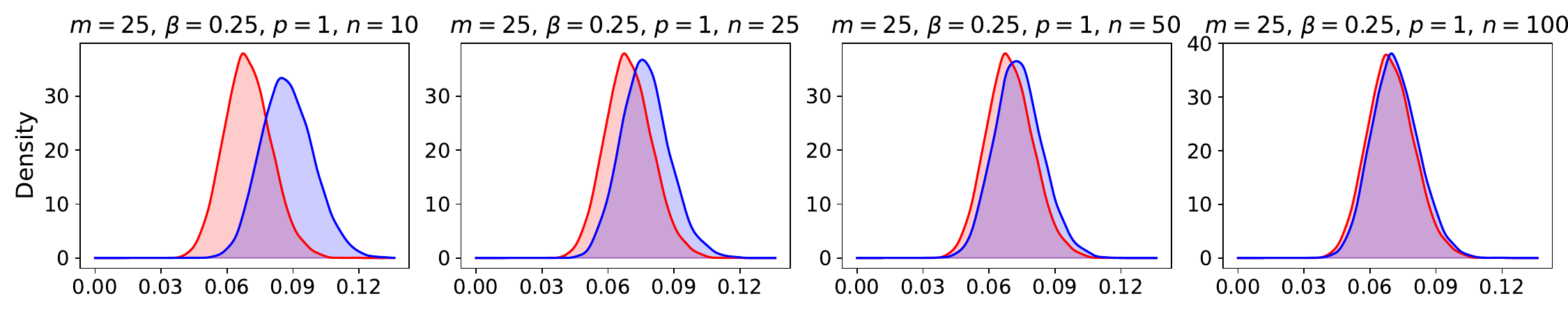}
    \end{subfigure}
    \begin{subfigure}{\textwidth}
        \centering
        \includegraphics[width=1\textwidth]{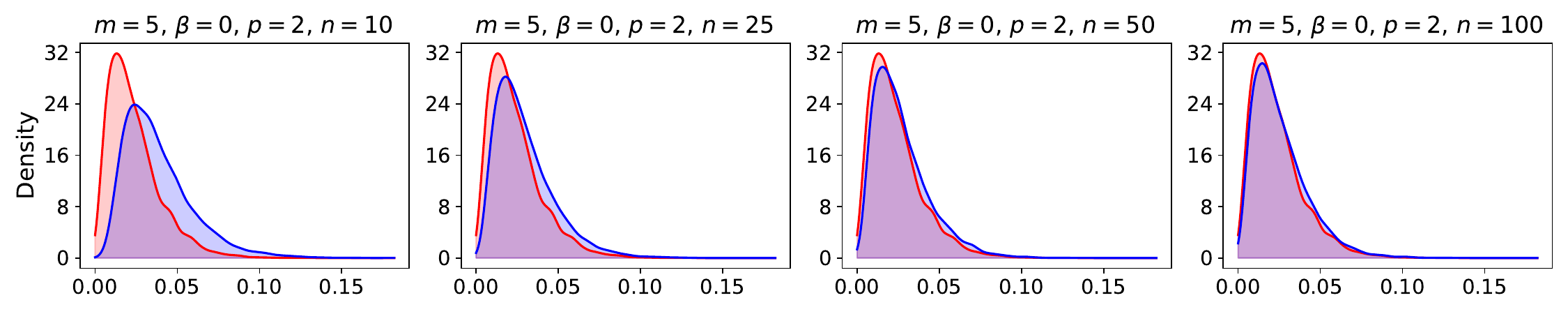}
    \end{subfigure}
    \begin{subfigure}{\textwidth}
        \centering
        \includegraphics[width=1\textwidth]{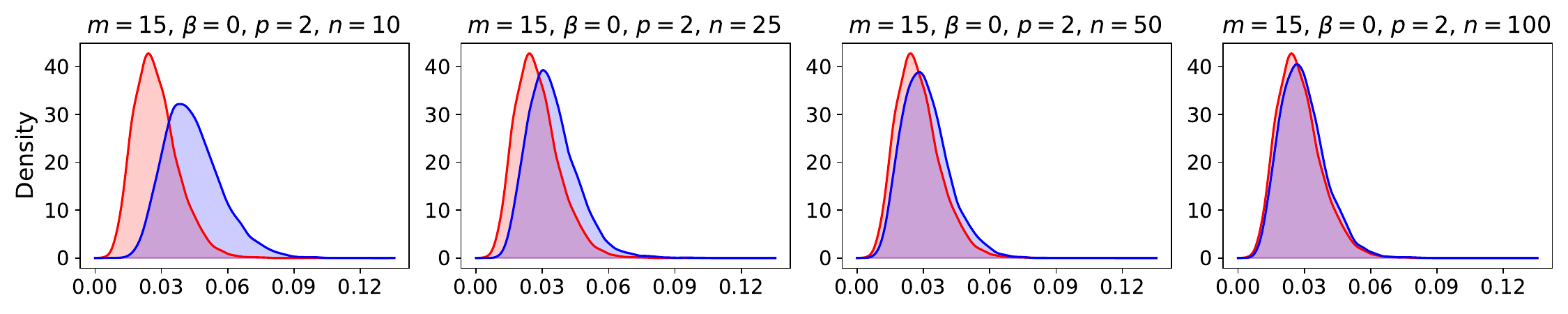}
    \end{subfigure}
    \begin{subfigure}{\textwidth}
        \centering
        \includegraphics[width=1\textwidth]{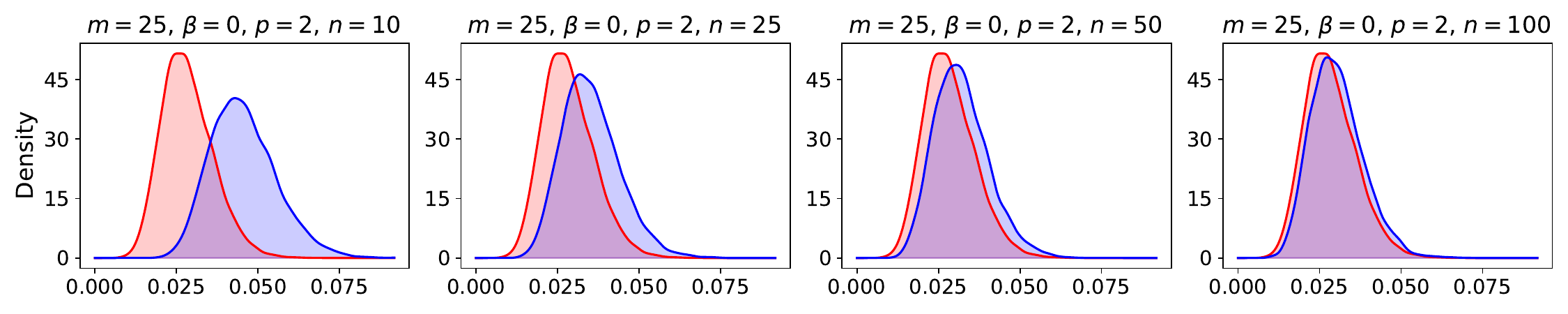}
    \end{subfigure}
    \caption{Kernel density estimates of $10,\!000$ realizations of $n^{p/2} T_{\beta,n}^p$ (blue) and limiting distribution in \autoref{theorem: Convergence under the Null} (red) for $n = 10,25,50,100$ using \autoref{example: rate of covnergence} with different hyperparameters.}\label{figure: Simulation Nullhypothese 2}
\end{figure}
\newpage
\begin{figure}[h!]
    \centering
    \begin{subfigure}{\textwidth}
        \centering
        \includegraphics[width=1\textwidth]{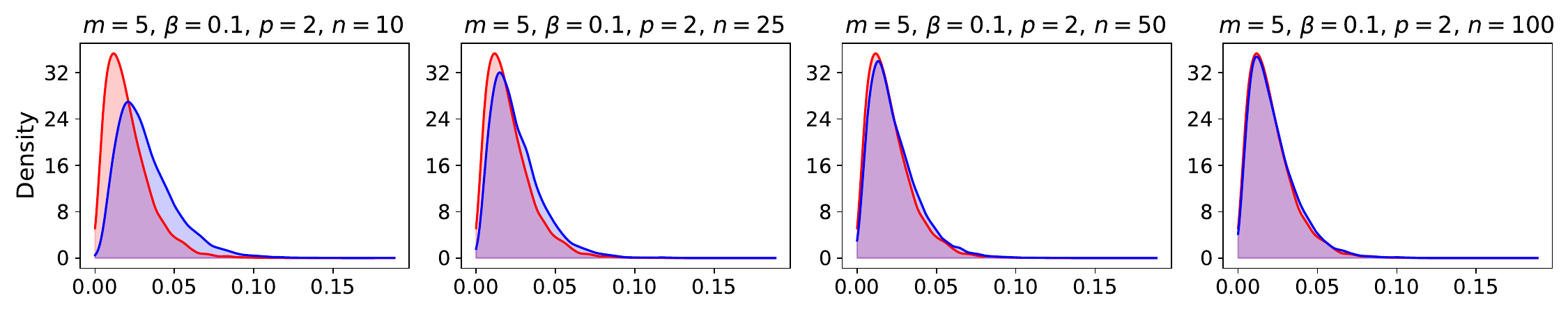}
    \end{subfigure}

    \begin{subfigure}{\textwidth}
        \centering
        \includegraphics[width=1\textwidth]{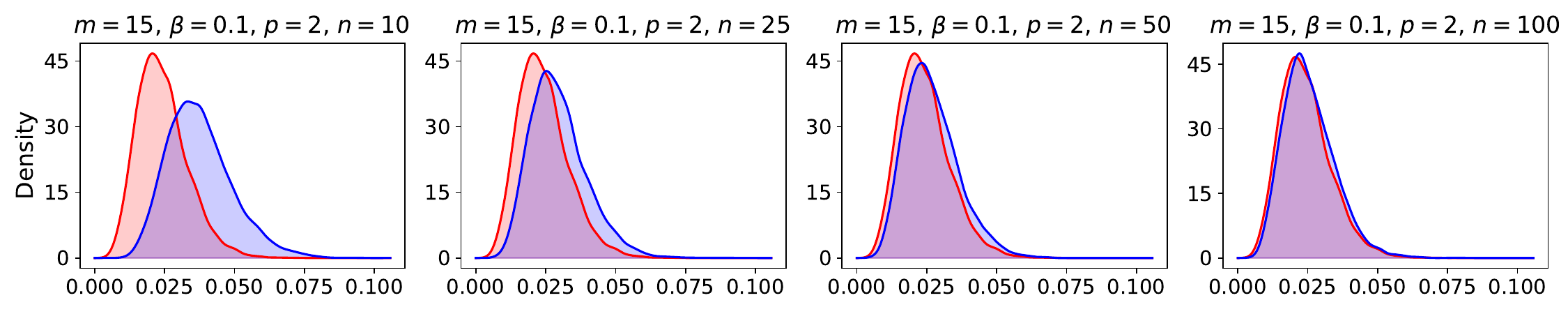}
    \end{subfigure}
    \begin{subfigure}{\textwidth}
        \centering
        \includegraphics[width=1\textwidth]{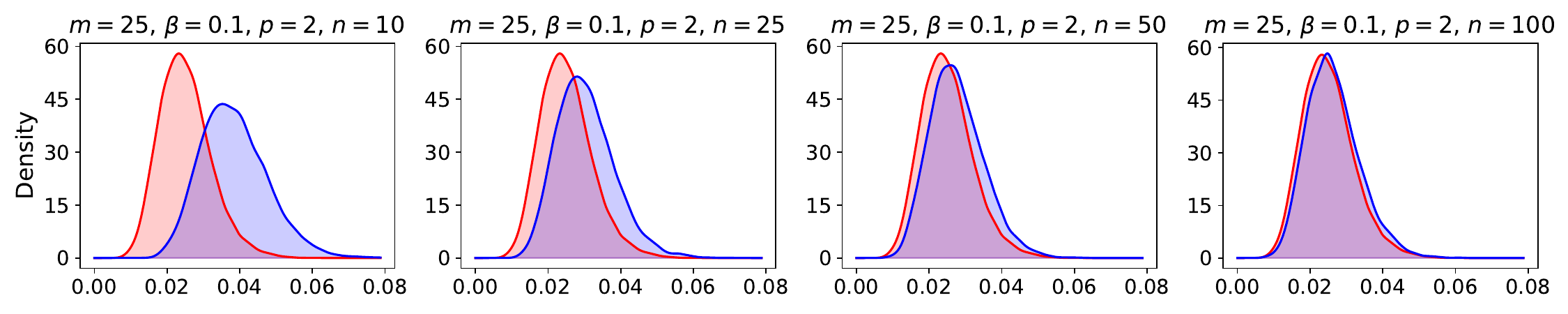}
    \end{subfigure}
    \begin{subfigure}{\textwidth}
        \centering
        \includegraphics[width=1\textwidth]{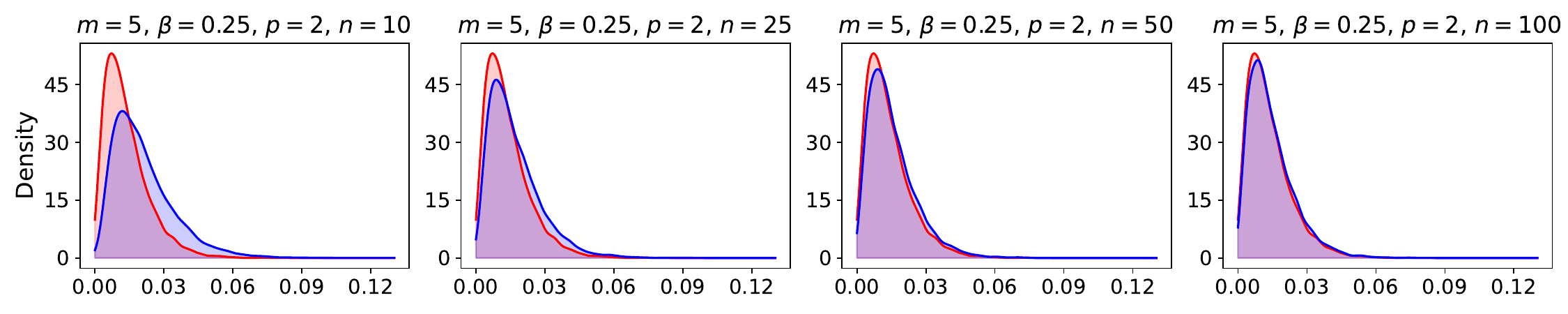}
    \end{subfigure}
    \begin{subfigure}{\textwidth}
        \centering
        \includegraphics[width=1\textwidth]{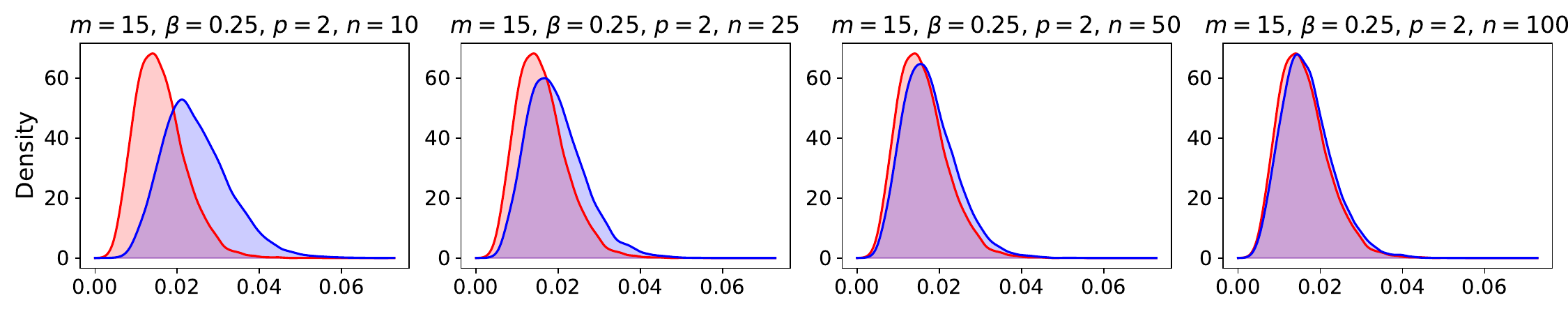}
    \end{subfigure}
    \begin{subfigure}{\textwidth}
        \centering
        \includegraphics[width=1\textwidth]{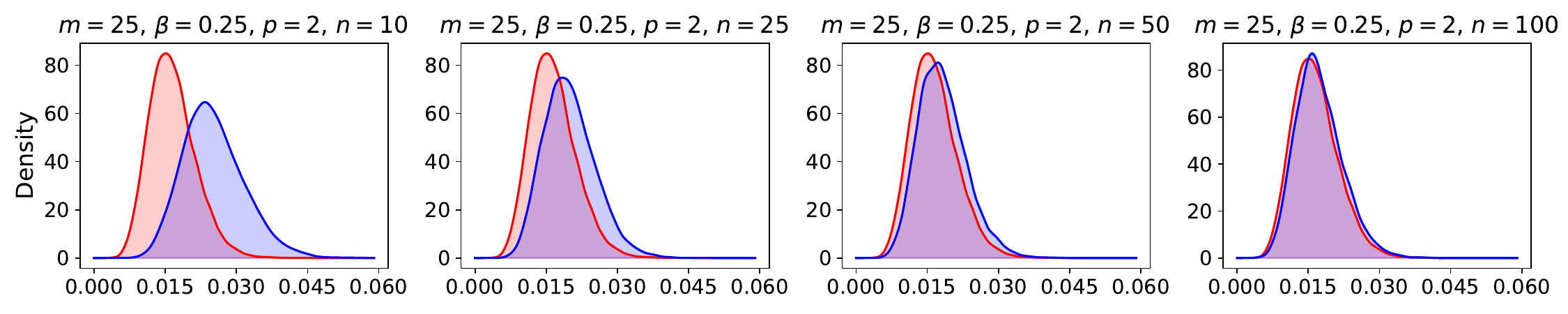}
    \end{subfigure}    
    \caption{Kernel density estimates of $10,\!000$ realizations of $n^{p/2} T_{\beta,n}^p$ (blue) and limiting distribution in \autoref{theorem: Convergence under the Null} (red) for $n = 10,25,50,100$ using \autoref{example: rate of covnergence} with different hyperparameters.}\label{figure: Simulation Nullhypothese 3}
\end{figure}
\newpage
\begin{figure}[h!]
    \centering
    \begin{subfigure}{\textwidth}
        \centering
        \includegraphics[width=1\textwidth]{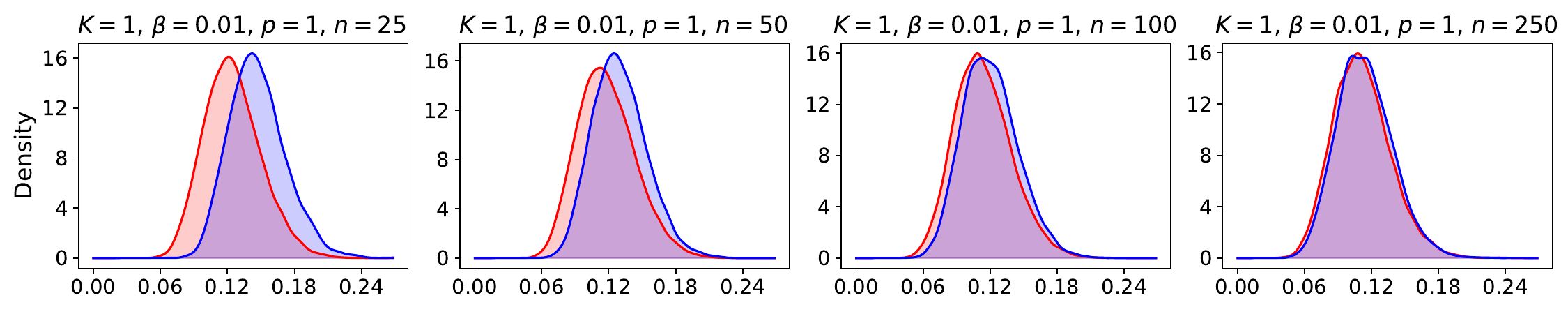}
    \end{subfigure}

    \begin{subfigure}{\textwidth}
        \centering
        \includegraphics[width=1\textwidth]{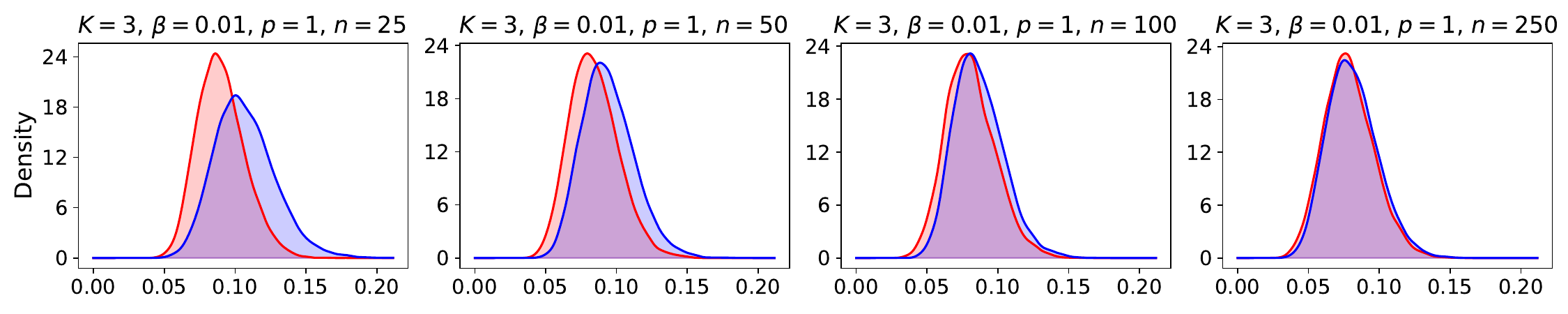}
    \end{subfigure}
    \begin{subfigure}{\textwidth}
        \centering
        \includegraphics[width=1\textwidth]{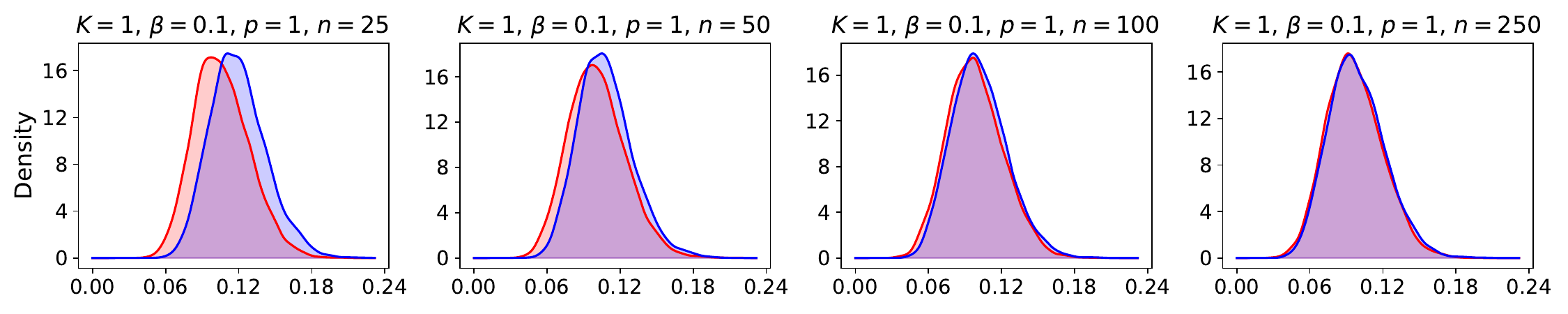}
    \end{subfigure}
    \begin{subfigure}{\textwidth}
        \centering
        \includegraphics[width=1\textwidth]{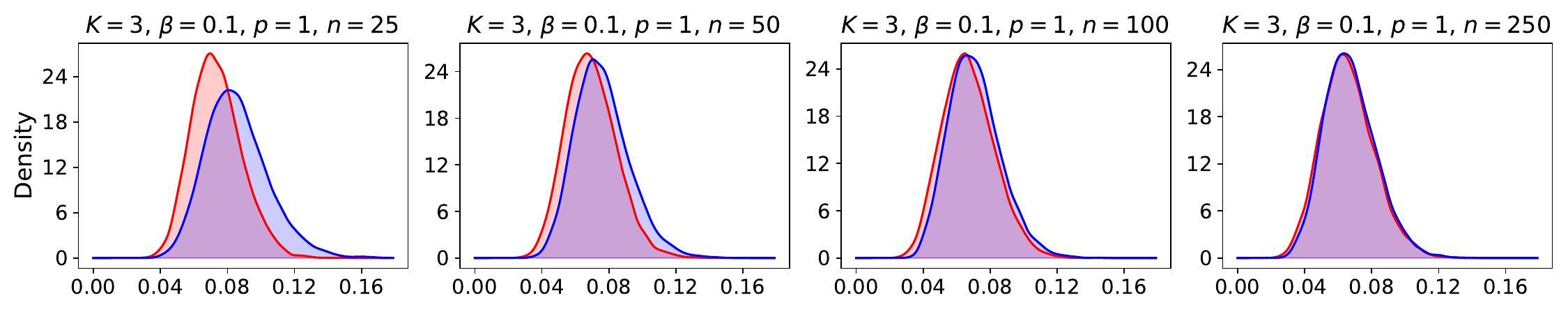}
    \end{subfigure}
    \begin{subfigure}{\textwidth}
        \centering
        \includegraphics[width=1\textwidth]{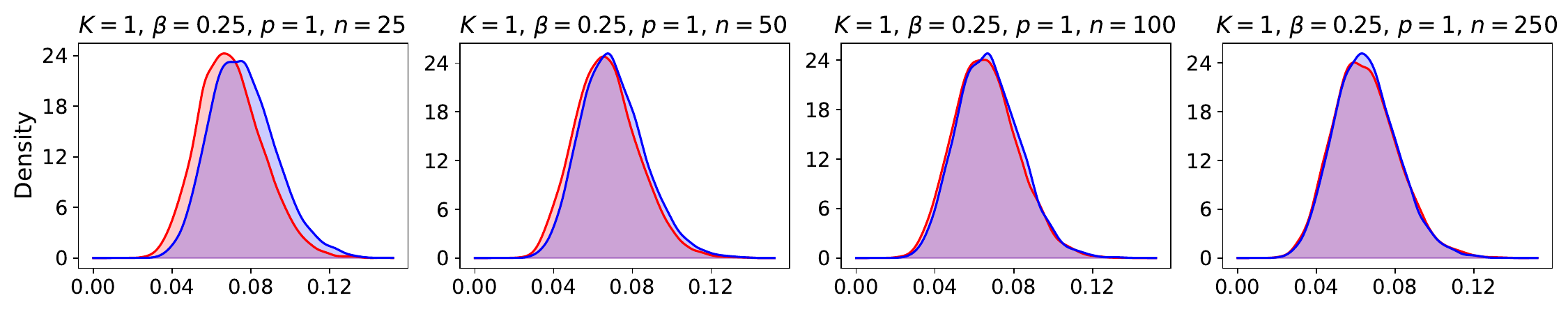}
    \end{subfigure}
    \begin{subfigure}{\textwidth}
        \centering
        \includegraphics[width=1\textwidth]{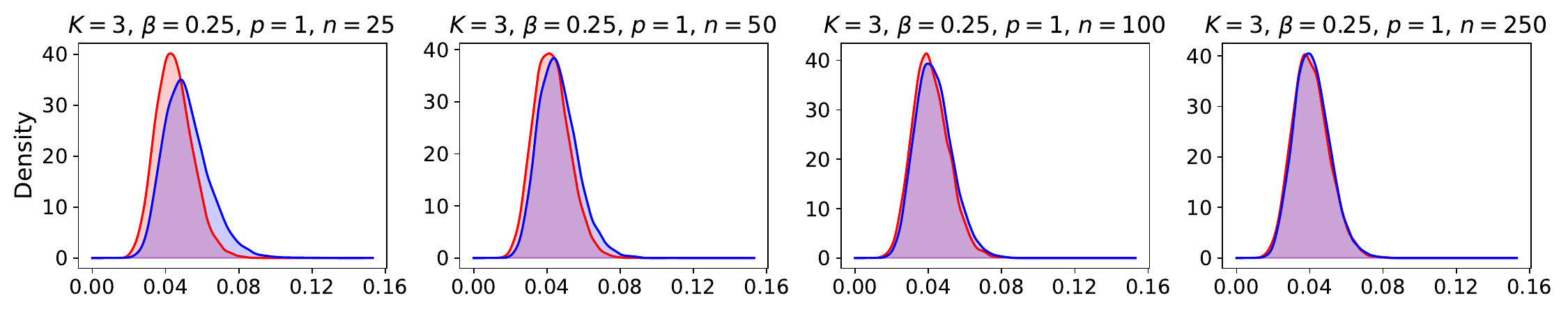}
    \end{subfigure}

    \caption{Kernel density estimates of $n^{p/2} T_{\beta,n}^{p, \star}$ (blue) and $n^{p/2}T_{\beta,n}^p$ (red) for $n = 25,50,100,250$ using \autoref{example: bs convergence example} based on $10,\!000$ realizations of $n^{p/2} T_{\beta,n}^p$ and one bootstrap sample from each realization, using different hyperparameters.}\label{figure: Simulation bootstrap 1}
\end{figure}
\newpage
\begin{figure}[h!]
    \centering
    \begin{subfigure}{\textwidth}
        \centering
        \includegraphics[width=1\textwidth]{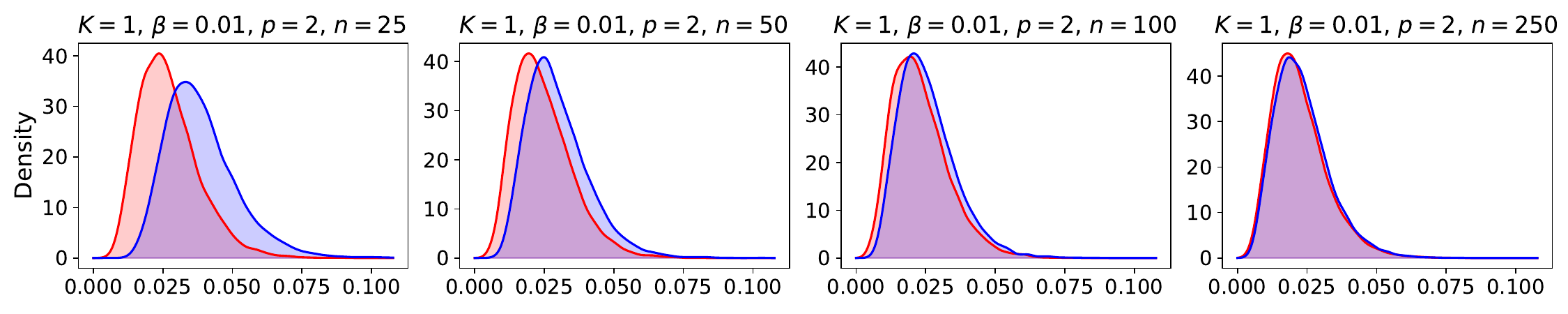}
    \end{subfigure}

    \begin{subfigure}{\textwidth}
        \centering
        \includegraphics[width=1\textwidth]{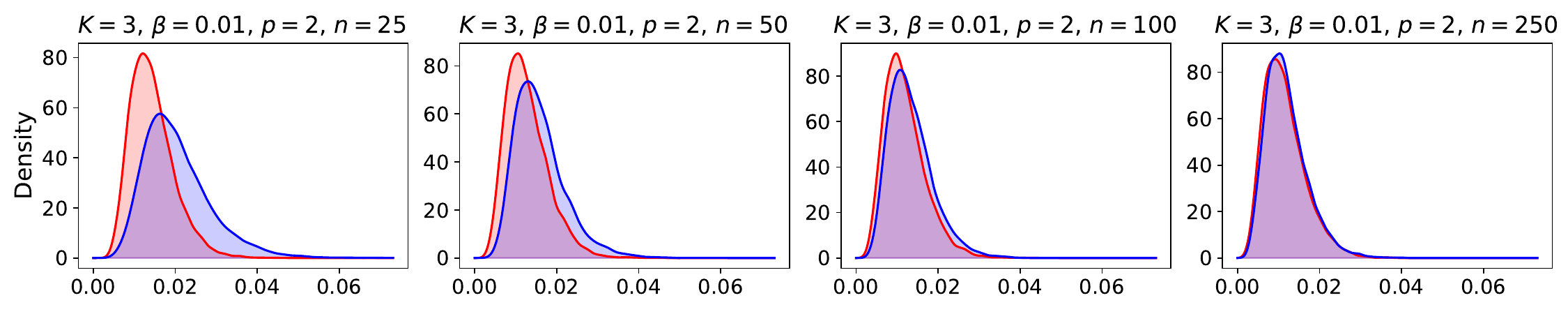}
    \end{subfigure}
    \begin{subfigure}{\textwidth}
        \centering
        \includegraphics[width=1\textwidth]{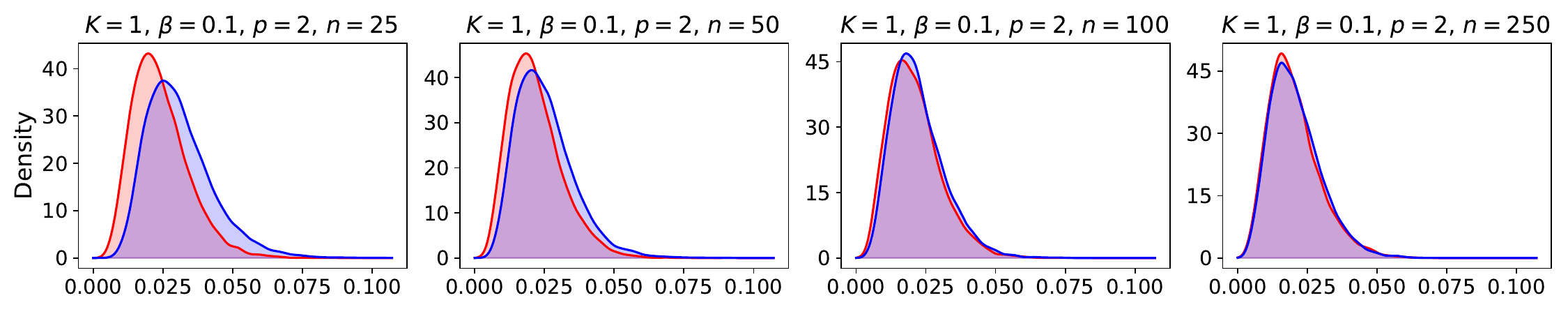}
    \end{subfigure}
    \begin{subfigure}{\textwidth}
        \centering
        \includegraphics[width=1\textwidth]{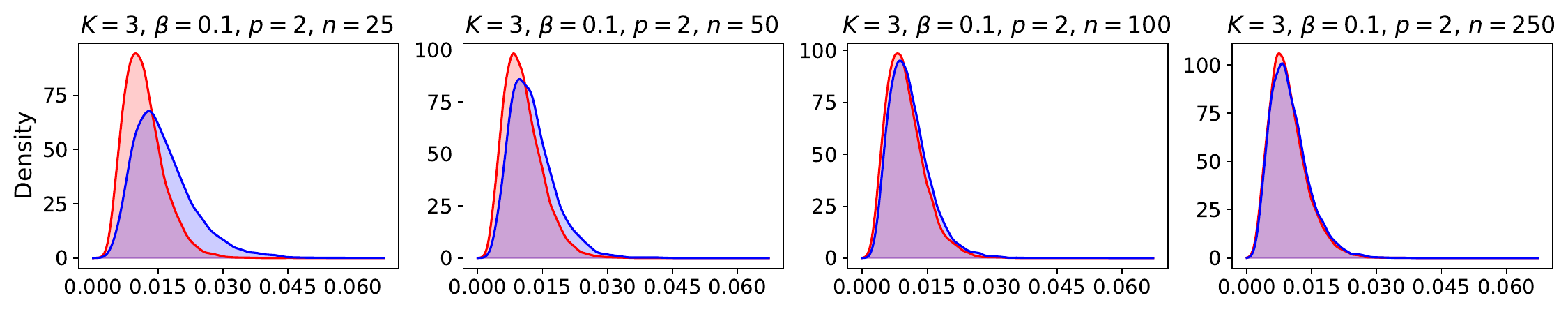}
    \end{subfigure}
    \begin{subfigure}{\textwidth}
        \centering
        \includegraphics[width=1\textwidth]{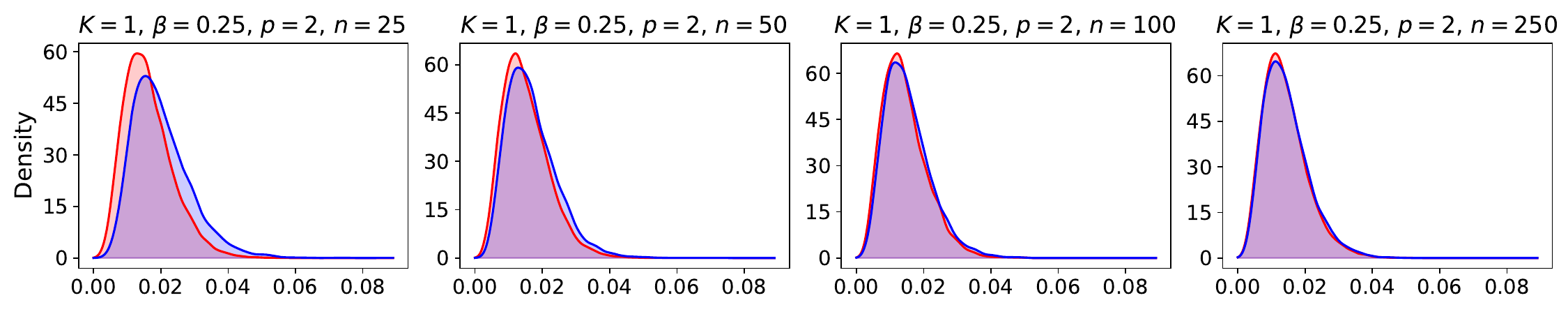}
    \end{subfigure}
    \begin{subfigure}{\textwidth}
        \centering
        \includegraphics[width=1\textwidth]{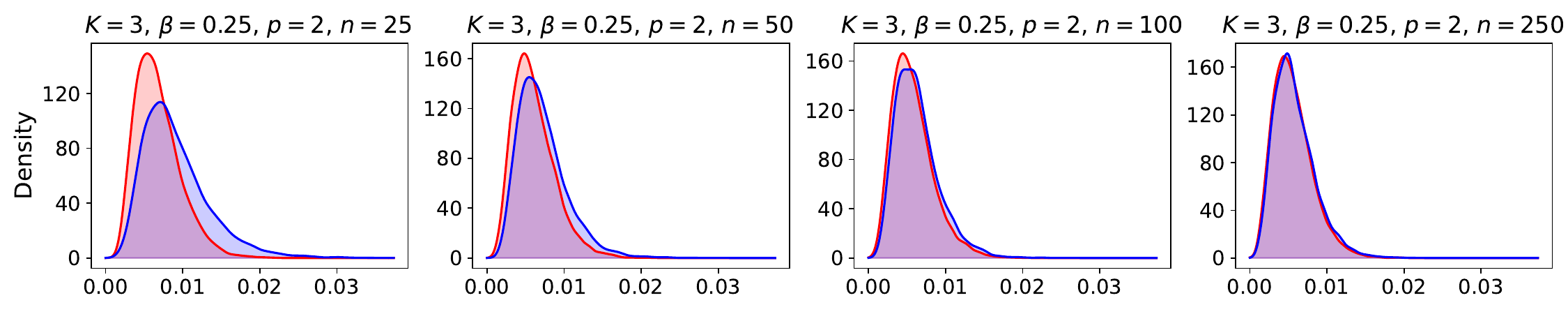}
    \end{subfigure}
    
    \caption{Kernel density estimates of $n^{p/2} T_{\beta,n}^{p, \star}$ (blue) and $n^{p/2}T_{\beta,n}^p$ (red) for $n = 25,50,100,250$ using \autoref{example: bs convergence example} based on $10,\!000$ realizations of $n^{p/2} T_{\beta,n}^p$ and one bootstrap sample from each realization, using different hyperparameters.}\label{figure: Simulation bootstrap 2}
\end{figure}
\newpage
\begin{figure}[h!]
    \centering
    \begin{subfigure}{\textwidth}
        \centering
        \includegraphics[width=1\textwidth]{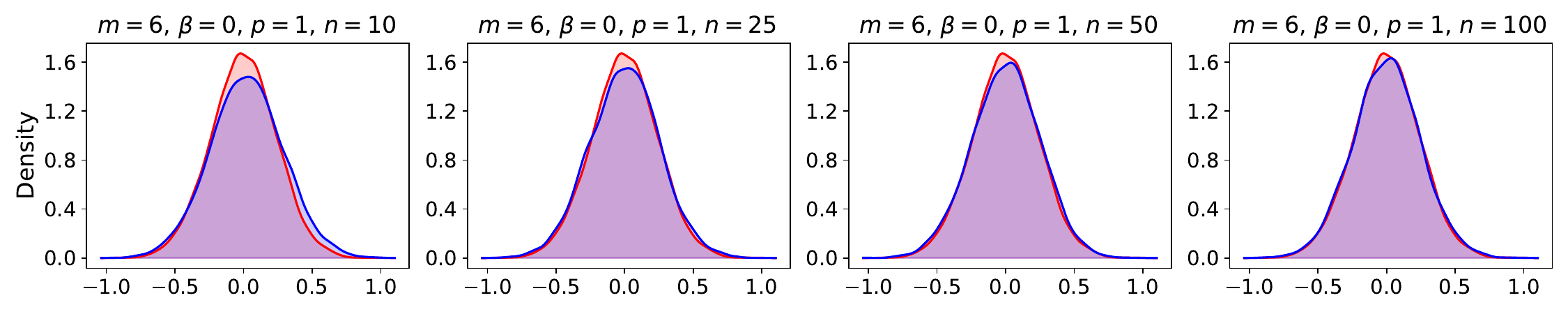}
    \end{subfigure}

    \begin{subfigure}{\textwidth}
        \centering
        \includegraphics[width=1\textwidth]{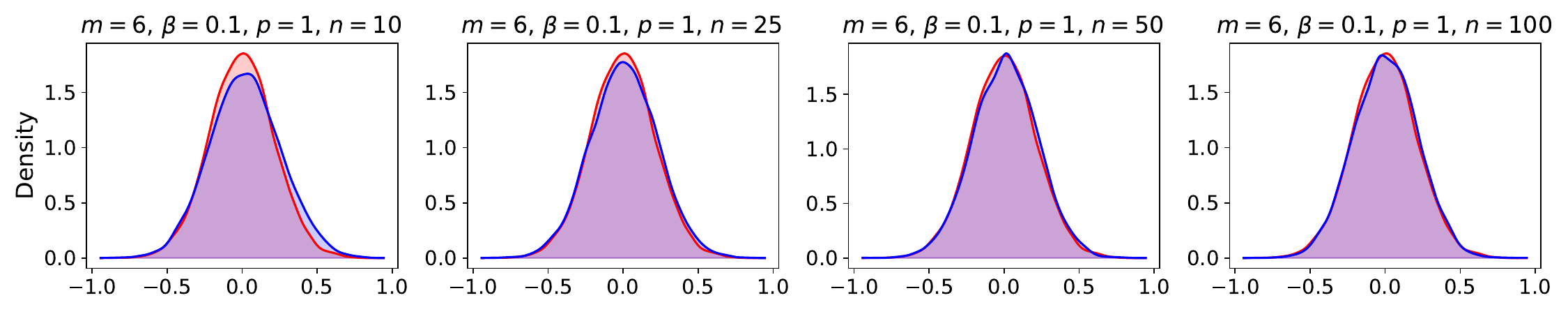}
    \end{subfigure}
    \begin{subfigure}{\textwidth}
        \centering
        \includegraphics[width=1\textwidth]{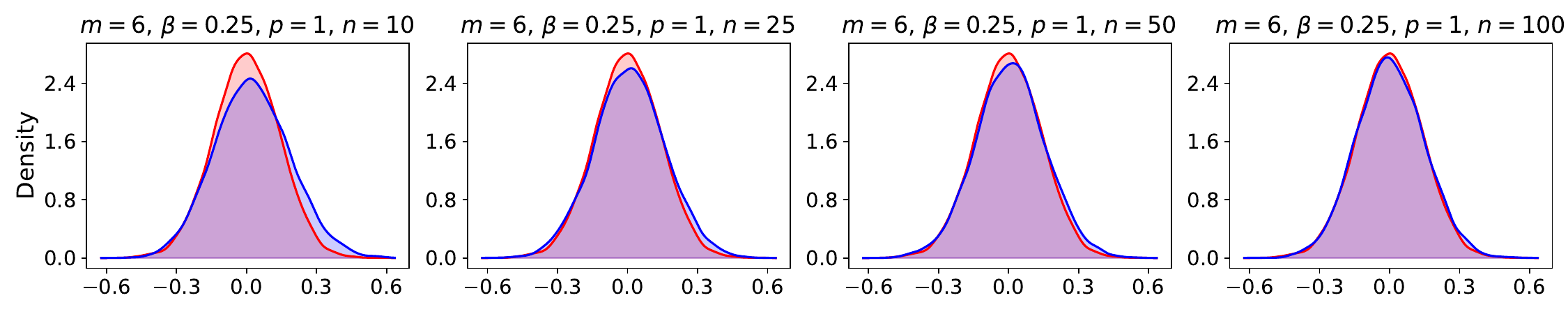}
    \end{subfigure}
    \begin{subfigure}{\textwidth}
        \centering
        \includegraphics[width=1\textwidth]{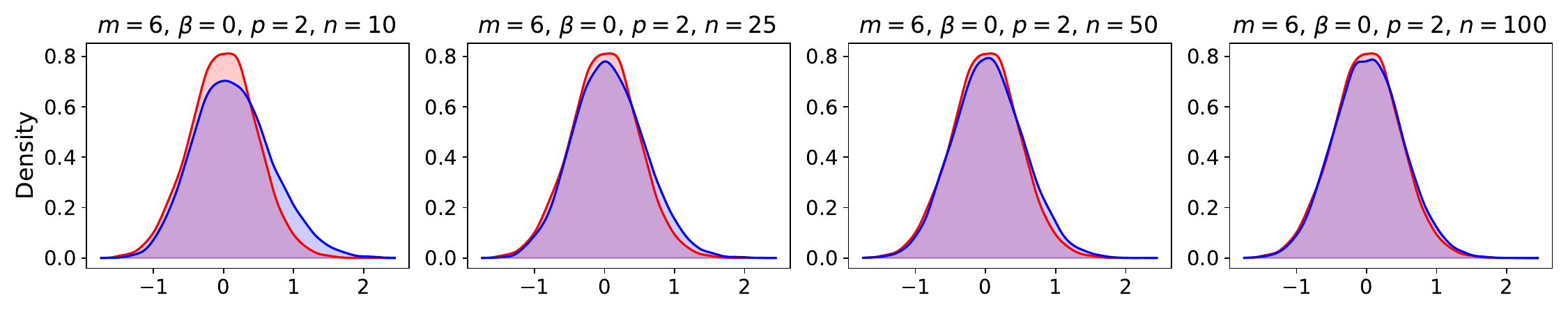}
    \end{subfigure}
    \begin{subfigure}{\textwidth}
        \centering
        \includegraphics[width=1\textwidth]{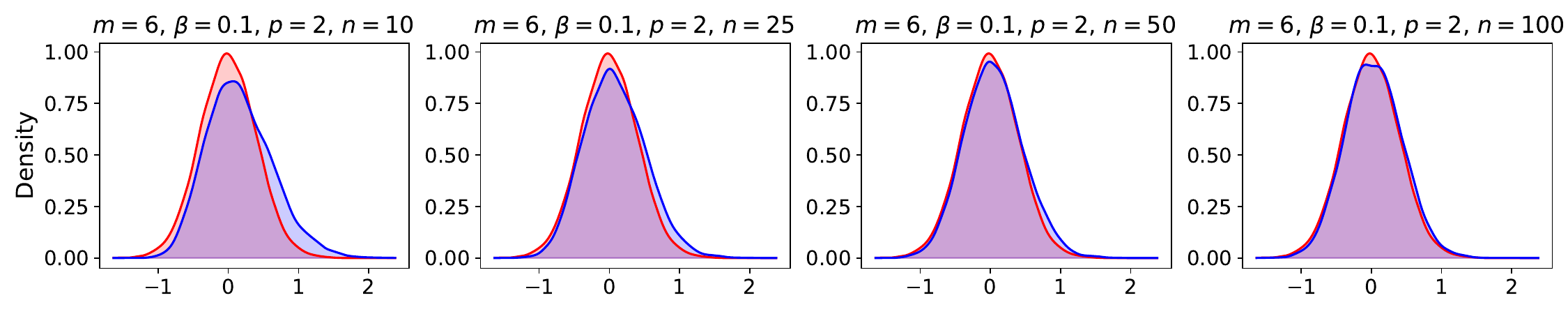}
    \end{subfigure}
    \begin{subfigure}{\textwidth}
        \centering
        \includegraphics[width=1\textwidth]{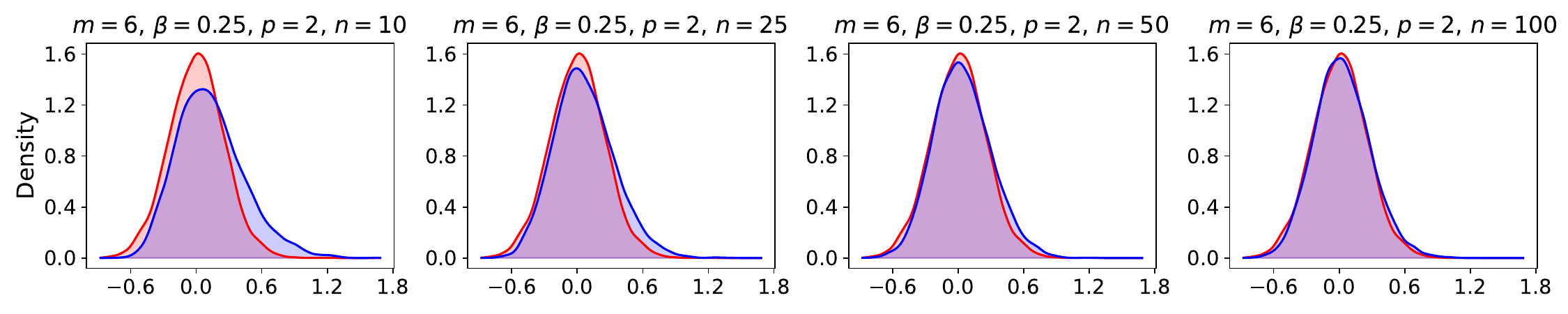}
    \end{subfigure}
    \caption{Kernel density estimates of $10,\!000$ realizations of limit distribution in \autoref{theorem: Convergence under the alt} (red) and $\sqrt{n} \left( T_{\beta,n}^p  - T_{\beta}^p\right)$ (blue) for $n = 10,25,50,100$ in \autoref{example: convergence alternative} using different hyperparameters}\label{figure: Simulation alternative 1}
\end{figure}
\newpage
\begin{figure}[h!]
    \centering
    \begin{subfigure}{\textwidth}
        \centering
        \includegraphics[width=1\textwidth]{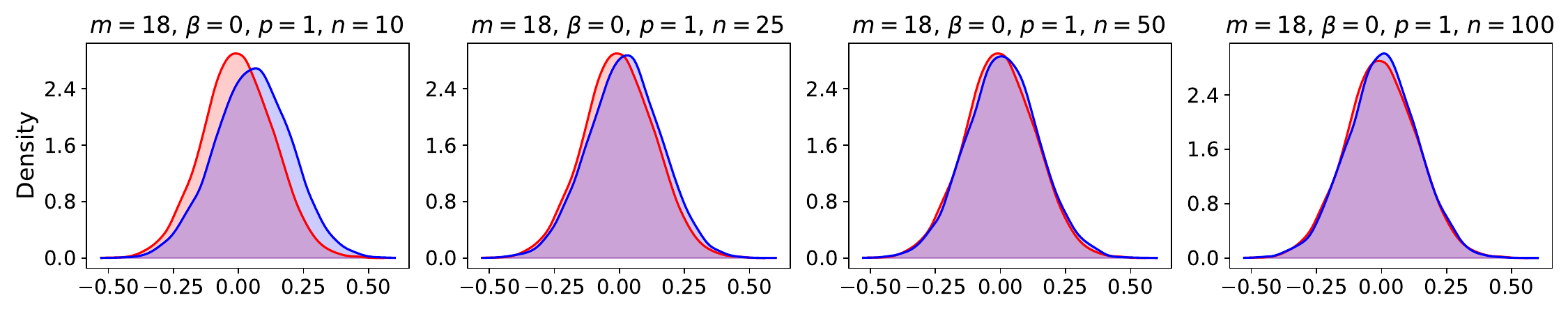}
    \end{subfigure}

    \begin{subfigure}{\textwidth}
        \centering
        \includegraphics[width=1\textwidth]{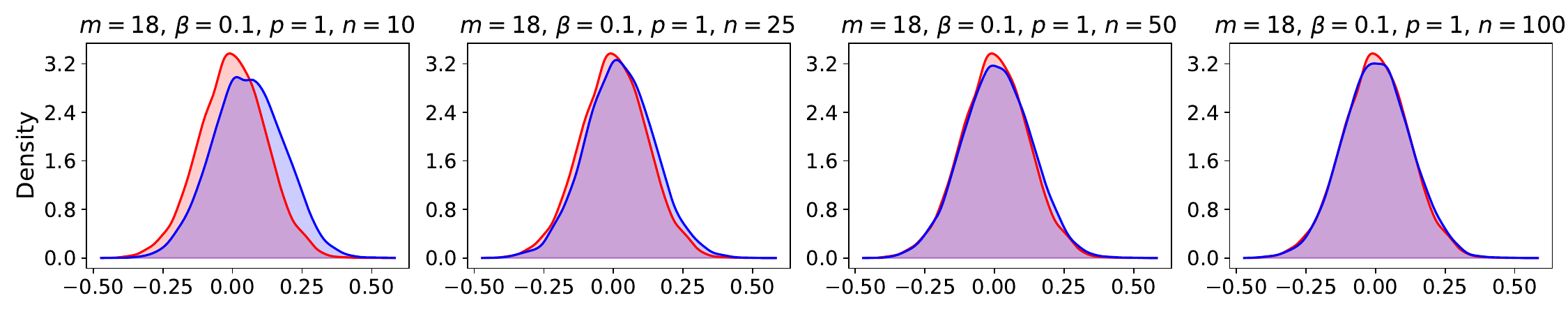}
    \end{subfigure}
    \begin{subfigure}{\textwidth}
        \centering
        \includegraphics[width=1\textwidth]{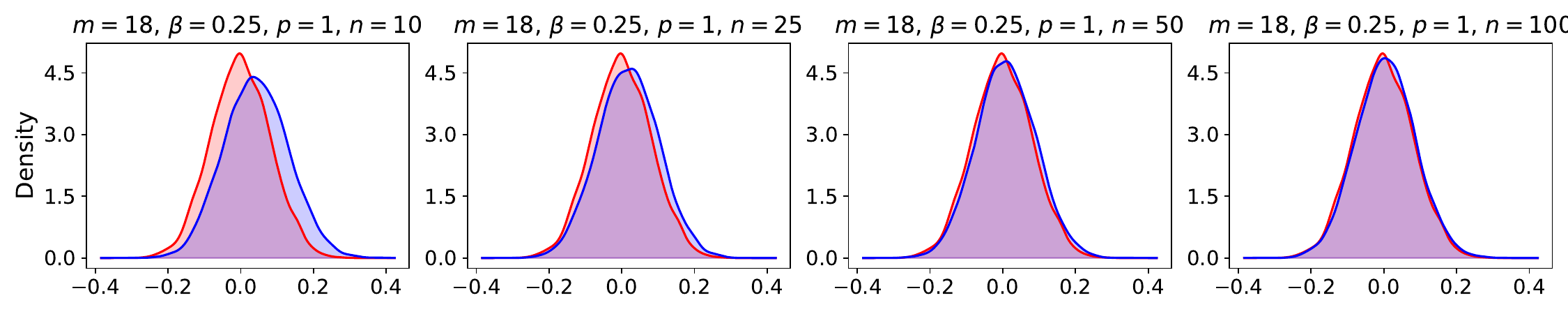}
    \end{subfigure}
    \begin{subfigure}{\textwidth}
        \centering
        \includegraphics[width=1\textwidth]{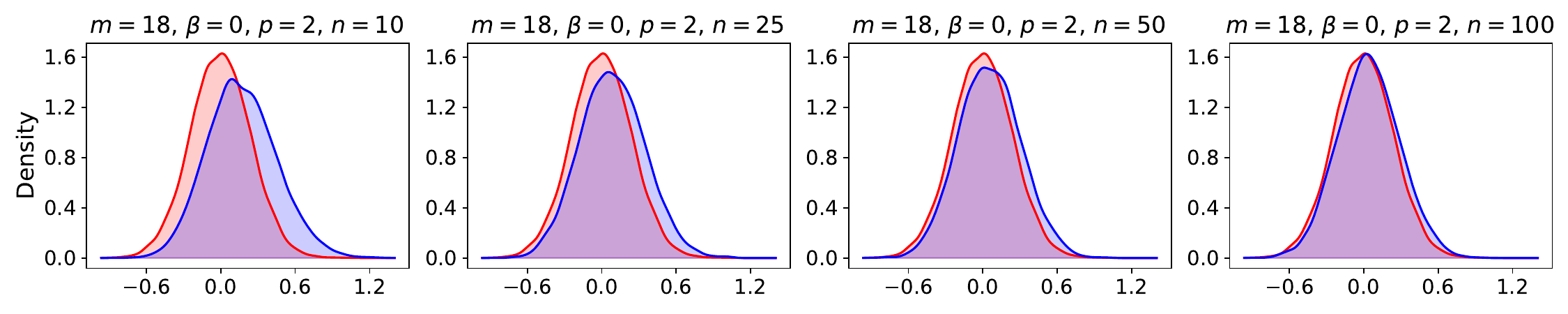}
    \end{subfigure}
    \begin{subfigure}{\textwidth}
        \centering
        \includegraphics[width=1\textwidth]{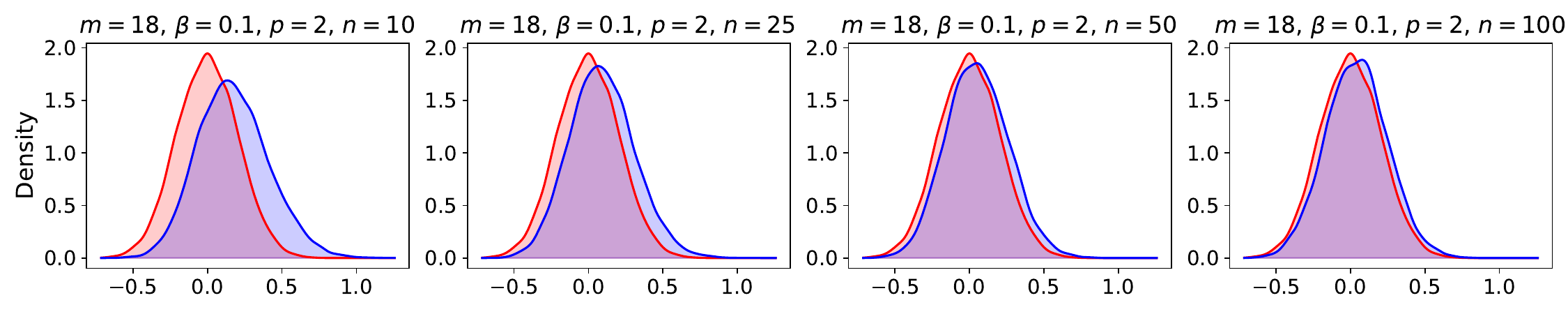}
    \end{subfigure}
    \begin{subfigure}{\textwidth}
        \centering
        \includegraphics[width=1\textwidth]{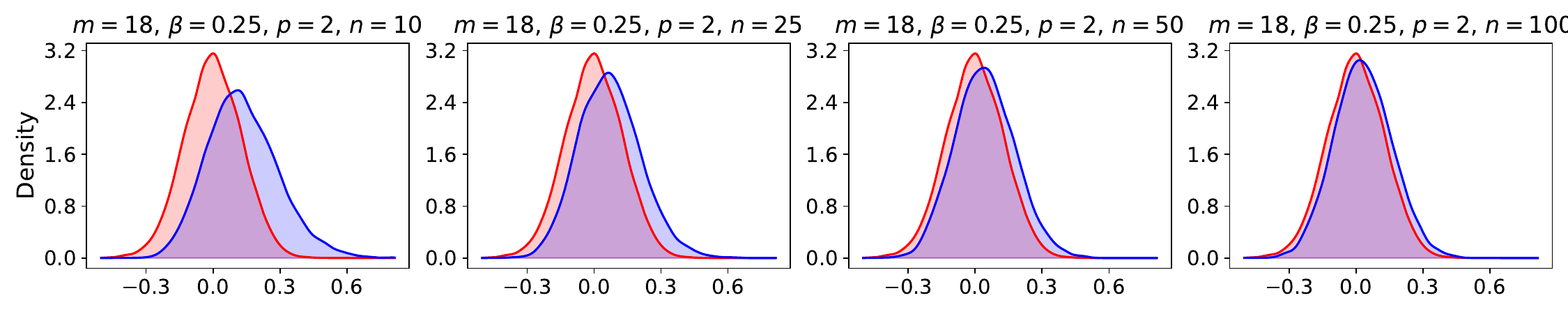}
    \end{subfigure}

    \caption{Kernel density estimates of $10,\!000$ realizations of limit distribution in \autoref{theorem: Convergence under the alt} (red) and $\sqrt{n} \left( T_{\beta,n}^p  - T_{\beta}^p\right)$ (blue) for $n = 10,25,50,100$ in \autoref{example: convergence alternative} using different hyperparameters}\label{figure: Simulation alternative 2}
\end{figure}
\clearpage
\newpage
\subsection{SLB Barycenter Test}\label{subsec:further-sim:test}
In this section, we present and discuss further results of the SLB barycenter test applied on the triangulated mesh dataset (recall Scenarios~\hyperref[item: scenario A]{A--C} from \autoref{sec: Simulations}). In said scenarios, the results using different configurations of hyperparameters are gathered in \autoref{table: power scenario A}, \autoref{table: power scenario B} and \autoref{table: power scenario C}, respectively. While results remain comparable between the different configurations, it is visible that a value of $\beta = 0.25$ seems to lead to worse results than $\beta =0.01,0.1$ (see \autoref{table: power scenario C}). This may be due to the fact that a larger trimming value leads to a higher loss of information about the underlying distance distributions. Further, in~\ref{item: scenario B}, there is a clear difference in the performance between $p=1$ and $p=2$, with $p=2$ reaching a higher empirical power. A potential reason for this is that the median (which is $\Lambda_1$) is less sensitive to outliers than the mean, which arises as the barycenter functional for $p=2$.
\begin{table}[h] 
\centering
\begin{tabular}{c c c c c c c}
        \toprule
         $p$ & $\beta$ & n = 10& n = 25 & n = 50 & n = 100 & n = 250\\ \midrule
         & 0.01 & 0.034 & 0.043 & 0.039 & 0.052 & 0.05 \\ 
        1 & 0.1 & 0.016 & 0.036 & 0.05 & 0.065 & 0.049 \\ 
         & 0.25 & 0.02 & 0.051 & 0.045 & 0.05 & 0.053 \\ \midrule
         & 0.01 & 0.009 & 0.044 & 0.054 & 0.054 & 0.057 \\ 
        2 & 0.1 & 0.005 & 0.035 & 0.049 & 0.045 & 0.044 \\ 
         & 0.25 & 0.002 & 0.036 & 0.043 & 0.042 & 0.048 \\ \bottomrule
\end{tabular}
\caption{Empirical level of $\Phi_{\beta,p}^\ast$ over $1,\!000$ repetitions of the test with $\alpha = 0.05$ in~\ref{item: scenario A}.}\label{table: power scenario A}
\end{table}
\begin{table}[h]
\centering
\begin{tabular}{c c c c c c c}
        \toprule
         $p$ & $\beta$ & n = 10& n = 25 & n = 50 & n = 100 & n = 250\\ \midrule
         & 0.01 & 0.063 & 0.13 & 0.265 & 0.501 & 0.906 \\ 
        1 & 0.1& 0.026 & 0.137 & 0.282 & 0.504 & 0.937 \\ 
         & 0.25 & 0.047 & 0.173 & 0.281 & 0.512 & 0.922 \\ \midrule
         & 0.01 & 0.033 & 0.154 & 0.354 & 0.712 & 0.997  \\ 
        2 & 0.1 & 0.026 & 0.166 & 0.365 & 0.747 & 0.998 \\ 
         & 0.25 & 0.01 & 0.158 & 0.409 & 0.773 & 0.999  \\ \bottomrule
\end{tabular}
\caption{Empirical power of $\Phi_{\beta,p}^\ast$ over $1,\!000$ repetitions of the test with $\alpha = 0.05$ in~\ref{item: scenario B}.}\label{table: power scenario B}
\end{table}
\begin{table}[h]
\centering
\begin{tabular}{c c  c  c c c c}
        \toprule
         $p$ & $\beta$ & n = 10& n = 25 & n = 50 & n = 100 & n = 250\\ \midrule
         & 0.01 & 0.029 & 0.159 & 0.559  & 0.936 & 1   \\ 
        1 & 0.1& 0.013 & 0.139 & 0.429 & 0.825 & 1   \\ 
         & 0.25 & 0.01 & 0.105 & 0.326 & 0.685 & 0.996 \\ \midrule
         & 0.01 & 0.029 & 0.270 & 0.694 & 0.984 & 1  \\ 
        2 & 0.1 & 0.019 & 0.141 & 0.490 & 0.896 & 1 \\ 
         & 0.25 & 0.002 & 0.1 & 0.33 & 0.739 & 0.995 \\ \bottomrule
\end{tabular}
\caption{Empirical power of $\Phi_{\beta,p}^\ast$ over $1,\!000$ repetitions of the test with $\alpha = 0.05$ in~\ref{item: scenario C}.}\label{table: power scenario C}
\end{table}
\end{document}